\documentclass[11pt]{amsart}
\usepackage[hmargin = 0.75 in, vmargin = 0.75 in]{geometry}
\usepackage{amssymb, amsmath,amsthm}
\newtheorem{thm}{Theorem}[section]
\newtheorem{prop}[thm]{Proposition}

\newtheorem{cor}[thm]{Corollary}
\newtheorem{lem}[thm]{Lemma}
\theoremstyle{definition}
\newtheorem{rem}[thm]{Remark}
\newtheorem{def1}[thm]{Definition}

\newcommand{\ra}{\rightarrow}
\newcommand{\bk}{\backslash}
\newcommand{\mc}{\mathcal}
\newcommand{\mb}{\mathbb}
\newcommand{\sg}{\sigma}

\newcommand{\R}{\mb{R}}

\newcommand{\N}{\mb{N}}

\newcommand{\llf}{\left\lfloor}
\newcommand{\e}{\epsilon}
\newcommand{\rrf}{\right\rfloor}
\newcommand{\mbf}{\boldsymbol}
\begin{document}
\title[Hal\'{a}sz Refinement]{A Strengthening of Theorems of Hal\'{a}sz and Wirsing}
\normalsize
\author[A. P. Mangerel]{Alexander P. Mangerel}
\address{Department of Mathematics\\ University of Toronto\\
Toronto, Ontario, Canada}
\email{sacha.mangerel@mail.utoronto.ca}
\maketitle
\begin{abstract}
Given an arithmetic function $g(n)$ write $M_g(x) := \sum_{n \leq x} g(n)$. We extend and strengthen the results of a fundamental paper of Hal\'{a}sz in several ways by proving upper bounds for the ratio of $\frac{|M_g(x)|}{M_{|g|}(x)}$, for any strongly multiplicative, complex-valued function $g(n)$ under certain assumptions on the sequence $\{g(p)\}_p$. We further prove an asymptotic formula for this ratio in the case that $|\text{arg}(g(p))|$ is sufficiently small uniformly in $p$. In so doing, we recover a new proof of an explicit lower mean value estimate for $M_{f}(x)$ for any non-negative, multiplicative function satisfying $c_1 \leq |f(p)| \leq c_2$ for $c_2 \geq c_1 > 0$, by relating it to $\frac{x}{\log x}\prod_{p \leq x} \left(1+\frac{f(p)}{p}\right)$. 
As an application, we generalize our main theorem in such a way as to give explicit estimates for the ratio $\frac{|M_g(x)|}{M_{f}(x)}$, whenever $f: \mb{N} \ra (0,\infty)$ and $g: \mb{N} \ra \mb{C}$ are strongly multiplicative functions that are uniformly bounded on primes and satisfy $|g(n)| \leq f(n)$ for every $n \in \mb{N}$. This generalizes a theorem of Wirsing and extends recent work due to Elliott.
\end{abstract}
\section{Introduction}
Let $g : \mb{N} \ra \mb{C}$ be a multiplicative function, i.e., such that for any coprime $m,n \in \mb{N}$, $g(mn) = g(m)g(n)$. Put 
\begin{equation*}
M_g(x) := \sum_{n \leq x} g(n). 
\end{equation*}
Many difficult problems in number theory can be recast into a problem about the rate of growth of $M_g(x)$, for $g$ multiplicative. For example, when $g = \mu$, the M\"{o}bius function, Landau proved that the Prime Number Theorem is equivalent to $M_{\mu}(x) = o(x)$ (see Theorem 8 in I.3 of \cite{Ten2}), and Littlewood showed that the Riemann hypothesis is equivalent to $M_{\mu}(x) = O_{\e}\left(x^{\frac{1}{2}+\e}\right)$ (see paragraph 14.25 in \cite{Tit}). In another direction, let $q \in \mb{N}$, $q \geq 2$ and suppose $\chi$ is a Dirichlet character modulo $q$, i.e., a periodic extension of a group homomorphism $\chi: \left(\mb{Z}/q\mb{Z}\right)^{\ast} \ra \mb{T}$, with $\chi(n) = 0$ for $(n,q) > 1$.  Finding sharp upper and lower bounds for $M_{\chi}(x)$ that are uniform in the conductor $q$ is a notoriously difficult problem whose study goes back at least to P\'{o}lya.  There are various important consequences of bounds for $M_{\chi}(x)$. For instance, if $\chi$ is a quadratic character then sharp upper bounds have implications in the study of the Class Number problem for real quadratic fields, whose study was initiated by Gau$\ss$ (see, for instance, \cite{Lam} and \cite{LamD}). \\
In this vein, it is of interest to relate $M_g(x)$ to other arithmetic data that are known, and the most basic such approach is to consider its growth relative to $x$. We say that $g$ possesses a \emph{mean value} if there exists a constant $\alpha \in \mb{C}$ such that $x^{-1}M_g(x) \ra \alpha$ as $x \ra \infty$. It is a classical topic in number theory to consider when a multiplicative function possesses a mean value. In the case where $g$ is non-negative this is simplest, and it is known that any such function taking values in $[0,1]$ possesses a mean value. For real-valued functions, Wintner \cite{Win} proved that if $g$ is multiplicative with $|g(n)| \leq 1$ for all $n \in \mb{N}$ then it is sufficient that $\sum_p \frac{1-g(p)}{p}$ converges in order for $g$ to have a non-zero mean value, and in particular, 
\begin{equation*}
\lim_{x \ra \infty} x^{-1}M_g(x) = \prod_p \left(1-\frac{1}{p}\right)\sum_{k \geq 0} g(p^k)p^{-k}.
\end{equation*}
Wirsing, improving on this result and others due to Delange, proved a famous conjecture of Erd\H{o}s that as long as $|g(n)| \leq 1$, $g$ possesses a mean value, and that this mean value is zero if, and only if, $\sum_p \frac{1-g(p)}{p}$ diverges \cite{Ell}. \\
It is known that this latter criterion is insufficient to describe the situation in the case that $g$ is complex-valued, e.g., when $g(n) = n^{i\alpha}$ (a so-called \emph{archimedean character}). This example also demonstrates that the limit of $x^{-1}M_g(x)$ may not even exist, as a simple contour integration argument shows that $M_{n^{i\alpha}}(x) \sim x^{1+i\alpha}/(1+i\alpha)$.  It turns out that the sole counterexamples to this criterion for complex-valued multiplicative functions are those $g$ that behave like $p^{i\alpha}$ times a real-valued multiplicative function at primes for some $\alpha \in \mb{R}$, and Hal\'{a}sz \cite{Hal1} generalized Wirsing's theorem, to complex-valued functions with $|g(n)| \leq 1$, accounting for this behaviour. In particular, he showed that $g$ has non-zero mean value if, and only if, $\sum_p \frac{1-\text{Re}(g(p)p^{-i\alpha})}{p}$ converges for some $\alpha$. It is conventional in the literature to say that if $g$ has a non-zero mean value and an $\alpha$ exists for which the latter series converges then $g$ ''pretends'' to be the product of a real function and an archimedean character.  \\
While the above results classify those situations in which a function $g$ has mean-value zero, it does not provide a rate of convergence of $x^{-1}M_g(x)$ to its limiting value, an important consideration in many number-theoretic problems.  Shortly after proving his mean value theorem, Hal\'{a}sz (\cite{Hal2}, \cite{Hal3}) proved in two separate papers \emph{explicit mean-value estimates} for multiplicative functions.  
%
Of particular interest to us is \cite{Hal2}, in which an estimate, uniform over a collection of complex-valued, multiplicative functions $g$, is given in the case that the values of $g$ at primes are not uniformly distributed in argument about 0. \\
(Due to the multiplicativity of $g$, its values at primes are significant in determining the nature of mean-value estimates; these can therefore be improved if certain assumptions are made about the uniformity of distribution of values of the sequence $\{g(p)\}_p$. For further examples, see, for instance, \cite{Hall}). \\
To be precise, suppose $g$ is a \emph{completely multiplicative} function, i.e., $g(mn) = g(m)g(n)$ for all $m,n \in \mb{N}$, with the following properties. First, assume there exists $\delta > 0$ with $\delta \leq |g(p)| \leq 2-\delta$ for all primes $p$. Second, suppose that there exists an angle $\theta \in [-\pi,\pi)$ and $\beta > 0$ for which $|\text{arg}(g(p))-\theta| \geq \beta$, where $\text{arg}(z)$ is taken to be the branch of argument with cut line along the negative real axis.  Then Hal\'{a}sz proved the existence of a constant $c = c(\delta,\beta) > 0$ such that
\begin{equation}
x^{-1}M_g(x) \ll_{\delta,\theta,\beta} \exp\left(\sum_{p \leq x} \frac{|g(p)|-1}{p} - c\sum_{p \leq x} \frac{|g(p)|-\text{Re}(g(p))}{p}\right) \label{HALLLL}.
\end{equation}
In the specific situation in which $|g(p)-1|$ is uniformly small over all primes $p$ in a precise sense then
\begin{equation}
M_g(x) \sim x\exp\left(\sum_{p \leq x} \frac{g(p)-1}{p}\right) =: Ax \label{HALLL}
\end{equation}
This last result
essentially implies that completely multiplicative functions $g$ whose Dirichlet series are sufficiently close to the $\zeta$ function for $s = \sg + i\tau$ with $\sg > 1$ and $|\tau|$ small will possess a mean value. It should be emphasized that the first result just mentioned is presented in \cite{Ell2} with $c$ a constant multiple of $\delta \beta^3$. In particular, the upper bound is less advantageous in cases where $\delta$ is small.\\
Thus far, all of the results that we have mentioned from the literature compare the summatory function $M_g(x)$ with $x$ (or rather, essentially  $M_1(x)$), and in \eqref{HALLL}, the Dirichlet series $G(s) := \sum_{n \geq 1} g(n)n^{-s}$ is compared with $\zeta(s)$.  A better standard of comparison for $G(s)$ than $\zeta(s)$, demonstrating similar nuances in the distribution of its coefficients and highlighting, instead, the effect of cancellation due to the arguments $\text{arg}(g(p))$ on the sum $M_g(x)$, is $\mc{G}(s) := \sum_{n \geq 1} |g(n)|n^{-s}$. 
With the broadened perspective of comparing the summatory functions of two multiplicative functions, neither of which the constant function 1, the more general problem of comparing $M_g(x)$ to $M_{f}(x)$ where $f$ is \emph{any} non-negative multiplicative function satisfying $|g(n)| \leq f(n)$ for all $n \in \mb{N}$, is also natural. This general viewpoint is elaborated upon in the next section as well (see the discussion surrounding Theorem \ref{WIRSINGEXT}). Besides a result of this nature extending a theorem of Wirsing, as discussed below, ratios of the form $M_{g}(x)/M_{f}(x)$ are useful in arithmetic applications. Indeed, non-negative multiplicative functions and their summatory functions are typically simpler to estimate, as properties such as positivity and monotonicity can be used to extract upper and lower bounds. One would therefore expect that evaluating $M_g$ indirectly through $M_{f}$, where $M_f$ is well-understood, may be more efficient than a direct attempt at evaluating $M_g$. \\
In our main theorems below, we consider estimates of the above kind in case $g$ is either completely multiplicative or \emph{strongly} multiplicative, i.e., $g(p^k) = g(p)$ for all $k \in \mb{N}$ and primes $p$. A consequence of Theorem \ref{WIRSINGEXT} is the following. Suppose $f,g$ are strongly or completely multiplicative functions such that: i) $f$ satisfies $\delta \leq |f(p)| \leq B$ for all primes $p$ with $B,\delta > 0$, and does not ''pretend'' to be a product of a real function and an archimedean character; ii) $g$ is real-valued such that $|\text{arg}(f(p)) + \beta g(p) -\tau \log p|$ is not too small (in a precise sense) for $p \leq x$ and any $|\tau| \leq \log^{O(1)} x$. Then for any $\alpha,\beta \in \mb{R}$ fixed,
\begin{equation*}
\sum_{n \leq x} f(n)^{\alpha}g(n)^{i\beta} \ll_B \left(\frac{B}{\delta}\right)^3\left(\sum_{n \leq x} |f(n)|^{\alpha}\right) \exp\left(-c\sum_{p \leq x} \frac{|f(p)|^{\alpha}-\text{Re}\left(f(p)^{\alpha}e^{i\beta\log g(p)}\right)}{p}\right),
\end{equation*}
where $c$ depends at most on the ratio $B/\delta$ (see Theorem \ref{HalGen} below). In particular, if $f = g$ then our theorem gives an estimate for the summatory functions of complex moments of real-valued multiplicative functions. In certain cases, depending on the distribution of the values of $f(p)$, we can even find asymptotic formulae for such quantities. Estimates for moments of this type are relevant in the context of the moment method in the analysis of number-theoretic, probabilistic models (for a related example, see \cite{LamD}). For another example, see \cite{Res2}.
\section{Results}
\subsection{Definitions and Conventions}
\begin{def1}\label{REAS}
Given $g$ a multiplicative function and $\mbf{E} := (E_1,\ldots,E_m)$ a partition of the primes, the pair $(g,\mbf{E})$ is said to be \emph{reasonable} if for any $t \geq 2$ sufficiently large and $\kappa > 0$ fixed there exists an index $1 \leq j_0 \leq m$ such that 
\begin{align}
\sum_{t^{\kappa} < p \leq t \atop p \in E_{j_0}} \frac{1}{p} &\geq \frac{1}{m} \sum_{t^{\kappa} < p \leq t} \frac{1}{p} \label{CONDIT1}\\
\sum_{p \leq t \atop p \in E_{j_0}} \frac{|g(p)|}{p} &\geq \frac{1}{m}\sum_{p \leq t} \frac{|g(p)|}{p}.\label{CONDIT2}
\end{align}
We also say that $(g,\mbf{E})$ is \emph{non-decreasing} if for each $1 \leq j \leq m$, $\{|g(p)|\}_{p \in E_j}$ is a non-decreasing sequence. 
\end{def1}
It is an obvious consequence of the pigeonhole principle that for any fixed $t$ and $\kappa$ one can find two (possibly distinct) indices $1 \leq j_1,j_2 \leq m$ such that \eqref{CONDIT1} holds with $j_1$ in place of $j_0$ and \eqref{CONDIT2} holds with $j_2$ in place of $j_0$. However, it need not be the case that $j_1 = j_2$. Thus, the ''reasonable'' property permits us to choose a common index satisfying both conditions simultaneously. Any partition containing even \emph{one} set with Dirichlet density, i.e., such that $\log_2^{-1} x\left(\sum_{p \leq x} \frac{1}{p}\right) \ra \lambda$ with $\lambda > 0$, and such that $|g(p)|$ is not too small on this set, will provide a reasonable pair $(g,\mbf{E})$. There is therefore a wealth of natural cases in which $(g,\mbf{E})$ is reasonable (even when $\frac{\delta}{B}$ is small). This definition is used in Section 6 alone. \\
\begin{def1}\label{GOOD}
We say that a subset of the primes $E$ is \emph{good} if there exists $\lambda_j > 0$ and a neighbourhood of $s = 1$ such that $\sum_{p \in E_j} \frac{1}{p^s} - \lambda_j \log\left(\frac{1}{s-1}\right)$ is holomorphic.
$\mbf{E}$ is a \emph{good partition} if, for each $1 \leq j \leq m$, $E_j$ is good.
\end{def1}
As a consequence of the Prime Number Theorem and accompanying zero-free regions associated to $\zeta$, it is evident that the set of all primes give a trivial, good partition, as is the case for a partition of (unions of) arithmetic progressions modulo any $q \in \mb{N}$. For a different example, it can be checked that, for a fixed $\tau \in \mb{R}$ and $0\leq \alpha < \beta \leq 1$, the set of primes 
\begin{equation*}
E := \left\{p : \left\{\frac{\tau \log p}{2\pi}\right\} \in (\alpha,\beta]\right\} 
\end{equation*}
satisfies $\sum_{p \in E \atop p \leq x} \frac{1}{p} \sim (\beta-\alpha) \log_2 x + O(1)$ (see, for instance, Lemma 3.6 of \cite{prev}). Hence, taking $\sg := 1+\frac{1}{\log x}$ and letting $x \ra \infty$, it is not hard to show (using \eqref{HEAD} and \eqref{TAIL} below) that $\sum_{p \in E_j} \frac{1}{p^{s}} - (\beta-\alpha) \log\left(\frac{1}{s-1}\right)$ is holomorphic near $s = 1$.  Thus, partitioning $[0,2\pi]$ into intervals $(\alpha_j,\beta_j]$ and choosing sets $E_j$ in analogy to the choice of $E$ above, we can generate a good partition. \\
By a \emph{squarefree-supported} multiplicative function, we mean a multiplicative function of the form $f(n) = \mu^2(n)g(n)$, where $g$ is multiplicative and $\mu$ is the M\"{o}bius function.  \\
For $x \geq 2$, $T > 0$, a set $E$ of primes and arithmetic functions $f,g: \mb{N} \ra \mb{C}$ taking values in the unit disc we put
\begin{align*}
D_E(g,n^{i\tau};x) &:= \sum_{p \in E \atop p \leq x} \frac{1-\text{Re}(g(p)\overline{f(p)})}{p} \\
\rho_E(x;g,T) &:= \min_{|\tau| \leq T} D_E(g,n^{i\tau};x). 
\end{align*}
Unless stated otherwise, given a set of primes $E$, $E(t) := \sum_{p \in E \atop p \leq t} \frac{1}{p}$. \\
We will frequently refer to a finite collection of parameters by a corresponding vector whose entries are those parameters. Thus, for instance, if $\{r_1,\ldots,r_k\}$ are positive real numbers then we refer to this collection with the shorthand $\mbf{r}$, defined as the vector $(r_1,\ldots,r_k)$. \\ \\
Let us define the collection of strongly and completely multiplicative functions $\mc{C}$ as follows: if $g \in \mc{C}$ then there is a partition $\{E_1,\ldots,E_m\}$ of the set of primes, and a set of primes $S$ such that there exist non-decreasing and non-increasing positive functions $B_j= B_j(x)$ and $\delta_j = \delta_j(x)$, respectively, such that: \\
i) for $p \in E_j \bk S$, $\delta_j \leq |g(p)|\leq B_j$, and given $\delta := \min_{1 \leq j \leq m} \delta_j$,  $|g(p)| < \delta$ for $p \in S$ (in case $g$ is completely multiplicative, we require that $B_j(x) < 2$ uniformly in $x$);\\
ii) for $P_x := \prod_{p \in S \atop p \leq x} p$, $P_x \ll_{\alpha} x^{\alpha}$ for any $\alpha > 0$. \\
We also define subcollections $\mc{C}_a$ and $\mc{C}_b$ of $\mc{C}$ as follows: \\
$g \in \mc{C}_a$ if: iii) there are positive functions $\eta_j = \eta_j(x)$ such that $|\text{arg}(g(p))| \leq \eta_j$ for each $p \in E_j$; \\
$g \in \mc{C}_b$ if: iv) $\mbf{E}$ is a good partition (in the sense of Definition \ref{GOOD}); \\
v) there are angles $\{\phi_1,\ldots,\phi_m\}$ and $\{\beta_1,\ldots,\beta_m\} \subset (0,\pi)$ such that $|\text{arg}(g(p))-\phi_j| \geq \beta_j$ for each $p \in E_j$; \\
vi) $D_{E_j}(g,n^{i\tau};x) \geq C\log_3 x$ for all $|\tau| \ll \log^A x$, where the implicit constant here is sufficiently large with respect to $r,A$ and $\mbf{\beta}$. \\
\subsection{Main Result and Consequences}
Our main result is the following.
\begin{thm} \label{HalGen}
i) Let $x \geq 3$. Suppose $g\in \mc{C}$, and let $B := \max_{1 \leq j \leq m} B_j$, where $\mbf{B}$ implicitly chosen for $g$. Set $\tilde{g}(p) := g(p)/B_j$ for $p \in E_j$, and extend $\tilde{g}$ as a strongly or completely multiplicative function. Then for any $\alpha,D > 2$, $T := \log^D x$,
\begin{equation*} \label{UPPERGen}
M_g(x) \ll_{\alpha,B,m} \frac{B^2}{\delta}\frac{P}{\phi(P)}M_{|g|}(x) \exp\left(-\sum_{1 \leq j \leq m} B_j\left(\rho_{E_j}(x;\tilde{g},T) - \sum_{p \in E_j \atop p \leq x} \frac{1-|\tilde{g}(p)|}{p}\right)\right).
\end{equation*}
Suppose $g \in \mc{C}_b$, and set $\gamma_{0,j} := \frac{27\delta_j}{1024 \pi B_j}\beta_j^3$. Then we have, more precisely,
\begin{equation} \label{UPPER}
M_g(x) \ll_{B,m,\mbf{\phi},\mbf{\beta},r,\alpha} \frac{B^2}{\delta}\frac{P}{\phi(P)}M_{|g|}(x)\exp\left(-\sum_{1 \leq j \leq m}c_j\sum_{p \leq x \atop p \in E_j} \frac{|g(p)|-\text{Re}(g(p))}{p}\right),
\end{equation}
where 
\begin{equation*}
c_j := \begin{cases} \frac{1}{2}\min\left\{\frac{1}{4m^2},\frac{\gamma_{0,j}}{1+\gamma_{0,j}}\right\} &\text{ if $(g,\mbf{E})$ is non-decreasing} \\
\frac{1}{2}\min_j\left\{\frac{\delta_j}{B_j}\right\}\min\left\{\frac{1}{4m^2},\frac{\gamma_{0,j}}{1+\gamma_{0,j}}\right\} &\text{ if $(g,\mbf{E})$ is not non-decreasing}.
\end{cases}
\end{equation*}
ii) Suppose now that $g \in \mc{C}_a$. Set $A := \exp\left(\sum_{p \leq x} \frac{g(p)-|g(p)|}{p}\right)$, and let $\eta := \max_{1 \leq j \leq m} \eta_j$. Then
\begin{equation} \label{ASYMP}
M_g(x) = M_{|g|}(x)\left(A+O_{B,m,\alpha}\left(\mc{R}\right)\right),
\end{equation}
where, for $d_1 := \delta^{-1}\sqrt{\eta}$ if $B\eta \leq \delta \leq \eta^{\frac{1}{2}}$ and $d_1 := 1$ otherwise, and $\gamma_0 := \min_{1 \leq j \leq m} \gamma_{0,j}$,
\begin{equation*}
\mc{R} := \frac{B^2}{\delta}\frac{P}{\phi(P)}\left(\eta^{\frac{1}{2}}|A|\left(d_1 + \log^{-\frac{\delta}{3}} x + \delta^{-\frac{1}{2}}e^{-\frac{d_1}{\sqrt{\eta}}}\right) + |A|^{\frac{\gamma_0}{4(1+\gamma_0)}}\left(\log^{-\frac{\delta\beta^3}{2}} x + \delta^{-1}e^{-\frac{d_1}{\sqrt{\eta}}}\right)\right).
\end{equation*}
\end{thm}
(We have left an explicit dependence on $B$ in each of the above statements because, in the case that $B$ and $\delta$ are of the same order of magnitude and \emph{small}, the relative sizes of $B$ and $\delta$ can be accounted for, mitigating the large factor $\delta^{-1}$.)\\ \\
Let us explain the motivations for our choices for the collections $\mc{C}$, $\mc{C}_a$ and $\mc{C}_b$.  In the definition of $\mc{C}$, we introduce a partition $\mbf{E}$ in order to deal with functions $g$ that may have vastly different behaviours in different sets. For an example of this, see \cite{Res2}, in which we consider functions $g$ that are constant on each set of a partition, but where the values of these constants can have vastly different moduli. We are also allowing a fairly small set of primes $S$ on which $g$ is arbitrarily small, or possibly zero. Indeed, $S$ is small, because, as follows from the prime number theorem, $\prod_{p \leq x} p \sim e^{(1+o(1))x}$. Hence, polynomial growth of the product of primes from $S$ implies that $S$ is quite sparse. Our choice on the growth of $S$ is two-fold. Firstly, we would like to be able to handle twists of functions in $\mc{C}$ by characters in our main result. Thus, if $g \in \mc{C}$ is completely multiplicative and $\chi$ is a Dirichlet character modulo $q \leq x$, we would like to have access to estimates for the function $n \mapsto \chi(n)g(n)$, which, of course, is zero on primes $p|q$. In such a case, if $S$ is the associated set for $g$ and $S' := S \cup \{p: p|q\}$ then ii) applies to $\chi g$ as well and thus $\chi g \in \mc{C}$. Secondly, while we would like to admit the possibility for $g \in \mc{C}$ to have a zero set, it is more difficult to handle estimates for $M_{f}(x)$, for $f$ non-negative and 0 on $S$, if $f$ is supported on a set containing very few primes, e.g., $y$-smooth numbers, with $y$ quite small relative to $x$.  For a discussion that highlights this difficulty, see \cite{Hil}. As is clear from our arguments below (and as is necessary in the proof of Theorem \ref{LOWERMV}), sharp lower estimates on summatory functions of non-negative functions is crucial in our analysis.  \\
The choice of subcollection $\mc{C}_a$ is made to consider the specific case in which $g$ is ''close'' to being real valued, and in which case asymptotic formulae exist for the ratio $M_g(x)/M_{|g|}(x)$ (as well as in a more general context; see the discussion surrounding Theorem \ref{WIRSINGEXT}). \\
Our choice of definition of $\mc{C}_b$ is motivated by a desire to give explicit estimates for the growth rate of the ratio $\frac{M_g(x)}{M_{|g|}(x)}$ in Theorem \ref{HalGen} (see \eqref{UPPER}).  First, introducing multiple angular distribution conditions provides us with more leniency in our choice of functions than what is provided by Hal\'{a}sz' condition for \eqref{HALLLL}. Moreover, our qualification of our partition as being good, in the sense of Definition \ref{GOOD}, also affords us a general context in which complex analytic arguments are available to us.
\\ \\
Let us make a few remarks about Theorem \ref{HalGen}. First, while we only prove the above theorem for strongly multiplicative functions, the reader will notice that these results also hold for completely multiplicative functions as well, as noted in the statement. We highlight the relevant modifications necessary to make our arguments valid for completely multiplicative functions in remarks following each result in which arithmetic conditions on $g$ are implicitly being used. \\
When $P$ is large the factor $\frac{P}{\phi(P)}$ can be as large as $\log_2 P \sim \log_2 x$. If the series $\sum_{p \leq x} \frac{|g(p)|-\text{Re}(g(p))}{p}$ is not large, or if some of the ratios $\frac{\delta_j}{B_j}$ are small enough then we cannot beat the trivial bound $|M_g(x)|/M_{|g|}(x) \leq 1$ coming from the triangle inequality. When the first of the two scenarios just listed occurs, it is often true that $g \in \mc{C}_a$ (i.e., $|g(p)-|g(p)||$ is small uniformly in $p$), so in this case we can still get a sharp estimate. \\
The reader will also note that, given a strongly (or completely) multiplicative function, one can always partition the primes into sets $E_j := \{p : \text{arg}(f(p)) \in I_j\}$, for some partition $\{I_1,\ldots,I_m\}$ of $[-\pi,\pi]$, with $\phi_j-\pi$ chosen as the midpoint of $I_j$ and $\beta_j$ the minimum distance from $\phi_j$ and the endpoints of $I_j$. Thus, our Theorem \ref{HalGen} actually applies in general to \emph{any} strongly (or completely) multiplicative function whose values at primes are bounded in absolute value below and above by positive real numbers. The explicit results particular to the subcollection $\mc{C}_b$ can be used furthermore when such a partition is good (e.g., if $\chi$ is a character with conductor $q$ then we can classify the values of $\chi$ at primes according to residue classes modulo $q$, so $E_j$ is a union of arithmetic progressions).\\
We have given choices of constants for non-decreasing pairs $(g,\mbf{E})$ as, in many natural cases (such as $\frac{\phi(n)}{n}$ or $\lambda$, i.e., \emph{Liouville's function}), $|g|$ is either constant or monotone in subsets of the primes. At any rate, the explicit estimates in Theorem \ref{HalGen} turn out to be better in this case.\\
The constant $\frac{27}{1024\pi}$ implicit in the definition of $\gamma_{0,j}$ is found as an admissible choice for $D$ in a pointwise estimate of the form $\left|\frac{G(s)}{\mc{G}(\sg)}\right| \ll \left|\frac{G(\sg)}{\mc{G}(\sg)}\right|^D$. We have not made an attempt to find the optimal constant. \\
It was pointed out to the author that Theorem \ref{HalGen} is hinted at in the concluding section of Chapter 21 of Elliott's monograph \cite{Ell2} as an unpublished result of Hal\'{a}sz, but the author has not been able to find an explicit proof of it in the literature. Furthermore, it is to be emphasized that the result here is substantially stronger and more general than what is suggested in \cite{Ell2} on several fronts. First, our estimate admits the potential to exploit \emph{localized} behaviour of primes, i.e., to treat several sets of primes, on which $g$ behaves differently, separately. Second, the result mentioned in \cite{Ell2} assumes $\delta$ is fixed with $x$, so the scope of the theorem above is somewhat broader than what is explicit from Hal\'{a}sz' result. Third, the constant in the exponential in \eqref{UPPER} is dependent only on the ratio $\frac{\delta_j}{B_j}$, meaning that the estimate is more flexible and more widely applicable in case $|g(p)|$ is frequently quite small . Finally, as mentioned above, our theorem allows us to consider \emph{any} strongly or completely multiplicative function, provided its argument does not resemble that of an archimedean character, rather than those whose arguments do not belong to some distinguished, omitted sector of the unit disc, as in \eqref{HALLLL}.\\ \\
The main thread of the proof of Theorem \ref{HalGen} is to consider the function $h(n) := g(n)-A|g(n)|$, where here $A$ is either 0 in i) or the choice of $A$ given in ii), and transform the functions $M_g(x)$ (which arises even in the analysis of $M_h(x)$ for $A\neq 0$) and $M_{|g|}(x)$ into sums over primes by interchanging $\frac{M_g(x)}{M_{|g|}(x)}$ with a ratio of integral averages of $N_g(x) := \sum_{n \leq x} g(n)\log n$ and $N_{|g|}(x)$. The convolution $\log n = 1\ast \Lambda$ allows us to then use the bounds on primes directly, and translate this ostensibly arithmetic problem into an analytic problem involving Dirichlet series, in which case complex and harmonic analytic techniques are at our disposal. The themes in this treatment are all due originally to Hal\'{a}sz, but as our goal differs from his, a number of natural distinctions and refinements in relation to his work must arise in our analysis. \\ 
In proving the above result, we will deduce the following lower mean value estimate for a class of (not necessarily strongly or completely) multiplicative functions. While lower estimates of this kind are already known in a broader context (see, for instance, \cite{Hil}), our proof is substantially simpler and fits naturally into the framework of the rest of the proof of Theorem \ref{HalGen}.
\begin{thm} \label{LOWERMV}
Let $\lambda:\mb{N} \ra [0,\infty)$ be a multiplicative function such that: a) there exists a set $S$ of primes for which $P_x:= \prod_{p \in S \atop p \leq x} p \ll_{\alpha} x^{\alpha}$; b) there exist $\delta, B > 0$ for which $\delta \leq \lambda(p) \leq B$ for $p \notin S$ and $|g(p)| \leq \delta$ for $p \in S$; and c) $\sum_{p,k \geq 2} \frac{\lambda(p^k)}{p^k} \ll_B 1$. Then for $x$ sufficiently large and $P = P_x$,
\begin{equation*}
\sum_{n \leq x} \lambda(n) \gg_B \delta \frac{\phi(P_x)}{P_x}\frac{x}{\log x} \sum_{n \leq x} \frac{\lambda(n)}{n} \gg_B \delta \frac{\phi(P_x)}{P_x}\frac{x}{\log x} \prod_{p \leq x} \left(1+\frac{\lambda(p)}{p}\right).
\end{equation*}
\end{thm} 
As an application of Theorem \ref{HalGen}, we extend a theorem of Wirsing. Satz 1.2.2 in \cite{Wir} states the following (in slightly different notation).
\begin{thm}[Wirsing] \label{WIRSING}
Let $f(n)$ be a non-negative multiplicative function satisfying the asymptotic estimate $\frac{1}{\log x}\sum_{p \leq x} \frac{f(p)\log p}{p} \sim \tau$ for some $\tau > 0$, and suppose that $f(p) \leq C$ for all primes $p$. If $g:\mb{N} \ra \mb{R}$ satisfies $|g(n)| \leq f(n)$ for all $n$ then
\begin{equation}
\lim_{x \ra \infty}\frac{M_g(x)}{M_{f}(x)} = \prod_p \left(1+\sum_{k \geq 1} p^{-k}g(p^k)\right)\left(1+\sum_{k \geq 1} p^{-k}f(p^k)\right)^{-1}, \label{LIMIT}
\end{equation}
where the right side of the above limit is interpreted to be zero if the partial products over primes up to $x$ vanish as $x \ra \infty$.
\end{thm}
It is natural to consider finding an explicit estimate for the rate of convergence to the limiting value of the ratio $\frac{M_g(x)}{M_{f}(x)}$ as $x \ra \infty$.  We can easily prove such an estimate in the case that the ratio converges to 0 as a direct application of Theorems \ref{LOWERMV} and \ref{HalGen} i) with $g$ even \emph{complex-valued}. We can also adapt our proof of Theorem \ref{HalGen} ii) in order to prove estimates of this type for more general limiting values. In \emph{neither} of these results do we require the regularity assumption $\sum_{p \leq x} \frac{f(p)\log p}{p} \sim \tau \log x$.
\begin{thm}\label{WIRSINGEXT}
i) Let $g \in \mc{C}$, and let $\{c_j\}_j$ be the collection of coefficients defined in Theorem \ref{HalGen}. Suppose $f:\mb{N} \ra [0,\infty)$ is a multiplicative function satisfying $\delta \leq |g(p)| \leq f(p)\leq B$, $|g(n)| \leq f(n)$ for all $n$, and such that the right side of \eqref{LIMIT} is zero. Then, for $P = P_x$,
\begin{equation*}
\frac{|M_g(x)|}{M_{f}(x)} \ll_{B,\phi,\beta} \left(\frac{B}{\delta}\right)^3 \left(\frac{P}{\phi(P)}\right)^2e^{-\mc{F}(x;g)}\exp\left(-\sum_{p \leq x} \frac{f(p)-|g(p)|}{p}\right),
\end{equation*}
where, for $D > 2$ and $T := \log^D x$ and $\tilde{g}(p) := g(p)/B_j$ for $p \in E_j$ as above,
\begin{equation*}
\mc{F}(x;g) := \begin{cases} \sum_{1 \leq j \leq m} c_j \sum_{p \in E_j \atop p \leq x} \frac{|g(p)|-\text{Re}(g(p))}{p} &\text{ if $g \in \mc{C}_b$} \\ \sum_{1 \leq j \leq m} B_j\left(\rho_{E_j}(x;\tilde{g},T) - \sum_{p \in E_j \atop p \leq x} \frac{1-|\tilde{g}(p)|}{p}\right) &\text{ otherwise}. \end{cases}
\end{equation*}
ii) Suppose either $f$ and $g$ are both strongly or both completely multiplicative, and $g$ additionally satisfies the hypothesis $|g(p)-f(p)| \leq \eta$ for each prime $p$. Then if $A := \exp\left(-\sum_{p \leq x} \frac{|g(p)|-g(p)}{p}\right)$ and $X := \exp\left(-\sum_{p \leq x} \frac{f(p)-|g(p)|}{p}\right)$ then
\begin{equation*}
M_{g} (x) = M_{f}(x)\left(\exp\left(-\sum_{p \leq x} \frac{f(p)-g(p)}{p}\right) + O_{B,m}\left(\frac{P}{\phi(P)}\left(\mc{R}_1|A| + \mc{R}_2X\right)\right)\right),
\end{equation*}
where we have put
\begin{align*}
\mc{R}_1 := \left(\frac{B}{\delta}\right)^2 \left(X\left(d_1\eta^{\frac{1}{2}} + \log^{-\frac{\delta\beta^3}{2}} x + \delta^{-1}e^{-\frac{d_1}{\sqrt{\eta}}}\right) + \log^{-\frac{2\delta}{3}}x + \delta^{-1}e^{-\frac{2d_1}{\sqrt{\eta}}}\right). \\
\mc{R}_2 := \left(\frac{B}{\delta}\right)^2 \left(d_1\eta^{\frac{1}{2}}|A| + \log^{-\frac{2\delta}{3}}x + \delta^{-1}e^{-\frac{2d_1}{\sqrt{\eta}}} + |A|^{\frac{\gamma_0}{4(1+\gamma_0)}}\left(\log^{-\frac{\gamma}{2}} x + \gamma^{-1}e^{-\frac{d_1}{\sqrt{\eta}}}\right)\right).
\end{align*}
\end{thm}
Clearly, part ii) contains the case where the sum $\sum_p \frac{f(p)-g(p)}{p}$ (and therefore the ratio on the right side of \eqref{LIMIT}) converges.\\
The case $m = 1$ \emph{is} Hal\'{a}sz' theorem when $f \equiv 1$. Note that in this case, $\mbf{E}$ is trivially good and, in the worst case, $c := c_1 = \frac{\delta}{B}$ in Theorem \ref{HalGen}, which improves on the coefficient for Hal\'{a}sz' theorem given in Elliott \cite{Ell2} when $B$ is small. \\
It should be noted that Elliott \cite{Ell3} showed very recently a theorem of the above type that holds for \emph{any} complex-valued multiplicative function $g$ and a multiplicative function $f$ such that $|g(n)| \leq f(n)$, with the lone assumptions that $f$ (and thus $g$) is uniformly bounded on primes and that $\sum_{p,k\geq 2} \frac{f(p^k)}{p^k} < \infty$. However, his results are \emph{not} effective in that his asymptotic formulae do not have \emph{explicit} error terms. In particular, in the situation of i) in Theorem \ref{WIRSINGEXT} he does not provide a rate of convergence to the limiting value 0.  His method is also completely different from ours. \\
\\
The structure of this paper is as follows. In Section 3 we prove some auxiliary results to be used in the remainder of the paper. In Section 4 we translate the problem of determining the ratio $|M_h(x)|/M_{|g|}(x)$ to a problem of relating $\int_{(\sg)} |H'(s)/s|^2 |ds|$ (appearing in Section 6) to the Dirichlet series $\mc{G}(\sg)$, a purely analytic quantity. In Section 4.1 specifically we derive a proof of Theorem \ref{LOWERMV} as a step in this translation process, while in Section 4.2, we generalize arguments of Hal\'{a}sz in a manner that is suitable for our purposes.  In Section 5, we collect all of the analytic estimates we need to finish the proof of Theorem \ref{HalGen}. In particular, we strengthen a uniform estimate due to Hal\'{a}sz for $\left|G(s)/\mc{G}(s)\right|$ in terms of $\left|G(\sg)/\mc{G}(\sg)\right|$, for $s = \sg + i\tau$ that was mentioned earlier. In deriving this estimate we pay careful attention to treat each of the prime sets $E_j$ separately. In Section 6, we complete the proof of Theorem \ref{HalGen}, by implementing the estimates from Section 5 in the context of the results of the previous two sections. 
In Section 7 we prove Theorem \ref{WIRSINGEXT}. 
\section{Auxiliary Lemmata}
We shall need the following easy estimate. 
\begin{lem} \label{OLD}
Let $x \geq 2$, $\sg = 1+\frac{1}{\log x}$ and $s = \sg + i\tau_j$, with $\tau \in \mb{R}$. Then
\begin{equation*}
\sum_p \left|\frac{1}{p^{\sg}}-\frac{1}{p^{s}}\right| \ll 1+ \left|\log\left|\frac{|\tau|}{\sg-1}\right|\right|.
\end{equation*}
\end{lem}
\begin{proof}
By partial summation with the prime number theorem with de la Vall\'{e}e-Poussin error term, we have
\begin{align} \label{TAIL}
\sum_{p > u} \frac{1}{p^{\sg}} &= \int_u^{\infty} v^{-\sg}d\pi(v) = \int_u^{\infty} \frac{dv}{v\log v} e^{-(\sg-1)\log v} + O\left(\int_u^{\infty} \frac{dv}{v^{1+\sg}}e^{-\sqrt{\log v}}\right) \\
&= \int_{(\sg-1)\log u}^{\infty} \frac{dv}{v}e^{-v} + O\left(e^{-\sqrt{\log u}}\right) \ll 1.
\end{align}
Also, for any $\tau$ and $s  =\sg + i\tau$, we have
\begin{equation} \label{HEAD}
\sum_{p\leq e^{\frac{1}{|s-1|}}} \left|\frac{1}{p^s}-\frac{1}{p}\right| = \sum_{p\leq e^{\frac{1}{|s-1|}}} \frac{1}{p}\left|e^{-(s-1)\log p}-1\right| \leq |s-1|\sum_{p \leq e^{\frac{1}{|s-1|}}} \frac{\log p}{p} \ll 1,
\end{equation}
by Mertens' first theorem.  Note that $e^{\frac{1}{|s-1}} \leq e^{\frac{1}{\sg-1}}$ since $|s-1| \geq \sg-1$. Applying \eqref{HEAD} to both of $\sg$ and $s$, we have 
\begin{align*}
\sum_p \left|\frac{1}{p^{\sg}}-\frac{1}{p^{s}}\right| &\leq \sum_{p \leq e^{\frac{1}{\sg-1}}} \left|\frac{1}{p^{\sg}} - \frac{1}{p^{s}}\right| + 2\sum_{p > e^{\frac{1}{\sg-1}}} \frac{1}{p^{\sg}} \\
&\leq \sum_{p \leq e^{\frac{1}{\sg-1}}} \left|\frac{1}{p^{\sg}}-\frac{1}{p}\right| + \sum_{p \leq e^{\frac{1}{|s-1|}}} \left|\frac{1}{p^{s}}-\frac{1}{p} \right| + 2\sum_{e^{\frac{1}{|s-1|}} < p \leq e^{\frac{1}{\sg-1}}} \frac{1}{p} +O(1) \\
&=\left|\log\left|\frac{|s-1|}{\sg-1}\right|\right| + O(1),
\end{align*}
by applying Mertens' second theorem (note that in the second last expression only one of the sums in brackets is non-empty). The second claim is immediate upon bounding $|s-1| \leq (\sg-1) + |\tau|$.
\end{proof}
The following lemma is a simple consequence of the Prime Number Theorem, but we prove it for completeness.
\begin{lem}\label{PRECMERTENS}
i) There is a constant $c \in \mb{R}$ such that for any $A > 0$ and $z$ sufficiently large, we have
\begin{equation*}
\sum_{n \leq z} \frac{\Lambda(n)}{n} = \log z + c + O\left(e^{-\theta\sqrt{\log z}}\right),
\end{equation*}
for some $\theta > 0$. \\
ii) Let $P \in \mb{N}$. If $x$ is sufficiently large and $\theta' > 0$ sufficiently small, and $y > xe^{-\theta'\sqrt{\log x}}$ then 
\begin{equation*}
\sum_{x-y < n \leq x \atop (n,P) = 1} \frac{\Lambda(n)}{n} = \frac{\phi(P)}{P}\frac{y}{x} + O\left(\left(\left(\frac{y}{x}\right)^2 + e^{-\theta'\sqrt{\log x}}\right)\log P\right)
\end{equation*}
\end{lem}
\begin{proof}
i) By the Prime Number theorem with de la Vall\'{e}e-Poussin error term, $\psi(t) = t + \mc{E}(t)$ with $|\mc{E}(t)| \leq Cte^{-\theta\log^{\frac{1}{2}}t}$, for some $C$ sufficiently large and $\theta$ a positive absolute constant. Hence, applying partial summation, we have
\begin{align*}
\sum_{n \leq z} \frac{\Lambda(n)}{n} &= \int_{2^-}^z \frac{d\psi(t)}{t} = \int_{1}^z \frac{dt}{t} + \int_{2^-}^z \frac{d\mc{E}(t)}{t} \\
&= \log z + \frac{\mc{E}(z)}{z} - \frac{\mc{E}(2)}{2} + \int_{2^-}^z \frac{\mc{E}(t)}{t^2} dt =: \log z + c + O\left(e^{-\frac{\theta}{2}\sqrt{\log z}}\right),
\end{align*}
where we set $c := \int_{2^-}^{\infty} \frac{\mc{E}(t)}{t^2} dt - \frac{\mc{E}(2)}{2}$. This estimate follows because 
\begin{equation*}
\left|\int_z^{\infty} \frac{\mc{E}(t)}{t^2} dt\right| \ll e^{-\frac{\theta}{2} \log^{\frac{1}{2}} z} \int_z^{\infty} \frac{dt}{t}e^{-\frac{\theta}{2} \sqrt{\log t}} \ll e^{-\frac{\theta}{2}\sqrt{\log z}},
\end{equation*}
and this upper bound also clearly exceeds $\frac{\mc{E}(z)}{z}$ (this estimate also implies that the integral in the definition of $c$ exists, so $c$ is well-defined). This implies the estimate in i). \\
ii) Applying i) first with $z = x-y$ and then with $z = x$, and subtracting these two results gives
\begin{equation*}
\sum_{x-y < n \leq x} \frac{\Lambda(n)}{n} = -\log\left(1-\frac{y}{x}\right) + O\left(e^{-\frac{\theta}{2}\sqrt{\log (x-y)}}\right) = \frac{y}{x} + O\left(\left(\frac{y}{x}\right)^2 + e^{-\theta'\sqrt{\log x}}\right).
\end{equation*}
Now, given that $\Lambda(md) = \Lambda(m)$ whenever $\Lambda(md) \neq 0$ and $m > 1$, we have
\begin{align*}
\sum_{x-y < n \leq x \atop (n,P)=1} \frac{\Lambda(n)}{n} &= \sum_{d | P} \mu(d)\sum_{(x-y)/d < m \leq x/d} \frac{\Lambda(md)}{md} = \sum_{d|P} \frac{\mu(d)}{d}\sum_{(x-y)/d < m \leq x/d} \frac{\Lambda(m)}{m} \\
&= \frac{y}{x}\sum_{d|P} \frac{\mu(d)}{d} + O\left(\left(\sum_{d\leq P} \frac{1}{d}\right)\left(\left(\frac{y}{x}\right)^2 + e^{-\theta'\sqrt{\log x}}\right)\right) \\
&= \frac{\phi(P)}{P}\frac{y}{x} + O\left(\left(\left(\frac{y}{x}\right)^2 + e^{-\theta'\sqrt{\log x}}\right)\log P\right).
\end{align*}
This completes the proof.
\end{proof}
The following estimates, due to Selberg, will provide us with a maximal growth order for our mean values, and will also play a role in the proof of Lemma \ref{MAIN1}.
\begin{lem}\label{SELBERG}
Let $x$ be sufficiently large, $B > 0$ and $0 < \rho \leq B$. Then 
\begin{equation} \label{SelSum}
\sum_{n \leq x} \rho^{\omega(n)} = x(\log x)^{\rho-1}F(\rho)\left(1+O_B\left(\frac{1}{\log x}\right)\right),
\end{equation}
where $F(\rho) := \Gamma(\rho)^{-1}\prod_p \left(1+\frac{\rho}{p-1}\right)\left(1-\frac{1}{p}\right)^{\rho}$. In particular, if $\delta > 0$ is fixed with $\delta \leq B$ and $g$ is a strongly multiplicative function satisfying $\delta \leq |g(p)| \leq B$ for each $p$ then $x\log^{-\delta + 1} x \ll_{\delta} \left|M_g(x)\right| \ll_B x\log^{B-1} x$. \\
\end{lem}
\begin{proof}
The first estimate, due to Selberg, is proved, for instance, in Chapter II.6 of \cite{Ten2}. The second follows from the first via the triangle inequality, given that $\delta^{\omega(n)} \leq |g(n)| = \prod_{p|n} |g(p)| \leq B^{\omega(n)}$ whenever $g$ is strongly multiplicative.
\end{proof}
\begin{rem}
An analogous estimate exists when $\omega$ is replaced by $\Omega$ and $B < 2$, but for a different choice of $F(\rho)$. Consequently, an upper bound for $|M_g(x)|$ of the same shape as above holds for completely multiplicative functions as well (indeed, $|f(m)| = \prod_{p^k || m} |f(p)|^k \leq B^{\sum_{p^k||m} k} = B^{\Omega(m)}$ in this case).
\end{rem}
A crucial element of the proof of Theorem \ref{HalGen} is a uniform bound in $\tau$ of $\left|\frac{G(s)}{\mc{G}(\sg)}\right|$, for $s = \sg + i\tau$, by a (fractional) power of $\left|\frac{\mc{G}(\sg)}{\mc{G}(\sg)}\right|$. As a first step in deriving such a bound, we will need to relate $\frac{G(s)}{\mc{G}(\sg)}$ to a power of $\frac{\mc{G}(s)}{\mc{G}(\sg)}$. 
The following simple trigonometric inequality will bridge the gap between the first step and the desired estimate (see Lemma \ref{INTBOUND} below).
\begin{lem} \label{TrigProp}
Let $m \in \N$. Then for $a_1,\ldots,a_m \in \R$,
\begin{align}
\sin^2\left(\sum_{1 \leq j \leq m} a_j\right) \leq m\sum_{1 \leq j \leq m} \sin^2 a_j \label{PROP1}. 
\end{align}
\end{lem}
\begin{proof}
This follows by induction on $m\geq 1$ upon setting $A := a_1 + \ldots + a_{m-1}$ via the inequality
\begin{equation*}
\sin^2(A+a_m) \leq (|\sin(A)||\cos(a_m)| + |\cos(A)||\sin(a_m)|)^2 \leq (|\sin(A)| + |\sin(a_m)|)^2,
\end{equation*}
and the inequality $\left(\sum_{1 \leq j \leq m} r_j\right)^2 \leq m\sum_{1 \leq j \leq m} r_j^2$.
\end{proof}
In spite of our limited local knowledge of $g$ on short intervals, we can at least approximate the behaviour of $M_{|g|}$ on intervals of the form $(a,ca]$ with $c > 1$. We will use this in the proof of Theorem \ref{LOWERMV} below.
\begin{lem}\label{DobApp}
Let $a \geq 2$, $B \geq 1$ and $1 < c \ll 1$ as $a \ra \infty$. Suppose $g:\mb{N} \ra (0,\infty)$ is a strongly or squarefree-supported multiplicative function, such that $g(p) \leq B$ for all primes $p$. Furthermore, assume that $\frac{M_g(t) \log t}{t} \ra \infty$ as $t \ra \infty$. Then $M_g(ca)-M_g(a) \ll_c BM_g(a)$.
\end{lem}
\begin{proof}
By the Prime Number theorem, 
\begin{align*}
\sum_{a < p \leq ca} g(p) &\leq B\left(\pi(ca) - \pi(a)\right) = B\left(\frac{ca}{\log a + \log c} - \frac{a}{\log a} + O\left(\frac{a}{\log^2 a}\right)\right) \\
&= B(c-1)\frac{a}{\log a}\left(1 + O_c\left(\frac{1}{\log a}\right)\right) = o_{B,c}(M_g(a)).
\end{align*}
Thus, $M_g(ca)-M_g(a) = \sum_{a < n \leq ca \atop a \text{ not prime}} g(n) + o_{B,c}(M_{g}(a))$. Call the sum in this last expression $S(a)$.  We split $S(a) = S_1(a) + S_2(a)$, where $S_1(a)$ is supported by integers $n \in (a,ca]$ satisfying $P^+(n) \leq c$, where $P^+(n)$ denotes the largest prime factor of $n$, and $S_2(a)$ is supported on the complement of the support of $S_1(a)$.  Now, since $g$ is strongly multiplicative and each $c$-smooth number has at most $\pi(c)$ distinct prime factors,
\begin{equation*}
S_1(a) = \sum_{a < n \leq ca \atop a \text{not prime}, P^+(n) \leq c} g(n) \leq B^{\pi(c)} \sum_{a < n \leq ca \atop P^+(n) \leq c} 1 =: B^{\pi(c)}\left(\Psi(ca,c)-\Psi(a,c)\right),
\end{equation*}
where $\Psi(u,v) := |\{n \leq u : P^+(n) \leq v\}|$. By a theorem of Ennola (see Theorem III.5.2 in \cite{Ten2}), we have, for $ c \ll 1$,
\begin{equation*}
\psi(u,c) = \frac{1}{\pi(c)!}\left(1+O_c\left(\frac{1}{\log u}\right)\right)\prod_{p \leq c} \frac{\log u}{\log p},
\end{equation*}
whence we have
\begin{equation*}
\psi(ca,c)-\psi(a,c) = \left(1+O_c\left(\frac{1}{\log a}\right)\right)\left(\left(1+\frac{\log c}{\log a} \right)^{\pi(c)}-1\right)\left(\prod_{p \leq c} \frac{\log a}{\log p}\right) \ll_c \log^{\pi(c)} a.
\end{equation*}
Hence, we have $S_1(a) \ll_c \frac{a}{\log a}$ easily. \\
Now, for each $n$ in the support of $S_2(a)$ we can write $n = p^k m$ with $p > c$, $(m,p) = 1$, and $\frac{a}{p^k} < m \leq \frac{ca}{p^k} \leq a$ (if $g$ is only supported on squarefrees then assume $k = 1$). For each such $m$, we let $n_a(m)$ denote the number of choices of $n \in (a,ca]$ composite for which $m = \frac{n}{p^k}$, with $p^k || n$. We claim that $n_a(m) \leq D$ uniformly in $a$, where $D\ll_c 1$. The statement of the lemma will then follow because
\begin{equation*}
\sum_{a < n \leq ca \atop n \text{ not prime}, P^+(n) > c} g(n) \leq B\sum_{m \leq a} n_a(m)g(m) \leq DBM_g(a).
\end{equation*}
Assume for the sake of contradiction that $\limsup_{a \ra \infty} \left(\max_{m \leq a} n_a(m)\right) = \infty$. Thus, for any $N$ and $a$ sufficiently large we can select an integer $m \leq a$ such that $n_a(m) := R \geq 2N+1$.  Accordingly, there exist $R$ prime powers $p_1^{k_1} < \ldots < p_R^{k_R}$ for which $p_j^{k_j}m \in (a,ca]$.  Let $r_{ij} := p_i^{k_i}p_j^{-k_j}$, with $i > j$. There are $\frac{1}{2}R(R-1) \geq 2N^2+1$ such ratios, and $r_{ij} \in (1,c]$.  We split this latter interval into $N^2$ intervals $(1+\frac{(c-1)l}{N^2}, 1+\frac{(c-1)(l+1)}{N^2}] =: I_N(l)$.  By the pigeonhole principle, there exists some $l_0$ for which there are two distinct pairs $(i_1,j_1)$ and $(i_2,j_2)$ such that $r_{i_1j_1} < r_{i_2j_2} \in I_N(l_0)$. Set $r := \frac{r_{i_2j_2}}{r_{i_1j_1}} \in \left(1,1+\frac{c-1}{N^2}\right)$. Assume for the moment that one of $k_{i_1},k_{i_2},k_{j_1},k_{j_2}$ is not divisible by some fixed prime $q$ (if $g$ is supported on squarefree integers then this is trivial since they are all 1 whenever $g(m) \neq 0$). Thus, let $\alpha_q := r^{\frac{1}{q}}$, which is an algebraic integer of degree $q$.  Note that $\mb{Q}(\alpha_q)$ is a Kummer extension of degree at least 2 with abelian Galois group corresponding to multiplication by $q$th roots of unity, with minimal polynomial $x^q - r$. It follows that $|\sg(\alpha_q)| = |\alpha_q|$ for each $\sg \in \text{Gal}(\mb{Q}(\alpha_q)/\mb{Q})$, and the Mahler measure of $x^q-r$ is $M(\alpha_q) = r$; thus, the Weil height of $\alpha$ is $\log\alpha = \frac{1}{q}\log r$. Now, by Dobrowolski's theorem (see Section 4.4 in \cite{BoG}), $\log \alpha \geq \frac{1}{4q} \left(\frac{\log_2 3q}{\log 3q}\right)^3$, whence it follows that $\log r \geq \frac{1}{4}\left(\frac{\log_2 3q}{\log 3q}\right)^3$.  Hence, for $q \ll 1$ as $a \ra \infty$, we see that we can choose $N$ large enough so that $\frac{1}{N^2} < \frac{1}{8}\left(\frac{\log_2 q}{\log q}\right)^3$. On the other hand, $\log r \leq \log\left(1+\frac{1}{N^2}\right) \leq \frac{2}{N^2}$, contradicting the conclusion of Dobrowolski's theorem. \\
Assume now that no such $2 \leq q \ll 1$ exists. Then $x^2 + r$ generates a quadratic extension for which $\log r \geq \frac{1}{2}\left(\frac{\log_2 6}{\log 6}\right)^3$, by Dobrowolski's theorem (which holds for extensions of degree at least 2). The same contradiction as above follows.
\end{proof}
Let $L_g(u) := \sum_{n \leq u} \frac{g(n)}{n}$ and $P_g(u) := \sum_{p \leq u} \left(1+\frac{g(p)}{p}\right)$. The following lemma relating $L_g(u)$ and $P_g(u)$ is standard, but we prove it for completeness. It will be necessary for us in conjunction with Theorem \ref{LOWERMV} in order to relate $M_{|g|}(x)$ with $\mc{G}(\sg)$.
\begin{lem} \label{LOGSUM}
Let $B \geq 1$ and let $g : \mb{N} \ra [0,\infty)$ be a multiplicative function with $|g(p)| \leq B$ for each prime $p$, and such that $\sum_{p^k, k \geq 2} \frac{g(p^k)}{p^k} \ll_B 1$. Then for each $u$ sufficiently large, $L_g(u) \asymp_B P_g(u)$. 
\end{lem}
The hypotheses of this lemma are clearly satisfied by a strongly or completely multiplicative function that is uniformly bounded on primes.
\begin{proof}
Set $S := \sum_{p, k \geq 2} \frac{g(p^k)}{p^k}$. The upper bound follows immediately from
\begin{align*}
L_g(u) &\leq \sum_{P^+(n) \leq u} \frac{g(n)}{n} = \prod_{p \leq u} \left(1+\frac{g(p)}{p} + \sum_{k \geq 2} \frac{g(p^k)}{p^k}\right) \leq P_g(u)\prod_{p \leq u}\left(1+\sum_{k \geq 2} \frac{g(p^k)}{p^k}\right) \leq e^SP_g(u).
\end{align*}
To derive a lower bound we can no longer bound the set of $n \leq u$ by those with largest prime factor less than $u$. Instead, we let $\kappa$ be a parameter to be chosen, and bound $L_g(u)$ from below by $P_g(u^{\kappa})$, since for $\kappa > 0$ this should be of the same order as $P_g(u)$. Precisely,
\begin{equation*}
L_g(u) \geq \sum_{n \leq u \atop P^+(n) \leq u^{\kappa}} \mu^2(n)\frac{g(n)}{n} = \sum_{P^+(n) \leq u^{\kappa}} \mu^2(n)\frac{g(n)}{n} - \sum_{n > u \atop P^+(n) \leq u^{\kappa}} \mu^2(n)\frac{g(n)}{n} = P_g(u^{\kappa}) - \sum_{n > u \atop P^+(n) \leq u^{\kappa}} \mu^2(n)\frac{g(n)}{n}.
\end{equation*}
We will also bound the second sum above by a multiple $P_g(u^{\kappa})$, using the condition $n > u$ to save a constant factor. Indeed, by Rankin's trick, for $\e > 0$ a second parameter to be chosen,
\begin{equation*}
\sum_{n > u \atop P^+(n) \leq u^{\kappa}} \mu^2(n)\frac{g(n)}{n} \leq u^{-\e}\sum_{P^+(n) \leq u^{\kappa}} \mu^2(n) \frac{g(n)}{n^{1-\e}} = u^{-\e} \prod_{p \leq u^{\kappa}} \left(1+\frac{g(p)}{p}e^{\e \log p}\right).
\end{equation*}
When $\e$ is sufficiently small then this last factor is approximately $P_g(u^{\kappa})$. Indeed, observe that for each $p \leq u^{\kappa}$,
\begin{equation*}
\left(1+\frac{g(p)}{p}e^{\e\log p}\right)\left(1+\frac{g(p)}{p}\right)^{-1} = 1+\frac{g(p)\left(e^{\e\log p}-1\right)}{p+g(p)} \leq 1+2B\e\frac{\log p}{p}.
\end{equation*}
so, setting $\e := \frac{1}{\kappa B \log u}$, we have
\begin{equation*}
P_g(u^{\kappa})^{-1}\prod_{p \leq u^{\kappa}} \left(1+\frac{g(p)}{p^{1-\e}}\right) \leq \prod_{p \leq u^{\kappa}} \left(1+2B\e\frac{\log p}{p}\right) \leq \exp\left(2B\e \sum_{p \leq u^{\kappa}} \frac{\log p}{p}\right) = e + o(1).
\end{equation*}
It therefore follows that
\begin{equation*}
\sum_{n > u \atop P^+(n) \leq u^{\kappa}} \mu^2(n)\frac{g(n)}{n} \leq u^{-\e}eP_g(u) = e^{1-\frac{1}{\kappa B}}P_g(u)(1+o(1)).
\end{equation*}
We now select $\kappa = \frac{1}{2B}$ if $B \geq \frac{1}{2}$, and $\kappa = \frac{1}{2}$ otherwise.  Then
\begin{equation*}
L_g(u) \geq P_g(u^{\kappa}) \left(1-2e^{1-\frac{1}{\kappa B}}\right)P_g(u^{\kappa}) \geq (1-2e^{-1})P_g(u^{\kappa}),
\end{equation*}
for $u$ sufficiently large. Finally, to complete the proof it suffices to show that $P_g(u) \asymp P_g(u^{\kappa})$. Indeed, when $u > B^{2B}$,
\begin{equation*}
\frac{P_g(u)}{P_g(u^{\kappa})} = \prod_{u^{\kappa} < p \leq u} \left(1+\frac{g(p)}{p}\right) \leq \exp\left(B\sum_{u^{\kappa} < p \leq u} \frac{1}{p}\right) \leq e^{(B+1)\log\left(\frac{1}{\kappa}\right)} \leq e^{(B+1)\log (4\max\{B,1/2\})}.
\end{equation*}
Hence, we have
\begin{equation*}
L_g(u) \geq (1-2e^{-1})P_g(u)\left(\frac{P_g(u^{\kappa})}{P_g(u)}\right) \geq (1-2e^{-1})e^{-(B+1)\log(4\max\{B,1/2\})}P_g(u),
\end{equation*}
and the proof is complete.
\end{proof}
As a consequence of the lower bound in the previous lemma (which is the non-trivial part of it), we have the following.
\begin{lem} \label{DENOMPARS}
Let $u \geq 3$ be sufficiently large and $\sg := 1+\frac{1}{\log u}$. Let $g$ satisfy the hypotheses of the Lemma \ref{LOGSUM}, let $G(s)$ be its Dirichlet series and assume that this converges absolutely in the half-plane $\text{Re}(s) > 1$. Then $L_g(u) \gg_B G(\sg)$.
\end{lem}
\begin{proof}
By Lemma \ref{LOGSUM} we have
\begin{align*}
L_g(u) &\gg_B \prod_{p \leq u} \left(1+\frac{g(p)}{p}\right) = \prod_{p} \left(1+\frac{g(p)}{p^{\sg}}\right)\left(\prod_{p \leq u} \left(1+\frac{g(p)}{p}\right)\left(1+\frac{g(p)}{p^{\sg}}\right)^{-1}\right)\prod_{p > u} \left(1+\frac{g(p)}{p}\right)^{-1} \\
&\geq \prod_p \left(1+\frac{g(p)}{p^{\sg}}\right)\left(\prod_{p > u} \left(1+\frac{g(p)}{p}\right)^{-1}\right),
\end{align*}
the last estimate following because $g$ is non-negative and $\sg > 1$. Now, 
\begin{equation*}
\prod_{p > u} \left(1+\frac{g(p)}{p^{\sg}}\right) \leq \exp\left(B\sum_{p > u} \frac{1}{p^{\sg}}\right) \ll_B 1,
\end{equation*}
by \eqref{TAIL}.
Also,
\begin{equation*}
G(\sg)\prod_p \left(1+\frac{g(p)}{p^{\sg}}\right)^{-1} \leq \prod_p \left(1+\sum_{\nu \geq 2} \frac{g(p^{\nu})}{p^{\nu\sg}}\right) \ll_B 1.
\end{equation*}
It follows that
\begin{equation*}
L_g(u) \gg_B G(\sg)\left(G(\sg)\prod_p \left(1+\frac{g(p)}{p^{\sg}}\right)^{-1}\right)^{-1} \gg_B G(\sg),
\end{equation*}
as claimed.
\end{proof}
\section{Arithmetic Estimates}
\subsection{Lower Bounds for $M_{|g|}$}
In this section we bound $M_{|g|}(t)$ from below.  Laterally, our estimate will essentially suffice to prove Theorem \ref{LOWERMV}. \\
We will need to relate $N_{|g|}(t)$ to an integral average of itself on a short interval. This same method, used in the next section as well to deal with $N_h(t)$ with $h(n) := g(n)-A|g(n)|$ will permit us a passage towards harmonic analytic methods. The next lemma will be essential in this regard.
\begin{lem} \label{STRONG}
Suppose $g$ is strongly multiplicative.
Then $N_g(x) = \sum_{d \leq x} \Lambda(d)g(d)M(x/d)$ for any $x \geq 1$.
\end{lem}
\begin{proof}
By definition, we have 
\begin{equation*}
N_g(x) = \sum_{n \leq x} g(n)\log(n) = \sum_{d \leq x} \Lambda(d) \sum_{m\leq x/d} g(md).
\end{equation*}
Write $d = p^l$. Then
\begin{equation*}
\sum_{m \leq x/p^l} g(p^lm) = \sum_{t \geq l} \sum_{p^tm \leq x \atop p\nmid m} g(p^tm) = \sum_{t \geq l} g(p^t)\sum_{p^tm \leq x \atop p\nmid m} g(m) = g(p)\sum_{t\geq l} \sum_{m \leq x/p^l \atop p^{t-l}||m} g(m) = g(p)\sum_{m \leq x/p^l} g(m),
\end{equation*}
so that
\begin{equation*}
\sum_{p^l \leq x} \Lambda(p^l) \sum_{m \leq x/p^l} g(mp^l) = \sum_{p^l \leq x} \Lambda(p^l) g(p)M_g(x/p^l).
\end{equation*}
\end{proof}
\begin{rem}
The above proof is completely trivial when $g$ is completely multiplicative, because $g(p^tm) = g(p^t)g(m)$, regardless of $m$. 
\end{rem}
\begin{lem} \label{COMPAVG}
Let $2 \leq y \leq t \leq x$ with $y = e^{-\log^c t}$ for $c < \frac{1}{2}$. For $A \in \mb{C}$ with $|A| \in [0,1]$ let $h(n) := g(n)-A|g(n)|$ and set $\mu:= \max_p |\text{arg}(g(p))|$. Then 
\begin{align*}
|N_h(t)| &= y^{-1}\int_{t-y}^t |N_h(u)| du + O_B\left(\frac{M_{|g|}(t)}{\log^2 t}\right)\\
&\leq y^{-1}\int_{t-y}^t du \left(\sum_{d \leq u} \Lambda(d)|g(d)||M_h(u/d)| + |A|\mu \sum_{d \leq u} \Lambda(d)|g(d)|M_{|g|}(u/d)\right) + O_B\left(\frac{M_{|g|}(t)}{\log^2 t}\right).
\end{align*}
\end{lem}
When $A = 0$ then $h = g$, so this lemma reduces trivially to a statement about sums of $g$ as well.
\begin{proof}
Observe that for $t-y < u \leq t$, Lemma \ref{SELBERG} implies that
\begin{align*}
||N_h(u)|-|N_h(t-y)|| &\leq |N_h(u)-N_h(t-y)| = \left|\sum_{t-y < n \leq u} \left(g(n)-A|g(n)|\right)\log n\right| \\
&\leq 2(1+|A|)\left(\sum_{t-y < n \leq t} |g(n)| \right) \log t \\
&\asymp_B (1+|A|)\left(t\log^{B-1}t-(t-y)\log^{B-1}(t-y)\right)\log t \\
&\leq (1+|A|)\left(t\left(1+\frac{\log \left(1+\frac{y}{t-y}\right)}{\log t}\right)^{B-1} - t\right)\log^{B} t + y(1+|A|)\log^B t\\
&\leq (1+|A|)\left(t\left(1+2B\frac{y}{(t-y)\log t}\right) - t\right)\log^B t +y(1+|A|)\log^B t\ll y\log^B t,
\end{align*}
the second last inequality holding when $t$ is sufficiently large in terms of $B$. Hence, for each $u \in (t-y,t]$ and any $K > 2$, 
\begin{equation*}
|N_h(u)| = |N_h(t)| + O\left(\frac{t}{\log^{K-1} t}\right) = |N_h(t)| + O\left(\frac{M_{|g|}(t)}{\log^2 t}\right),
\end{equation*}
upon taking $K = 4$ and using the lower bound from Lemma \ref{SELBERG}.
It then follows that
\begin{equation} \label{FIRST}
|N_h(t)| = y^{-1}\int_{t-y}^{t} |N_h(u)| du + O\left(\frac{M_{|g|}(t)}{\log^2 t}\right).
\end{equation}
Now, taking $u \in (t-y,t]$ and using $\log = 1 \ast \Lambda$, we have
\begin{equation*}
|N_h(u)| = \left|\sum_{d\leq u} \Lambda(d)\sum_{m \leq u/d} \left(g(md)-A|g(md)|\right) \right| = \left|\sum_{d \leq u} \Lambda(d) \sum_{m \leq u/d} g(md) - A\sum_{d\leq u} \Lambda(d)\sum_{m \leq u/d} |g(md)|\right|.
\end{equation*}
Note that both $g$ and $|g|$ are strongly multiplicative. Applying Lemma \ref{STRONG}, we get
\begin{align}
|N_h(u)| &= \left|\sum_{d \leq u} \Lambda(d) g(d) \sum_{m \leq u/d} g(m) - A\sum_{d\leq u} \Lambda(d)|g(d)|\sum_{m \leq u/d} |g(m)|\right| \\
&\leq \left|\sum_{d \leq u} \Lambda(d)g(d) \sum_{m \leq u/d} \left(g(m) -A |g(m)|\right)\right| + |A|\left|\sum_{d\leq u} \Lambda(d)(g(d)-|g(d)|)\sum_{m \leq u/d} |g(m)|\right| \nonumber\\
&\leq \sum_{d \leq u} \Lambda(d)|g(d)||M_h(u/d)| + |A|\sum_{d \leq u} \Lambda(d)|g(d)|\left|\frac{g(d)}{|g(d)|} - 1\right|M_{|g|}(u/d) \label{TWOSUMS}.
\end{align}
%
As $\frac{g(p)}{|g(p)|} = e^{i\arg{g(p)}}$, the mean value theorem (of calculus) implies that if $d = p^k$ for some $k \in \mb{N}$,
\begin{equation*}
\left|\frac{g(d)}{|g(d)|} - 1\right| = \left|e^{i\arg(g(p))} - 1\right| = \left|\arg(g(p))\right| \leq \mu.
\end{equation*}
Inserting this estimate into \eqref{TWOSUMS}, we get
\begin{equation*}
|N_h(u)| \leq \sum_{d \leq u} \Lambda(d)|g(d)||M_h(u/d)| + |A|\mu\sum_{d \leq u} \Lambda(d)|g(d)|M_{|g|}(u/d).
\end{equation*}
Inputting this estimate into \eqref{FIRST} for each $u \in (t-y,t]$ implies the claim.
\end{proof}
\begin{rem}
When $g$ is completely multiplicative, in which case $|g(d)/|g(d)| - 1| = |\text{arg}(g(p^k))|$ for $k$ possibly at least 2, we must split the sum as
\begin{equation} \label{COMPCHANGE}
\sum_{d \leq u} \Lambda(d)|g(d)|\left|\frac{g(d)}{|g(d)|} - 1\right|M_{|g|}(u/d) \leq \mu\sum_{p \leq u} \Lambda(p) |g(p)|M_{|g|}(u/p) + \mu\sum_{k \geq 2} k\sum_{p^k \leq u} \Lambda(p^k)|g(p^k)|M_{|g|}(u/p^k).
\end{equation}
We claim that we can bound the second sum by $\ll_B \mu M_{|g|}(u)\log u$. When $B \leq 1$, this can easily be verified since, by Lemma \ref{SELBERG},
\begin{equation*}
\sum_{k \leq 2} k\sum_{p^k \leq u} \Lambda(p^k)|g(p^k)|M_{|g|}(u/p^k) \leq B^2u \sum_p \frac{\log p}{p^2} \sum_{k \geq 0}\frac{(k+2)B^k}{p^k} \ll B^2 u \ll_{\delta} B^2 M_{|g|}(u)\log^{1-\delta u}.
\end{equation*}
For $1 < B < 2$ we must work a little harder. Let $Q \geq 2$ be a parameter to be chosen. We split the second sum into sums over $[1,Q]$ and $[Q,x]$.  In the first interval, using $M_{|g|}(u/d)|g(d)| \leq M_{|g|}(u)$, we have
\begin{equation*}
\sum_{p^k \leq Q \atop 2 \leq k \leq \log Q/\log 2} k\Lambda(p^k) |g(p^k)| M_{|g|}(u/p^k) \ll \left(\sum_{a \leq Q^{\frac{1}{2}}} \Lambda(a)\right)M_{|g|}(u)\log^2 Q \ll M_{|g|}(u) Q^{\frac{1}{2}} \log^2 Q.
\end{equation*}
In the second interval, we further split the sum over $Q < p^k \leq u$, according as $p \leq P_0$ or not, for $P_0 \geq 2$ a parameter to be chosen. In the first case, we have
\begin{equation*}
\sum_{Q < p^k \leq u \atop p \leq P_0, k \geq 2} k\Lambda(p^k) |g(p^k)| M_{|g|}(u/p^k) = \sum_{n \leq u} |g(n)| \sum_{p^k | n}^{\ast} k\log p,
\end{equation*}
where the asterisk indicates that the support of the sum consists of $p^k | n$ such that $Q < p^k \leq u$, $p \leq P_0$ and $k \geq 2$.  Since $k \log p > \log Q$, each $n$ has at most $\frac{\log n}{\log Q}$ factors $p^k$ of the latter shape and it follows that
\begin{equation} \label{SECONDONE}
\sum_{n \leq u} |g(n)| \sum_{p^k | n}^{\ast} k\log p \leq \frac{\log u}{\log Q} M_{|g|}(u).
\end{equation}
In the remaining sum, we apply Lemma \ref{SELBERG} and the elementary identity $\sum_{l \geq 1} lt^l = t/(1-t)^2$ to get
\begin{align}
&\sum_{Q < p^k \leq u \atop p > P_0, k \geq 2} k\Lambda(p^k) B^k M_{|g|}(u/p^k) \ll_B Bu\log^{B-1}(u/Q)\sum_{p} \frac{\log p}{p^2}\sum_{k \geq 1 \atop p^k > Q} k\frac{B^{k-1}}{p^{k-1}} \nonumber\\
&\ll Bu\log^{B-1}(u/Q)\sum_{p > P_0} \frac{\log p}{p}\sum_{k \geq \log Q/\log p} k\frac{B^k}{p^k} \nonumber\\
&\ll B^2 u \log^{B-1}(u/Q) \sum_{p > P_0} \frac{\log p}{(p-B)^2}\left(\frac{B}{p}\right)^{\frac{\log Q}{\log p}} \nonumber.
\end{align}
The last expression is 
\begin{equation*}
e^{-\frac{\log Q}{\log p}\left(\log p - \log B\right)} \leq e^{-\left(1-\frac{\log B}{\log P_0}\right)\log Q}.
\end{equation*}
Thus,
\begin{equation*}
\sum_{Q < p^k \leq u \atop k \geq 2, p > P_0} k\Lambda(p^k) B^k M_{|g|}(u/p^k) \ll_B B^2 u \log^{B-1}(u/Q)e^{-(1-\frac{\log B}{\log P_0})\log Q}.
\end{equation*}
Write $Q = \log^{2-r} u$, for $0 < r < 2$. Choosing $r > 0$ and $P_0$ such that $1+(2-r)\left(1-\frac{\log B}{\log P_0}\right) \geq B-\delta$, which is possible since $B-\delta < 2-\delta$, it follows that this last bound is $\ll B^2u \log^{1-\delta} u \ll B^2 N_{|g|}(u)\log^{-\delta}u$, as above (when we considered $B \leq 1$). Of course, with this choice of $Q$, we have $M_{|g|}(u) Q^{\frac{1}{2}}\log^2 Q \ll M_{|g|}(u)\log^{1-r/2 +o(1)}u$, and we may take $r = 1/2$, $P_0 = B^3$, and then the bound for the sum \eqref{SECONDONE} is $\ll_B M_{|g|}(u) \log u/\log_2 u$.  All told, this indeed shows that
\begin{equation*}
\sum_{p^k \leq Q \atop k \geq 2} k\Lambda(p^k) |g(p^k)| M_{|g|}(u/p^k) \ll B^2 \mu M_{|g|}(u) \log u.
\end{equation*}
Thus, in the case of completely multiplicative functions we only add the term $\mu M_{|g|}(t)\log t$ when $u = t$. Once divided by $\log t$ in transitioning from $|M_{h}(t)|$ to $|N_h(t)|$ (see Lemma \ref{NUM}) and by $M_{|g|}(t)$ when calculating the ratio $|M_h(t)/M_{|g|}(t)|$, this error term has the same order of magnitude $\mu$ as what results in the strongly multiplicative case (see Proposition \ref{STARTER}).
\end{rem}
We next bound $M_{|g|}$ from below by an integral of itself on a longer interval. It turns out that this latter integral, in turn, is bounded below by $\mc{G}(\sg)$, which will be of use in later sections.
\begin{lem} \label{DENOM}
Let $t$ be sufficiently large, $g: \mb{N} \ra \mb{C}$ and given $S := \{p : g(p) = 0\}$, let $P_t := \prod_{p \leq t \atop p \in S} p$. Suppose that $P_t \ll_{\alpha} t^{\alpha}$ for any $\alpha > 0$. Then
\begin{equation*}
M_{|g|}(t) \gg_{B,r} \delta \frac{\phi(P_t)}{P_t}\frac{t}{\log t} \int_1^{t} \frac{M_{|g|}(u)}{u^2} du.
\end{equation*}
\end{lem}
\begin{proof}
Write $\Delta_{f}(u) := M_{f}(u)\log u - N_{f}(u)$, for $u \in (1,t]$ and an arithmetic function $f$. By definition, we have 
\begin{equation*}
\Delta_{|g|}(u) = \sum_{n \leq u} |g(n)|\log(u/n) = \sum_{n \leq u} |g(n)|\int_n^u \frac{dv}{v} = \int_1^u \frac{M_{|g|}(v)}{v}dv. \label{EXACT}.
\end{equation*}
Now, for $u \leq t$ fixed but arbitrary and $u_1 \leq \frac{u}{\log^{K_1}u}$ for some $K_1 > 0$ to be chosen, we have
\begin{equation}
\Delta_{|g|}(u) \leq M_{|g|}(u_1)\int_1^{u_1} \frac{dv}{v} + M_{|g|}(u) \int_{u_1}^u \frac{dv}{v} \leq M_{|g|}(u)\log u \left(\frac{M_{|g|}(u_1)}{M_{|g|}(u)} + \left(1-\left(\frac{\log u_1}{\log u}\right)\right)\right). \label{DELTBOUND}
\end{equation}
By Lemma \ref{SELBERG}, it follows that 
\begin{equation*}
\frac{M_{|g|}(u_1)}{M_{|g|}(u)} \ll_{B,\delta} \frac{u_1\log^{B-1}u_1}{u}\log^{1-\delta} u \ll \log^{B-\delta-K_1} u.  
\end{equation*}
Also, $1-\frac{\log u_1}{\log u} \leq K_1\frac{\log_2 u}{\log u}$.
Hence, choosing $K_1 := B+2$, we have 
\begin{equation*}
N_{|g|}(u) \leq M_{|g|}(u)\log u\left(1+ 2\frac{\log_2 u}{\log u}\right) \leq 2M_{|g|}(u) \log u
\end{equation*}
for $u > e^{5}$. \\
Furthermore, given $y = te^{-\log^{c} t}$ for any $0 < c < \frac{1}{2}$ then by Lemma \ref{COMPAVG} we have, uniformly in $t$,
\begin{equation}
M_{|g|}(t) \gg \frac{1}{\log t}N_{|g|}(t) =  \frac{1}{y\log t}\int_{t-y}^{t} N_{|g|}(u) du +O\left(\frac{M_{|g|}(t)}{\log^3 t}\right)\geq \frac{t^2}{y\log t} \int_{t-y}^t \frac{N_{|g|}(u)}{u^2} du +O\left(\frac{M_{|g|}(t)}{\log^3 t}\right), \label{INTAVGLOW}
\end{equation}
the second estimate coming from \eqref{FIRST} with $A = 0$ and $|g|$ in place of $g$. We exploit this integral average as follows.  Applying Lemma \ref{STRONG}, we get
\begin{align*}
\int_{t-y}^t \frac{N_{|g|}(u)}{u^2} du &= \int_{t-y}^t\frac{du}{u^2}\left(\sum_{m \leq u} |g(m)|\log m\right) = \int_{t-y}^t \frac{du}{u^2} \left(\sum_{a \leq u} \Lambda(a)|g(a)| \sum_{m \leq u/a} |g(m)|\right) \\
&= \sum_{a \leq t} |g(a)|\frac{\Lambda(a)}{a}\int_{(t-y)/a}^{t/a} \frac{du}{u^2}M_{|g|}(u) \geq \delta \int_1^{t} \frac{M_{|g|}(u)}{u^2} du \left(\sum_{(t-y)/u < a \leq t/u \atop (a,P_t) = 1} \frac{\Lambda(a)}{a}\right),
\end{align*}
the last inequality following by reordering summation and integration (noting that $(t-y)/a < u \leq t/a$ if, and only if, $(t-y)/v < a \leq t/v$) and using $|g(n)|\Lambda(n) \geq \delta \Lambda(n)$.
Now, as $y = te^{-\log^{c}t}$ with $c < \frac{1}{2}$ we may apply ii) of Lemma \ref{PRECMERTENS} (noting that $\log P_te^{-\theta \sqrt{\log t}} = o_r(1)$ , so that 
\begin{equation*}
M_{|g|}(t) \gg_{B,r} \frac{\phi(P_t)}{P_t}\frac{t^2}{y \log t}\left(\int_{t-y}^t \frac{N_{|g|}(u)}{u^2} du\right)\gg \delta \frac{t}{\log t}\int_{1}^t \frac{M_{|g|}(u)}{u^2} du.
\end{equation*}
\end{proof}
\begin{rem} \label{REMSQFSUPP}
For the purposes of the proof of Theorem \ref{LOWERMV}, we require a version of Lemma \ref{DENOM} that holds when $g$ is squarefree-supported. In this case, Lemma \ref{STRONG} no longer holds, and our argument must change slightly at the juncture at which this lemma is quoted. Indeed, we have
\begin{align*}
\int_{t-y}^t \frac{N_{|g|}(u)}{u^2} du &= \int_{t-y}^t\frac{du}{u^2}\left(\sum_{m \leq u} |g(m)|\log m\right) = \int_{t-y}^t \frac{du}{u^2} \left(\sum_{a \leq u} \Lambda(a)|g(a)| \sum_{m \leq u/a \atop (m,a) = 1} |g(m)|\right) \\
&= \sum_{a \leq t} |g(a)|\frac{\Lambda(a)}{a}\int_{(t-y)/a}^{t/a} \frac{du}{u^2}\sum_{m \leq u \atop (m,a) = 1} |g(m)| \geq \delta \int_1^{t} \frac{du}{u^2} \sum_{m \leq u} |g(m)|\left(\sum_{(t-y)/u < a \leq t/u \atop (a,P_tm) = 1} \frac{\Lambda(a)}{a}\right).
\end{align*}
Proceeding as above and using the inequality $\frac{\phi(ab)}{ab} \geq \frac{\phi(a)}{a}\frac{\phi(b)}{b}$ for $a = P_t$ and $b = m$ and each $m \leq u$, it follows that
\begin{equation*}
\int_{t-y}^t \frac{N_{|g|}(u)}{u^2} du \gg \delta \frac{\phi(P_t)}{P_t}\frac{y}{t}\int_1^{t} \frac{M_{g'}(u)}{u^2} du,
\end{equation*}
where we have set $g'(m) := g(m)\phi(m)/m$. Note that this argument actually works for any multiplicative $g$.
\end{rem}
\begin{lem}\label{CHEAP}
Let $t \geq 2$ and $\sg := 1+\frac{1}{\log t}$. Then $\int_1^t \frac{M_{|g|}(u)}{u^2} du \gg_B L_{|g|}(t) \gg_B \mc{G}(\sg)$.
\end{lem}
\begin{proof}
By partial summation, we have
\begin{equation*}
\sum_{n \leq t} \frac{|g(n)|}{n} = \frac{M_{|g|}(t)}{t} + \int_1^t \frac{M_{|g|}(v)}{v^2}dv.
\end{equation*}
Now, observe that 
\begin{equation*}
t\int_{t/2}^t \frac{M_{|g|}(v)}{v^2} dv \geq \int_{t/2}^t \frac{M_{|g|}(v)}{v} dv \geq M_{|g|}(t/2)\log 2 \geq \frac{\log 2}{1+DB} M_{|g|}(t),
\end{equation*}
the last estimate following by Lemma \ref{DobApp} with $c = 2$ and some $D > 0$. Thus, 
\begin{equation*}
\sum_{n \leq t} \frac{|g(n)|}{n} \geq \frac{M_{|g|}(t)}{t} + \frac{\log 2}{1+DB} \frac{M_{|g|}(t)}{t} = \left(1+\frac{\log 2}{1+DB}\right)\frac{M_g(t)}{t}, 
\end{equation*}
and thus
\begin{equation}
\int_1^t \frac{M_{|g|}(v)}{v^2}dv \geq \left(1-\frac{1+DB}{1+DB+\log 2}\right)L_{|g|}(t). \label{LOWLg}
\end{equation}
Since $L_{|g|}(t) \gg_B \mc{G}(\sg)$ by Lemma \ref{DENOMPARS}, the claim follows.
\end{proof}
Note that the proofs of the last lemma above work equally well when $g$ is a squarefree-supported multiplicative function since Lemma \ref{DobApp} holds. With this observation and Remark \ref{REMSQFSUPP}, Theorem \ref{LOWERMV} follows immediately.
\begin{proof}[Proof of Theorem \ref{LOWERMV}]
Let $\lambda$ be a non-negative, multiplicative function satisfying $\delta \leq \lambda(p) \leq B$ for all primes $p$. Let $\lambda'(m) := \mu^2(m)\lambda(m)\phi(m)/m$. Combining Lemma \ref{CHEAP} with Remark \ref{REMSQFSUPP} gives
\begin{align*}
M_{\lambda}(t) &\geq \sum_{n \leq x} \mu^2(n)\lambda(n) \gg_B \delta \frac{\phi(P_t)}{P_t}\frac{t}{\log t} \int_1^t \frac{M_{\lambda'}(u)}{u^2} du \gg_B \delta \frac{\phi(P_t)}{P_t}\frac{t}{\log t} \sum_{n \leq t} \frac{\lambda'(n)}{n} \\
&\gg_B \delta \frac{\phi(P_t)}{P_t}\frac{t}{\log t} \prod_{p \leq t} \left(1+\frac{\lambda(p)(p-1)}{p^2}\right) \geq \delta \frac{\phi(P_t)}{P_t}\frac{t}{\log t} \prod_{p \leq t} \left(1+\frac{\lambda(p)}{p}\right)\left(1-\frac{\lambda(p)}{p^2}\right) \\
&\gg_B \delta \frac{\phi(P_t)}{P_t}\frac{t}{\log t} \prod_{p \leq t} \left(1+\frac{\lambda(p)}{p}\right),
\end{align*}
the last estimate coming from Lemma \ref{LOGSUM} and the convergence of the product $\prod_{p} \left(1-\frac{\lambda(p)}{p^2}\right)$.
\end{proof}
\subsection{Upper Bounds for $|M_h|$}
The following constitutes an analogue of Lemma \ref{DENOM} for $M_h(t)$. Our upper bound integral will instead be dealt with via Parseval's theorem in Section 6.
\begin{lem} \label{NUM}
Let $t$ be sufficiently large, let $|A| \in [0,1]$, and put $h(n):= g(n)-A|g(n)|$ and $\mu := \max_{p} \text{arg}(g(p))$. Set $R_h(\beta) := \max_{2 \leq u \leq x} \frac{|M_h(u)|}{M_{|g|}(u)} \log^{\beta} u$, for $\beta > 0$. 
Then for any $\lambda > 0$ and $\kappa \in (0,1)$,
\begin{align*}
|M_h(t)| &\ll_B B\frac{t}{\log t}(\frac{1}{(1-\kappa)\log t}\int_{t^{\kappa}}^{t} \frac{|M_h(u)|\log u}{u^{2}}du +R_h(\lambda)\left(\int_1^{t^{\kappa}} \frac{M_{|g|}(u)}{u^2\log^{\lambda}(3u)} du +M_{|g|}(t)\frac{\log_2 t}{\log^{\lambda} t}\right) \\
&+ |A|\mu\int_1^{t} \frac{M_{|g|}(u)}{u^{2}} du +\frac{M_{|g|}(t)}{\log^2 t}).
\end{align*}
\end{lem}
\begin{proof}
First, using \eqref{EXACT}, we can write
\begin{align*}
|\Delta_{h}(t)|&= \left|\int_1^t \frac{M_{h}(u)}{u}du\right| \leq \int_1^t \frac{|M_h(u)|}{u}du 
\leq R_h(\lambda)\int_1^t du\frac{M_{|g|}(u)}{u\log^{\lambda} u} du.
\end{align*}
Arguing as in \eqref{DELTBOUND}, 
\begin{equation}
\int_1^t \frac{M_{|g|}(u) }{u\log^{\lambda} u} du \ll (B+1)M_{|g|}(t)\frac{\log_2 t}{\log^{\lambda} t}, \label{REMAINDER}
\end{equation}
whence follows the estimate $|\Delta_{h}(t)| \leq (B+1)R_h(\lambda)M_{|g|}(t)\frac{\log_2 t}{\log^{\lambda} t}$. \\
We now consider $M_h(t)$. Using the above observation,
\begin{equation}
|M_h(t)| = \frac{1}{\log t} |N_h(t)-\Delta_{h}(t)| \leq \frac{1}{\log t}|N_h(t)| + (B+1)R_h(\lambda)M_{|g|}(t)\frac{\log_2 t}{\log^{1+\lambda} t}. \label{GAPMN}
\end{equation}
Now, as before, take $y = te^{-\log^{c} t}$ with $0 < c < \frac{1}{2}$. Applying Lemma \ref{COMPAVG}, we have
\begin{align}
&|N_{h}(t)| \leq y^{-1}\int_{t-y}^{t} |N_{h}(u)| du + O\left(\frac{M_{|g|}(t)}{\log^2 t}\right) \leq \frac{t^2}{y} \int_{t-y}^{t} \frac{|N_{h}(u)|}{u^2} du + O\left(\frac{M_{|g|}(t)}{\log^2 t}\right) \nonumber\\
&\leq \frac{t^2}{y}\left(\sum_{a \leq t} \Lambda(a)|g(a)|\int_{t-y}^{t} \frac{du}{u^2} |M_h(u/a)| + |A|\mu \sum_{a \leq t} \Lambda(a)|g(a)|\int_{t-y}^t \frac{du}{u^2} M_{|g|}(u/a)\right) + O\left(\frac{M_{|g|}(t)}{\log^2 t}\right) \nonumber\\
&=: \frac{t^2}{y}(I_1+|A|\mu I_2) + O\left(\frac{M_{|g|}(t)}{\log^2 t}\right) \label{SETUP1}
\end{align}
Let $Y := t^{\kappa}$. We divide the sum over $a$ in $I_1$ into the segments $[1,Y]$ and $[Y,t]$. Over the first segment, we make the substitution $v := u/a$ and insert a logarithmic factor, giving
\begin{align}
&\sum_{a \leq Y} \Lambda(a)|g(a)| \int_{t-y}^t \frac{du}{u^2}|M_h(u/a)| \leq B\sum_{a \leq Y} \frac{\Lambda(a)}{a}\int_{(t-y)/a}^{t/a} \frac{du}{u^2}|M_h(u)| \label{COMPGEN}\\
&\leq B\sum_{a \leq Y} \frac{\Lambda(a)}{a\log((t-y)/a)}\int_{(t-y)/a}^{t/a} \frac{du}{u^2}|M_h(u)|\log u \nonumber\\
&\ll B\frac{1}{\log t}\int_{(t-y)/Y}^{t} \frac{du}{u^2}|M_h(u)|\log u \sum_{(t-y)/v < a \leq t/v} \frac{\Lambda(a)}{a}\left(1-\frac{\log a}{\log t}\right)^{-1} \nonumber\\
&\leq B\frac{1}{(1-\kappa)\log t}\int_{(t-y)/Y}^{t} \frac{du}{u^2}|M_h(u)|\log u \sum_{(t-y)/v < a \leq t/v} \frac{\Lambda(a)}{a} \nonumber\\
&\ll B\frac{y}{(1-\kappa)t\log t} \int_{t^{\kappa}}^t \frac{du}{u^2}|M_h(u)| \log u, \nonumber
\end{align}
the last estimate following from Lemma \ref{PRECMERTENS}. Over the second segment, we simply have
\begin{align*}
&\sum_{Y< a \leq t} \Lambda(a)|g(a)| \int_{t-y}^t \frac{du}{u^2}|M_h(u/a)| \leq B\sum_{Y < a \leq t} \frac{\Lambda(a)}{a}\int_{(t-y)/a}^{t/a} \frac{du}{u^2}|M_h(u)| \\
&= B\int_1^Y \frac{du}{u^2}|M_h(u)| \sum_{(t-y)/v < n \leq t/v} \frac{\Lambda(a)}{a} \ll B\frac{y}{t}\int_1^Y \frac{du}{u^2}|M_h(u)|.
\end{align*}
We estimate $I_2$ in the same as with the second segment above, thus
\begin{align*}
I_2 &\leq B\sum_{a \leq t}\frac{\Lambda(a)}{a} \int_{(t-y)/a}^{t/a} \frac{dv}{v^2}M_{|g|}(v) = B\int_{1}^{t} \frac{dv}{v^2}M_{|g|}(v)\left(\sum_{(t-y)/v<a \leq t/v} \frac{\Lambda(a)}{a}\right) \asymp B \frac{y}{t}\int_1^{t} \frac{dv}{v^2}M_{|g|}(v).
\end{align*}
%
It follows that
\begin{align}
\int_{t-y}^t \frac{|N_{h}(u)|}{u^2} du &\ll B \frac{y}{t}(\frac{1}{(1-\kappa)\log t}\int_{t^{\kappa}}^{t} \frac{|M_{h}(u)|\log u}{u^2} du+R_h(\lambda)\int_1^{t^{\kappa}} \frac{M_{|g|}(u)}{u^2\log^{\lambda}(3u)}du \nonumber\\
&+ |A|\mu \int_1^{t} \frac{M_{|g|}(u)}{u^2}du) \label{SETUP2},
\end{align}
which, when combined with \eqref{GAPMN}, gives the claim.
\end{proof}
\begin{rem}\label{NUMCOMP}
The above argument can be modified to deal with completely multiplicative functions $g$ as well, provided we take $B < 2$. The lone emendment must be made in the inequality in \eqref{COMPGEN} and the corresponding estimate on the interval $[Y,t]$ (which proceeds similarly), and these modifications can be made in the following way. Splitting the sum into pieces with $a$ prime and $a$ a prime power with multiplicity at least 2, and applying Cauchy-Schwarz to the second, we have
\begin{align*}
&\sum_{a \leq Y} \frac{\Lambda(a)}{a}|g(a)|\int_{(t-y)/a}^{t/a} \frac{du}{u^2} |M_h(u)| \leq B\sum_{p \leq y} \frac{\Lambda(p)}{p}\int_{(t-y)/p}^{t/p} \frac{du}{u^2}|M_h(u)| + \sum_{p^k \leq y \atop k \geq 2} \Lambda(p) \left(\frac{B}{p}\right)^k \int_{(t-y)/p^k}^{t/p^k} \frac{du}{u^2}|M_h(u)| \\
&\leq B\sum_{a \leq Y} \frac{\Lambda(a)}{a}\int_{(t-y)/a}^{t/a} \frac{du}{u^2}|M_h(u)| + \left(\sum_{p, k\geq 2} \log^2 p \left(\frac{B}{p}\right)^k\right)^{\frac{1}{2}} \left(\sum_{p^k \leq Y \atop  k \geq 2} \frac{\Lambda(p)^2}{p^k}B^k \left(\int_{(t-y)/p^k}^{t/p^k} \frac{du}{u^2}|M_h(u)|\right)^2 \right)^{\frac{1}{2}} \\
&\leq B\sum_{a \leq Y} \frac{\Lambda(a)}{a}\int_{(t-y)/a}^{t/a} \frac{du}{u^2}|M_h(u)| + 2B^2\left(\sum_{p} \frac{\log^2 p}{p(p-B)}\right)^{\frac{1}{2}} \sum_{a \leq Y} \frac{\Lambda(a)}{a}\left(\int_{(t-y)/a}^{t/a} \frac{du}{u^2}|M_h(u)|\right),
\end{align*}
where, in the last inequality, we used the fact that if $k/2$ is a half-integer (necessarily strictly greater than 1) then the corresponding term in $p^{k}$ is smaller than that for $p^{2\llf k/2\rrf}$ (since the integral is longer in the latter case, and $(B/p)^k \leq (B/p)^{2\llf k/2 \rrf}$ for all $p$). 
\end{rem}
\begin{rem}\label{REMGEN}
It should be noted that all of the arguments in this section remain valid if we replace the function $h(n) := g(n)-A|g(n)|$ with the function $h(n) := a(n)-Ab(n)$, where $a$ and $b$ are real-valued, strongly or completely multiplicative functions such that $|a(n)| \leq b(n)$, $\delta \leq |a(p)| \leq b(p) \leq B$ (with $B < 2$ in case $g$ is completely multiplicative) and $\left|\frac{a(p)}{b(p)} -1 \right| \leq \eta$. This will be of value in the proof of ii) in Theorem \ref{WIRSINGEXT}.
\end{rem}
\section{Analytic Estimates}
In this section, we collect estimates for several line integrals related to the integral bounds in the previous section.  Throughout, for $\sg > 1$ the operator $\int_{(\sg)}$ denotes integration on the line $\{s \in \mb{C} : \text{Re}(s) = \sg\}$. \\
On occasion, we will find it convenient to estimate $G(s)$ directly in terms of the value of $g$ at primes.  To this end, we prove the following simple result. It appears in \cite{Ell} in a different form, so we prove it in our context for completeness.
\begin{lem} \label{StrongComp}
Let $g : \mb{N} \ra \mb{C}$ be a strongly multiplicative function for which there is a $B \in (0,\infty)$ such that $|g(p)| \leq B$ for all primes $p$, and let $G$ be its Dirichlet series, valid for $\text{Re}(s) > 1$. Then we can write
\begin{equation}
G(s) = G_0(s)\exp\left(\sum_p g(p)p^{-s}\right), \label{USESTRONG}
\end{equation}
where $G_0(s)$ is a Dirichlet series converging absolutely and uniformly in the half-plane $\sg > \frac{1}{2}$, and such that, uniformly in $\tau \in \mb{R}$ and $\sg > 1$, $|G_0(s)| \leq 2e^{B(B+1)}\log(1+B)^{4B}$. \\
\end{lem}
\begin{rem}
In several lemmata to follow, we employ a factorization of the type $A'(s) = \frac{A'(s)}{A(s)}\cdot A(s)$, where $A = G$ or $\mc{G}$. Since we can always exclude measure-zero sets from our integrals, we may ignore each case in which $G(s)$ and/or $\mc{G}(s)$ is non-zero, since $G$ and $\mc{G}$ are holomorphic on the half-plane $\text{Re}(s) > 1$, and thus their zero sets are discrete.  It will, of course, still be necessary to show that the resulting factored expression is still bounded in neighbourhoods of these zeros, and we shall do this in what follows.
\end{rem}
\begin{proof}
Set $G_0(s) := G(s)\exp\left(-\sum_p g(p)p^{-s}\right)$ and estimate $G_0$.  First, for $\sg > 1$ and $s = \sg + i\tau$ fixed, we write 
\begin{equation*}
A(s) := \prod_{p \leq B^{\frac{1}{\sg}} + 1} \left(1+\frac{g(p)}{p^s-1}\right)\exp\left(-g(p)p^{-s}\right), 
\end{equation*}
so that by definition,
\begin{equation*}
G_0(s) = A(s) \left(\prod_{p > B^{\frac{1}{\sg}}+1} \left(1+\frac{g(p)}{p^s-1}\right)e^{-g(p)p^{-s}}\right) = A(s) \exp\left(\sum_{p > B^{\frac{1}{\sg}}+1} \left(\log\left(1+\frac{g(p)}{p^s-1}\right) - g(p)p^{-s}\right)\right).
\end{equation*}
In particular, when $p > B^{\frac{1}{\sg}} + 1$ we have $|\frac{g(p)}{p^s-1}| \leq \frac{|g(p)|}{p^{\sg}-1} < 1$, whence, Taylor expanding the factor $\log\left(1+\frac{g(p)}{p^s-1}\right)$ for each $p$, we get
\begin{align}
\log\left(1+\frac{g(p)}{p^s-1}\right) - g(p)p^{-s} &= \left(\log\left(1+\frac{g(p)}{p^s-1}\right) - \frac{g(p)}{p^{s}-1}\right) + \left(\frac{g(p)}{p^s-1} - \frac{g(p)}{p^s}\right) \label{Str}\\
&= -\frac{g(p)^2}{(p^{s}-1)^2}\sum_{l \geq 2} (-1)^{l} \frac{1}{l}\left(\frac{g(p)}{p^s-1}\right)^{l-2} + \frac{g(p)}{p^s(p^s-1)} \nonumber.
\end{align}
Summing all of these terms and estimating the result gives
\begin{align*}
\left|\sum_{p > B^{\frac{1}{\sg}}+1} \left(\log\left(1+\frac{g(p)}{p^s-1}\right) - g(p)p^{-s}\right)\right| &\leq B(B+1)\sum_{p > B^{\frac{1}{\sg}}+1} \frac{1}{p^{\sg}(p^{\sg}-1)}.
\end{align*}
This estimate implies immediately that the second factor in the definition of $G_0(s)$ converges absolutely and uniformly for $\sg > \frac{1}{2}$. Clearly, as $A(s)$ is a finite product of functions that are holomorphic in the half-plane $\text{Re}(s) > 0$, $G_0(s)$ must also converge absolutely and uniformly in the half-plane $\sg > \frac{1}{2}$, and thus be holomorphic there as well. Moreover, for $\sg > 1$, $\sum_p \frac{1}{p^{\sg}(p^{\sg}-1)} \leq \sum_{n \geq 2} \frac{1}{n(n-1)} = 1$, so that
$\left|G_0(s)\right| \leq e^{B(B+1)}|A(s)|$.  In the expression $A(s)$, we can simply bound trivially via 
\begin{equation*}
\left|1+\frac{g(p)}{p^s-1}\right| \leq 1+\frac{B}{p^{\sg}-1}\leq 1+\frac{2B}{p^{\sg}}, 
\end{equation*}
which gives
\begin{equation*}
|A(s)| \leq \exp\left(2B\sum_{p \leq B^{\frac{1}{\sg}}+1} \frac{1}{p^{\sg}}\right) \leq 2\log(1+B)^{4B}.
\end{equation*}
Combining these two estimates implies the claim. 
\end{proof}
\begin{rem}
A similar argument holds for $g$ completely multiplicative, with $B < 2$. While the form of $G_0(s)$ differs, the analytic estimates are essentially identical (dealing with factors $(1-g(p)/p^s)^{-1}$ instead of $1+g(p)/(p^s-1)$).
\end{rem}
It will be helpful to treat integrals in $G(s)$ in terms of those same integrals in $\mc{G}(s)$ without losing the effect of $\tau$ in the integral, since the latter has non-negative coefficients.  The following tool is helpful in this regard.
\begin{lem}\label{MODMONT}
Let $A(s) := \sum_{n \geq 1} a_nn^{-s}$ be a Dirichlet series such that there exists a non-negative real sequence $\{b_n\}_n$ such that $|a_n| \leq b_n$ for all $n \in \mb{N}$. Furthermore, let $B(s) := \sum_{n \geq 1} b_nn^{-s}$ and suppose it converges absolutely and uniformly in the half-plane $\sg > \sg_0$. 
Then for any $T \geq 0$ and $\sg > \sg_0$,
\begin{equation*}
\int_{-T}^T |A(\sg + i\tau)|^2 |ds| \leq 3\int_{-T}^{T} |B(\sg+i\tau)|^2 |ds|.
\end{equation*}
\end{lem}
\begin{proof}
This result, due essentially to Montgomery, is Lemma 6.1 in III.4 of \cite{Ten2}.  
\end{proof}
In Lemma \ref{SHARPH} below we will need to know that $G_0(s)$ and $\mc{G}_0(s)$ are non-zero for $|\tau|$ small on the line $\text{Re}(s) = \sg$. To this end, we have the following.
\begin{lem} \label{ZEROLEM}
Let $B > 0$. Suppose $g$ is strongly multiplicative satisfying $|g(p)| \leq B$ for each $p$. Suppose, moreover, that $|\text{arg}(g(p))| < 1$ for each $p$. If $B \leq 1$ then $G(\sg + i\tau) \neq 0$ for every $\sg > 1$, $\tau \in \R$. If $B > 1$ and $G(\sg + i\tau) = 0$ then $|\tau| \geq \frac{\sg}{\log B}$, and the same is true of the Dirichlet series $\mc{G}(s) := \sum_{n \geq 1} \frac{|g(n)|}{n^s}$.
\end{lem}
\begin{proof}
As $G(s) = \prod_p \left(1+\frac{g(p)}{p^s-1}\right)$, this can only be zero if there exists some $p$ such that $1-g(p) = p^s$, i.e., $p^{\sg} \leq |g(p)|+1 \leq B+1$; it thus suffices to consider $p < B^{\frac{1}{\sg}} + 1$. Clearly, when $B\leq 1$ this is impossible. Thus, consider $B > 1$, and suppose that $|\tau| < \frac{\sg}{\log B}$. Then, for $p$ such that $g(p) = 1-p^s$, we must have 
\begin{equation*}
1-\text{Re}(g(p)) = p^{\sg}\cos(|\tau| \log p) \geq p^{\sg}\cos\left(\frac{\sg \log p}{\log B}\right) \geq \frac{1}{\sqrt{2}}p^{\sg}. 
\end{equation*}
Clearly, since $|\arg(g(p))| < 1$, $\text{Re}(g(p)) > 0$, which implies that $\frac{1}{\sqrt{2}} \leq \frac{1}{p}$ for some prime $p$, a clear contradiction.  The first claim follows. 
The second one is immediate by applying the first claim to $|g|$.
\end{proof}
\begin{rem}
When $g$ is completely multiplicative and $|g(p)| < 2$ uniformly, $1-g(p)/p^s$ can never be zero and always has bounded norm, so no such lemma is needed in this case.
\end{rem}
In the proof of part i) of Theorem \ref{HalGen}, we will need to compute the ratio of $\int_{(\sg)} |G'(s)|^2 \frac{|ds|}{|s|^2}$ with $\mc{G}(\sg)^2$, coming from Lemma \ref{CHEAP}. An essential ingredient in this computation is an estimate in $\tau$, uniform over a large range, for the ratio of $\left|\frac{G(s)}{\mc{G}(\sg)}\right|$. This provides the main term in the estimate in \eqref{UPPERGen} and \eqref{UPPER}. We establish an estimate for this ratio in the following two lemmata. The first is geared towards \eqref{UPPERGen}, the second towards \eqref{UPPER}.  \\
We recall that $\tilde{g}$ is defined on primes by $\tilde{g}(p) := g(p)/B_j$ in case $|g(p)| \leq B_j$ whenever $p \in E_j$, and that
\begin{equation*}
\rho_{E_j}(x;\tilde{g},T) := \min_{|\tau| \leq T} D_{E_j}(\tilde{g},p^{i\tau};x) := \min_{|\tau| \leq T} \left(\sum_{p \in E_j \atop p \leq u} \frac{1-\text{Re}\left(\tilde{g}(p)p^{-i\tau}\right)}{p}\right),
\end{equation*}
for $T > 0$ and $x \geq 2$.
\begin{lem}\label{EASYHalDecay}
Let $\sg := 1+ \frac{1}{\log t}$, and suppose $g \in \mc{C}$. For $D > 1$, $u \geq 2$ and $T \ll (\sg-1)^{-D}$, we have, uniformly in $|\tau| \leq T$,
\begin{equation*}
\left|\frac{G(\sg + i\tau)}{\mc{G}(\sg)}\right| \ll 
\exp\left(-\sum_{1 \leq j \leq m} B_j\left(\rho_{E_j}(x;\tilde{g},t) + \sum_{p \in E_j \atop p \leq x} \frac{1 -|\tilde{g}(p)|}{p}\right)\right).
\end{equation*}
\end{lem}
\begin{proof}
By Lemma \ref{StrongComp}, whenever $s$ is not a zero of $G$ we have 
\begin{equation}
\left|\frac{G(s)}{\mc{G}(\sg)}\right| \ll_B \exp\left(\sum_p\left(\frac{g(p)}{p^s}-\frac{|g(p)|}{p^{\sg}}\right)\right) = \exp\left(-\sum_p \frac{|g(p)|-\text{Re}(g(p)p^{-i\tau}}{p^{\sg}}\right). \label{ONE}
\end{equation}
By \eqref{TAIL} and \eqref{HEAD}, we have
\begin{align*}
\sum_{p} \frac{|g(p)|-\text{Re}(g(p)p^{-i\tau})}{p^{\sg}} = \sum_{p \leq x} \frac{|g(p)|-\text{Re}(g(p)p^{-i\tau})}{p} + O(B).
\end{align*}
Decomposing the sum over primes into the sums over each set $E_j$ in the partition, we have
\begin{align*}
\sum_{p \leq x} \frac{|g(p)|-\text{Re}(g(p)p^{-i\tau}}{p} &= \sum_{1 \leq j \leq m} \left(\sum_{p \in E_j \atop p \leq x} \frac{|g(p)|-B_j}{p} + B_j\sum_{p \in E_j \atop p \leq x} \frac{1-\text{Re}(g(p)p^{-i\tau}/B)}{p}\right) \\
&= \sum_{1 \leq j \leq m} \left(B_j\rho_{E_j}(g;x) - \sum_{p \in E_j \atop p \leq x} \frac{|g(p)|-B_j}{p}\right).
\end{align*}
The proof of the claim follows upon reinserting this expression into \eqref{ONE} and using the definition of $\tilde{g}$.
\end{proof}
Define $\mc{G}_j(s) := \prod_{p \in E_j} \left(1+\frac{|g(p)|}{p^{s}-1}\right)$ for each $1 \leq j \leq m$, which is well-defined and non-zero for $\sg > 1$ provided that $\mc{G}(s) \neq 0$ (indeed, since each $\mc{G}_j(s)$ converges absolutely and uniformly on compact subsets of $\text{Re}(s) > 1$, they are all pole-free, so any zero of $\mc{G}(s)$ corresponds to a zero of (at least) one of the factors $\mc{G}_j$), and define the functions $G_j(s)$ analogously. Note in particular that $\prod_{1 \leq j \leq m} \mc{G}_j(s) = \mc{G}(s)$ by the partition property. Also, set $\beta := \min_{1 \leq j \leq m} \beta_j$. 
\begin{lem}\label{HalDecay}
Let $t \geq 3$, $\sg = 1 + \frac{1}{\log t}$ and suppose $\mc{G}$ is non-zero on the line $\text{Re}(s) = \sg$. Then for any $s  = \sg + i\tau$,
\begin{equation*}
\left|\frac{G(s)}{\mc{G}(\sg)}\right| \ll_{m,B,\mbf{\phi},\mbf{\beta}} \left(\prod_{1 \leq j \leq m} \left|\frac{\mc{G}_j(s)}{\mc{G}_j(\sg)}\right|^{\frac{\delta_j\beta_j^3}{B_j}}\right)^{\frac{27}{1024\pi}}.
\end{equation*}
\end{lem}
\begin{rem} \label{REMDECAY}
Note that the trivial bound on $|g(p)|$ shows, using this estimate and Lemma \ref{StrongComp}, that
\begin{equation*}
\left|\frac{G(s)}{\mc{G}(\sg)}\right| \ll_{m,B,\mbf{\phi},\mbf{\beta}} \exp\left(-\frac{27}{1024\pi}\beta^3\delta\sum_p \frac{1}{p^{\sg}}\left(1-\cos\tau \log p\right)\right) \ll_B \left|\frac{\zeta(s)}{\zeta(\sg)}\right|^{\frac{27}{1024\pi}\delta\beta^3}.
\end{equation*}
This recovers (up to a constant factor) the coefficient for the exponential in \cite{Ell2}, as mentioned in Section 2. When $|\tau| \leq 2$, the Laurent series expansion of $\zeta$ around $s = 1$ gives
\begin{equation*}
\left|\frac{\zeta(s)}{\zeta(\sg)}\right| \ll \left|\frac{\sg-1}{\sg-1+i\tau}\right| \leq  \left|1+i\frac{\tau}{\sg-1}\right|^{-1} \leq \left(1+\left(\frac{|\tau|}{\sg-1}\right)^2\right)^{-\frac{1}{2}} =  e^{-\frac{1}{2}\log\left(1+\left(\frac{|\tau|}{\sg-1}\right)^2\right)} \ll e^{-\log\left(1+\frac{|\tau|}{\sg-1}\right)},
\end{equation*}
while if $2 < |\tau| < (\sg-1)^{-D}$ for $D > 2$, using $|\zeta(s)| \ll \log |\tau|$ (see, for instance, Theorem 7 in Chapter II.3 of \cite{Ten2})
\begin{equation*}
\left|\frac{\zeta(s)}{\zeta(\sg)}\right| \ll (\sg-1)\log|\tau| \leq \exp\left(-\frac{1}{2}\log\left(1+\frac{1}{\sg-1}\right)\right),
\end{equation*}
for $\sg$ sufficiently close to 1. Thus,
\begin{equation*}
\left|\frac{G(s)}{\mc{G}(\sg)}\right| \ll_{m,B,\mbf{\phi},\mbf{\beta}}  \begin{cases} \exp\left(-\frac{27}{1024 \pi} \log\left(1+\frac{|\tau|}{\sg-1}\right)\right) &\text{ if $|\tau| \leq 2$} \\ \exp\left(-\frac{27}{1024\pi}\log\left(1+\frac{1}{\sg-1}\right)\right) &\text{ if $2 < |\tau| \leq (\sg-1)^{-D}$}. \end{cases}
\end{equation*}
We will use this estimate repeatedly in this section and the next.
\end{rem}
\begin{proof}
We begin with an idea of Hal\'{a}sz. By Lemma \ref{StrongComp}, whenever $s$ is not a zero of $G$ we have 
\begin{equation*}
\left|\frac{G(s)}{\mc{G}(\sg)}\right| \ll_B \exp\left(\sum_p\left(\frac{g(p)}{p^s}-\frac{|g(p)|}{p^{\sg}}\right)\right) = \exp\left(-\sum_p \frac{|g(p)|}{p^{\sg}}(1-\cos(\xi_p))\right), 
\end{equation*}
where, for $\theta_p := \text{arg}(g(p))$, put $\xi_p := \theta_p - \tau \log p$.
Let $c \in (0,1)$ be a constant to be chosen. For each $j$, if we restrict to those primes for which $|p^{i\tau} - e^{i\phi_j}| \leq c\beta_j$ then, in light of the condition $|e^{i\theta_p}-e^{i\phi_j}| = |\theta_p-\phi_j| \geq \beta_j$, it follows that
\begin{equation*}
|\xi_p| = |1-e^{i\xi_p}| = |p^{i\tau}-e^{i\theta_p}| \geq |e^{i\theta_p}-e^{i\phi_j}| - |e^{i\phi_j}-p^{i\tau}| \geq (1-c)\beta_j.
\end{equation*}
We note that the collection of $p$ for which this condition holds is non-empty by our assumption on $\rho_j(x;gT)$. In particular, $1-\cos \xi_p \geq \frac{1}{2}(1-c)^2\beta_j^2$ whenever $p \in E_j$.  \\
Thus, we will choose functions $h_j:\mb{T} \ra [0,1]$ such that $h_j(e^{i\theta}) \leq \frac{1}{2}(1-c)^2\beta_j^2$ uniformly when $|\theta-\phi_j| \leq c\beta_j$. Specifically, we let $h_j(e^{i\theta}) = \frac{1}{2}(1-c)^2\beta_j$ for each $|\theta-\phi_j| \leq c(1-c)\beta_j$, $h(e^{i\theta}) = 0$ for $|\theta - \phi_j| \geq c\beta_j$, and interpolate in the intervals $[\phi_j-c\beta_j,\phi_j-c(1-c)\beta_j]$ and $[\phi_j+c(1-c)\beta_j,\phi_j+c\beta_j]$ with twice-continuously differentiable functions. Since $h_j \in L^2(\mb{T})$ for each $j$, we write $h_j(e^{i\theta}) = \sum_{n \in \mb{Z}} a_{j,n}e^{in\theta}$, with $a_{j,n} := \frac{1}{2\pi}\int_{-\pi}^{\pi} h_j(e^{i\theta})e^{-in\theta} d\theta$. Since $h_j \in \mc{C}^2(\mb{T})$, integration by parts, periodicity of $h_j$ and the extreme value theorem imply that for $l \neq 0$,
\begin{equation*}
2\pi a_{j,l} = \frac{1}{l}\int_{-\pi}^{\pi} h_j'(e^{i\theta})e^{-il\theta} d\theta = \frac{1}{l^2}\int_{-\pi}^{\pi} h_j''(e^{i\theta})e^{-il\theta} d\theta \ll \frac{1}{l^2+1}.
\end{equation*}
We will use this condition later. \\
It follows that, for $s_n := \sg -in\tau$,
\begin{align} 
\sum_{p\in E_j} \frac{|g(p)|}{p^{\sg}}(1-\cos \xi_p) &\geq \sum_{p\in E_j} \frac{|g(p)|}{p^{\sg}}h_j(p^{i\tau}) = \sum_{p \in E_j} \frac{|g(p)|}{p^{\sg}} \sum_{n \in \mb{Z}} a_{j,n}p^{in\tau} \nonumber\\
&= \sum_{n \in \mb{Z}} a_{j,n}\sum_{p \in E_j} |g(p)|p^{-s_n}, \label{RATBOUND}
\end{align}
this last interchange justified by the convergence of $\sum_{p \in E_j} \frac{1}{p^{\sg}}$ for $\sg > 1$ fixed. Since $\sum_{n \in \mb{Z}} a_{j,n} \ll 1$, it follows by Lemma \ref{StrongComp} that
\begin{equation} \label{FOURIERDECOMP}
\sum_{n \in \mb{Z}} a_{j,n}\sum_{p\in E_j} |g(p)|p^{-s_n} = \sum_{n \in \mb{Z}} a_{j,n} \log \mc{G}_j(s_n) + O_B(1).
\end{equation}
Observe that when $n = 0$, 
\begin{equation*}
a_{j,0} = \frac{1}{2\pi} \int_{-\pi}^{\pi} h_j(e^{i\theta}) d\theta \geq \frac{1}{4\pi}(1-c)^2\beta_j^2\int_{\theta_0-c(1-c)\beta_j}^{\theta_0+c(1-c)\beta_j} d\theta = \frac{1}{2\pi}c(1-c)^3\beta_j^3. 
\end{equation*}
Choosing the optimal $c = \frac{1}{4}$ gives $a_{j,0} = \frac{27}{512\pi} \geq \frac{1}{10\pi}$ as the constant in front of $\beta_j^3$. (Note that replacing $(1-c)^3$ with any higher power of $(1-c)$ gives an optimal value when $c$ is small, so the resulting constant here is most advantageous.)
Of course, $h_j$ being non-negative, we can bound each term $|g(p)| \geq \delta_j$. Now, observe that since $\log u = \log|u| + O(1)$ for any non-zero complex number (and an appropriate branch of logarithm) and any $k \in \mb{N}$,
\begin{align*}
&\sum_{n \in \mb{Z}} a_{j,n} \log \mc{G}_{j}(s_n) \geq \delta_j\left(a_{j,0}\log\zeta_{E_j}(s_0) + \sum_{n \neq 0} a_{j,n}\log \zeta_{E_j}(s_n)\right) \\
&=\delta_j\left(a_{j,0}\log\zeta_{E_j}(s_0) + \sum_{n \neq 0} a_{j,n}\log \left(\zeta_{E_j}(s_n)/\zeta_{E_j}(s_k)\right) + \sum_{n \neq 0} a_{j,n} \log\zeta_{E_j}(s_k)\right)\\
&= \delta_j\left(a_{j,0}\log |\zeta_{E_j}(s_0)| + \sum_{n \neq 0} a_{j,n}\log \left|\zeta_{E_j}(s_n)/\zeta_{E_j}(s_k)\right| + (h_j(1)-a_{j,0})\log|\zeta_{E_j}(s_k)| + O_{\phi_j,\beta_j}(1)\right) \\
&=\delta_j\left(a_{j,0}\log\left|\zeta_{E_j}(s_0)/\zeta_{E_j}(s_k)\right| + \sum_{n \neq 0}  a_{j,n}\log \left|\zeta_{E_j}(s_n)/\zeta_{E_j}(s_k)\right| + h_j(1)\log|\zeta_{E_j}(s_k)| +O_{\phi_j,\beta_j}(1)\right)
\end{align*}
Since $E_j$ is good, there is some $\lambda_j \in (0,1)$ such that if $f_j(s) := \log\zeta_{E_j}(s)-\lambda_j \log\left(\frac{1}{s-1}\right)$ then $f_j$ is holomorphic in a neighbourhood of $\sg = 1$. Moreover, $f_j$ extends by convergence to $|\tau| \leq 2$ with $\sg > 1$.  Now, if $|n\tau|,|k\tau| \leq 2$, we have
\begin{equation*}
\log\left|\zeta_{E_j}(s_n)/\zeta_{E_j}(s_k)\right| = \lambda_j\log\left|\frac{s_k-1}{s_n-1}\right| + O(1) \leq \log\left(1+|n|/|k|\right) \leq \log(1+n).
\end{equation*}
If at least one of $|n\tau|$ or $|k\tau|$ exceeds 2 but $|\tau| \leq 2$, upon noting that $\zeta_{E_j}$ is non-zero in zero-free regions of $\zeta$ and applying Theorem 16 in II.3 of \cite{Ten2}, we have
\begin{equation*}
\log\left|\zeta_{E_j}(s_n)/\zeta_{E_j}(s_k)\right| \leq 2\log_2\max\{|k\tau|,|n\tau|\} \leq 2\log_2(1+\max\{2k,2n\}).
\end{equation*}
Hence, when $|\tau| \leq 2$, we have
\begin{align*}
&a_{j,0}\log\left|\zeta_{E_j}(s_0)/\zeta_{E_j}(s_k)\right| + \sum_{n \neq 0}  a_{j,n}\log \left|\zeta_{E_j}(s_n)/\zeta_{E_j}(s_k)\right| + h_j(1)\log|\zeta_{E_j}(s_k)| +O_{\phi_j,\beta_j}(1) \\
&= a_{j,0}\log\left|\zeta_{E_j}(s_0)/\zeta_{E_j}(s_k)\right| + O\left(k \log_2(2+k) + \sum_{n \geq 1} \frac{\log (1+n)}{1+n^2}\right)  = a_{j,0} \log\left|\zeta_{E_j}(s_0)/\zeta_{E_j}(s_k)\right| + O(k\log_2(2+k)),
\end{align*}
where we used $h_j(1)\log\left|\zeta_{E_j}(s_j)\right| = O\left(\log(1+2k)\right)$. On the other hand, if $|\tau| \geq 2$ then, again by this last argument,
\begin{equation*}
\sum_{n \neq 0} a_{j_n}\log\left|\zeta_{E_j}(s_n)\right| \leq \sum_{n \neq 0} a_{j,n}\log_2 |n\tau| = O\left(\log_2(1+|\tau|)\right).
\end{equation*}
Hence, using this same bound for $\log\left|\zeta_{E_j}(s_k)\right|$, it follows that in all cases,
\begin{equation*}
\sum_{n \in \mb{Z}} a_{j,n} \log \mc{G}_{j}(s_n) \geq \delta_ja_{j,0}\log\left|\zeta_{E_j}(\sg)/\zeta_{E_j}(s_k)\right| + O_k\left(\log_2 |\tau| + 1\right) \geq \frac{\delta_j}{B_j} \log\left|\mc{G}_j(\sg)/\mc{G}_j(s_k)\right| + O_k\left(\log_2(1+|\tau|)\right).
\end{equation*}
Note that this bound is non-trivial because $|\tau| \leq \log^D x$ and $D_{E_j}(g/|g|,p^{i\tau},\sg) \gg \frac{B_j}{\delta_j}\log_3 x$, with some sufficiently large implicit constant by assumption. At any rate, replacing $a_{j,0}$ by $a_{j,0}/2$ to compensate, it follows that
\begin{align*}
\left|\frac{G_j(s)}{\mc{G}_j(\sg)}\right| &\ll_{B, \phi_j,\beta_j,k} \exp\left(-\frac{\delta_j}{2B_j}a_{j,0}\sum_{p \in E_j} |g(p)|\left(\frac{1}{p^{\sg}}-\frac{1}{p^{s_k}}\right)\right) \ll_B \left|\frac{\mc{G}_j(s_k)}{\mc{G}(\sg)}\right|^{\frac{\delta_ja_{j,0}}{2B_j}} = \left|\frac{\mc{G}_j(s_k)}{\mc{G}(\sg)}\right|^{\frac{27\delta_j}{1024 \pi B_j}\beta_j^3},
\end{align*}
the second last estimate applying Lemma \ref{StrongComp} again and the last resulting from our earlier choice of $c$.
\end{proof}
Set $\gamma_{0,j} := \frac{27\delta_j}{1024\pi B_j}\beta_j^3$, and $\gamma_0 := \min_{1 \leq j \leq m} \gamma_{0,j}$. The major contribution to the upper bound in Theorem \ref{HalGen} comes from establishing a bound for the maximum of $\left|\frac{G(s)}{\mc{G}(\sg)}\right|$ for $|\tau| \leq T$. We deal with this as follows.
\begin{lem} \label{INTBOUND}
Let $\sg = 1+\frac{1}{\log t}$ for $t \geq 2$ sufficiently large. Then for each $\tau \in \mb{R}$ and each $j$,
\begin{equation*}
\left|\frac{G_j(s)}{\mc{G}_j(\sg)}\right| \ll_{B,\phi_j,\beta_j} \left|\frac{G_j(\sg)}{\mc{G}_j(\sg)}\right|^{\frac{\gamma_{0,j}}{2(1+\gamma_{0,j})}}. 
\end{equation*}
\end{lem}
As a corollary, we have the weaker bound 
\begin{equation*}
\left|\frac{G(s)}{\mc{G}(\sg)}\right| \ll_{B,m,\mbf{\phi},\mbf{\beta}} \left|\frac{G(\sg)}{\mc{G}(\sg)}\right|^{\frac{\gamma_0}{2(1+\gamma_0)}}.
\end{equation*}
\begin{proof}
Applying Lemma \ref{TrigProp} with $k = 1$, observe that for each $j$,
\begin{equation*}
\sum_{p \in E_j} \frac{|g(p)|}{p^{\sg}}(1-\cos \theta_p) \leq 2\left(\sum_{p \in E_j} \frac{|g(p)|}{p^{\sg}}(1-\cos (\theta_p-\tau\log p)) + \sum_{p \in E_j} \frac{|g(p)|}{p^{\sg}}(1-\cos \tau \log p)\right),
\end{equation*}
with $\theta_p := \text{arg}(g(p))$ as above. Combining this with Lemma \ref{StrongComp} gives
\begin{align*}
\left|\frac{G_j(\sg)}{\mc{G}_j(\sg)}\right| &\gg_B \exp\left(-\sum_{p\in E_j} \frac{|g(p)|}{p^{\sg}}(1-\cos \theta_p)\right) \\
&\geq \exp\left(-2\sum_{p \in E_j} \frac{|g(p)|}{p^{\sg}}(1-\cos (\theta_p-\tau\log p))\right) \cdot  \exp\left(-2\sum_{p \in E_j} \frac{|g(p)|}{p^{\sg}}(1-\cos \tau \log p)\right)\\
&\gg_B \left(\left|\frac{G_j(s)}{\mc{G}_j(\sg)}\right|\left|\frac{\mc{G}_j(s)}{\mc{G}_j(\sg)}\right|\right)^2 \gg_{B,\phi_j,\beta_j} \left|\frac{G_j(s)}{\mc{G}_j(\sg)}\right|^{2\left(1+\frac{1}{\gamma_{0,j}}\right)},
\end{align*}
the last estimate coming by Lemma \ref{HalDecay}. The first claim of the lemma follows immediately. The second comes from observing that the map $u \mapsto \frac{u}{2(1+u)}$ is increasing and that $\left|\frac{G_j(\sg)}{\mc{G}_j(\sg)}\right| \leq 1$ by the triangle inequality.
\end{proof}
In what follows, given a Dirichlet series $F(s)$ that is absolutely and uniformly convergent in the half-plane $\text{Re}(s) > 1$, we will write 
\begin{equation*}
J_F(\sg) := \int_{(\sg)} |F'(s)| \frac{|ds|}{|s|^2},
\end{equation*}
for any $\sg > 1$. In order to derive Theorem \ref{HalGen} it will be essentially sufficient to derive sharp bounds for $\mc{G}(\sg)^{-2}J_{F}(\sg)$ in terms of $\sg$ when $F = G$ and $H$. 
Therefore, in the next several lemmata we estimate $J_F(\sg)$ for $F = G$ and $H$. 
\begin{lem} \label{MAIN1}
Let $k \geq 1$. Let $D > 2$ be fixed and set $T := (\sg-1)^{-D}$. Then
\begin{equation*}
J_{G}(\sg) \ll_{m,B,\mbf{\phi},\mbf{\beta}} \mc{G}(\sg)^2\left(B^2(\sg-1)^{-1}\prod_{1 \leq j \leq m} \left|\frac{G_j(\sg)}{\mc{G}_j(\sg)}\right|^{\frac{\gamma_{0,j}}{1+\gamma_{0,j}}} + 1 \right)
\end{equation*}
\end{lem}
\begin{proof}
By definition,
\begin{align}
\mc{G}(\sg)^{-2}J_{G}(\sg)&= \int_{(\sg)} \left|\frac{G'(s)}{G(s)}\right|^2\left|\frac{G(s)}{\mc{G}(\sg)}\right|^2 \frac{|ds|}{|s|^2} \nonumber\\
&\leq \left(\int_{|\tau| \leq T} \left|\frac{G'(s)}{G(s)}\right|^2 \left|\frac{G(s)}{\mc{G}(\sg)}\right|^2 \frac{|ds|}{|s|^2} + \int_{|\tau| > T} \left|\frac{G'(s)}{G(s)}\right|^2 \frac{|ds|}{|s|^2}\right), \label{GOODONE}
\end{align}
upon using $|G(s)| \leq \mc{G}(\sg)$ in the second integral.  
Factoring $G'(s) = \frac{G'(s)}{G(s)} \cdot \frac{G(s)}{\mc{G}(\sg)} \cdot \mc{G}(\sg)$ for $s$ not a zero of $G$, Lemma \ref{StrongComp} implies that
\begin{align} 
\frac{G'(s)}{\mc{G}(\sg)}&= \frac{G(s)}{\mc{G}(\sg)}\frac{d}{ds}\log G(s) = \frac{G(s)}{\mc{G}(\sg)}\left(\frac{d}{ds}\left(\sum_p g(p)p^{-s}\right) + \frac{d}{ds}\log G_0(s)\right) \nonumber\\
&= -\frac{G(s)}{\mc{G}(\sg)}\sum_p \frac{g(p)\log p}{p^s} + \frac{G_0'(s)}{\mc{G}_0(\sg)}\exp\left(-\sum_p \frac{|g(p)|}{p^{\sg}}(1-\cos(\xi_p(\tau)))\right). \label{BALANCE2}
\end{align}
Now, since $G_0(s)$ is uniformly convergent on $\sg > \frac{1}{2}$, $\left|G_0'(s)\right| \ll_B 1$ uniformly on the line $\text{Re}(s) = \sg$. Similarly, $\mc{G}_0(\sg) \gg_B 1$. Thus, the second term above in \eqref{BALANCE2} is $\ll_B 1$, even in neighbourhoods of zeros of $G$. Taking absolute values and squaring, and bounding $\left|\frac{G(s)}{\mc{G}(\sg)}\right| \leq 1$ trivially in the first term, the second integral in \eqref{GOODONE} becomes
\begin{align*}
\int_{|\tau| > T} \left|\frac{G'(s)}{\mc{G}(\sg)}\right|^2 \frac{|ds|}{|s|^2} &\ll_B \int_{|\tau| > T} \left(\left|\sum_p \frac{g(p)\log p}{p^s}\right|^2 + 1\right)\frac{|ds|}{|s|^2} \ll \left(B^2\left(\sum_p \frac{\log p}{p^{\sg}}\right)^2 + 1\right)\int_{|\tau| > T} \frac{|ds|}{|s|^2}  \\
&\ll B^2 (\sg-1)^{-2}T^{-1} = B^2(\sg-1)^{D-2},
\end{align*}
the second last estimate following from the estimates
\begin{equation*}
\sum_p \frac{\log p}{p^{\sg}} = \sum_{n \geq 1} \frac{\Lambda(n)}{n^{\sg}} + O(1) = -\frac{\zeta'(\sg)}{\zeta(\sg)} + O(1) = \frac{1}{\sg-1} + O(1),
\end{equation*}
from the Laurent series of $\zeta$.  \\
For the first integral above, we ignore the (measure zero) set of $s$ for which $\mc{G}(s)=0$. Using the factorization mentioned earlier and \eqref{BALANCE2}, and interchanging $\sum_p \frac{g(p)\log p}{p^s}$ with $\sum_{n \geq 1} g(n)\Lambda(n)n^{-s}$ as above, we have
\begin{align*}
\int_{|\tau| \leq T} \left|\frac{G'(s)}{G(s)}\right|^2 \left|\frac{G(s)}{\mc{G}(\sg)}\right|^2 \frac{|ds|}{|s|^2} &\ll \left(\max_{|\tau| \leq T} \left|\frac{G(s)}{\mc{G}(\sg)}\right|\right)^2\int_{|\tau| \leq T} \left|\sum_p \frac{g(p)\log p}{p^s}\right|^2 \frac{|ds|}{|s|^2} + O_B(1) \\
&\leq \left(\max_{|\tau| \leq T} \left|\frac{G(s)}{\mc{G}(\sg)}\right|\right)^2\int_{(\sg)} \left|\sum_{n \geq 1} \frac{g(n)\Lambda(n)}{n^s}\right|^2 \frac{|ds|}{|s|^2} + O_B(1).
\end{align*}
Now, applying Parseval's theorem, the integral in this last expression is
\begin{equation} \label{PARS}
\int_{(\sg)} \left|\sum_{n \geq 1} \frac{g(n)\Lambda(n)}{n^s}\right|^2 \frac{|ds|}{|s|^2} = 2\pi\int_0^{\infty} e^{-2v\sg} \left|\sum_{n \leq e^v} g(n)\Lambda(n) \right|^2 dv \leq B^2 \int_0^{\infty} e^{-2v(\sg-1)}\left(e^{-v} \sum_{n \leq e^v} \Lambda(n)\right)^2 dv.
\end{equation}
As in the proof of the last lemma, this integral is $\ll (\sg-1)^{-1}$. 
It thus follows that
\begin{equation*}
\int_{|\tau| \leq T} \left|\frac{G'(s)}{G(s)}\right|^2 \left|\frac{G(s)}{\mc{G}(\sg)}\right|^2 \frac{|ds|}{|s|^2} \ll B^2(\sg-1)^{-1}\left(\max_{|\tau| \leq T} \left|\frac{G(s)}{\mc{G}(\sg)}\right|\right)^2 + O_B(1).
\end{equation*}
Inserting the estimates for our two integrals into \eqref{GOODONE}, we get
\begin{align*}
\mc{G}(\sg)^{-2}J_{G}(\sg)&\ll_B B^2\left(\max_{|\tau| \leq T} \left|\frac{G(s)}{\mc{G}(\sg)}\right|\right)^2 + B^2(\sg-1)^{D-2} + 1\ll B^2(\sg-1)^{-1}\left(\max_{|\tau| \leq T} \left|\frac{G(s)}{\mc{G}(\sg)}\right|\right)^2 + 1,
\end{align*}
since $D > 2$. Appealing to Lemma \ref{INTBOUND} completes the proof.
\end{proof}
In the case that $A = A_t := \exp\left(\sum_{p \leq t} \frac{g(p)-|g(p)|}{p}\right)$ and $|\theta_p| \leq \eta_j$ for all $p \in E_j$ and all $j$, we will need to take advantage of the small argument condition to derive an asymptotic formula. In order to account for this condition, therefore, it will be helpful to know that when $\text{Im}(s)$ is sufficiently small, $G(s)$ is approximated well by $A\mc{G}(s)$. In this spirit, we prove the following lemma.
\begin{lem}\label{SHARPH}
Let $g$ be as above, such that for each $1\leq j \leq m$ and $p \in E_j$, we have $|\text{arg}(g(p))| \leq \eta_j < 1$, and let $\eta := \max_{1\leq j \leq n} \eta_j$. Assume also that $B\eta < 1$. For $t \leq x$ sufficiently large and fixed, let $h(n) := g(n)-A|g(n)|$ and $H(s) := \sum_{n \geq 1} \frac{h(n)}{n^s}$, for $\text{Re}(s) > 1$, where $A = A_t$ is as above. Let $\sg = 1+\frac{1}{\log t}$ and $s = \sg + i\tau$. Then for any $K> 0$ such that $B\eta \log(1+K) < 1$ and $K(\sg-1)\log B < 1$,
\begin{equation*}
\int_{|\tau| \leq K(\sg-1)} |H'(\sg + i\tau)|^2 \frac{d\tau}{\sg^2+\tau^2} \ll_B B^2(1+B^2)\eta^2|A|^2\min\left\{\delta^{-2},\log^2(1+K)(1+K)^{-2\delta}\right\}(\sg-1)^{-1} .
\end{equation*}
\end{lem}
\begin{proof}
Observe that if $s$ is not a zero of $G(s)$ or $\mc{G}(s)$ then
\begin{align*}
H'(s) &= G'(s)-A\mc{G}'(s) = G(s)\frac{d}{ds}\log G(s) - A\mc{G}(s) \frac{d}{ds}\log \mc{G}(s) \\
&= H(s)\frac{d}{ds}\left(\sum_p \frac{g(p)}{p^s}\right) + G(s)\cdot \frac{G_0'(s)}{G_0(s)} + A\mc{G}(s) \frac{d}{ds}\left(\sum_p \frac{g(p)-|g(p)|}{p^s}\right) -A\mc{G}(s) \cdot \frac{\mc{G}_0'(s)}{\mc{G}_0(s)},
\end{align*}
by Lemma \ref{StrongComp}. Using the fact that both $\frac{G_0'(s)}{G_0(s)}\cdot \frac{G(s)}{\mc{G}(\sg)}$ and $\frac{\mc{G}_0'(s)}{\mc{G}_0(s)}\cdot \frac{\mc{G}(s)}{\mc{G}(\sg)}$ are both $O_B(1)$ as in Lemma \ref{MAIN1}, we have (at non-zeros of $G$ and $\mc{G}$)
\begin{align*}
H'(s) &= 
-H(s)\sum_{p} \frac{g(p)\log p}{p^s} + A\mc{G}(s)\sum_p \frac{(|g(p)|-g(p))\log p}{p^s}  + O_B(\mc{G}(\sg)).
\end{align*}
Write $\mc{G}(\sg)^{-2}\int_{|\tau| \leq K(\sg-1)} |H'(s)|^2 \frac{|ds|}{|s|^2} \leq 3(J_1 + J_2 + O_B(1))$, where 
\begin{align*}
J_1 &:= \mc{G}(\sg)^{-2}\int_{|\tau| \leq K(\sg-1)} |H(s)|^2 \left|\sum_p \frac{g(p)\log p}{p^s}\right|^2 \frac{|ds|}{|s|^2}, \\
J_2 &:= |A|^2\int_{|\tau| \leq K(\sg-1)} \left|\sum_p \frac{(|g(p)|-g(p))\log p}{p^s}\right|^2\left|\frac{\mc{G}(s)}{\mc{G}(\sg)}\right|^2 \frac{|ds|}{|s|^2}.
\end{align*}
It follows from Lemma \ref{OLD} that $\left|\frac{\mc{G}(s)}{\mc{G}(\sg)}\right| \ll e^{-\delta \log\left(1+\frac{|\tau|}{\sg-1}\right)}$. We will either use this upper bound, or the trivial bound $1$. \\
We begin by computing a uniform estimate for $|H(s)|$. Since $K(\sg-1) \log B < 1$, by Lemma \ref{ZEROLEM} $G$ and $\mc{G}$ are non-zero and holomorphic for $|\tau| \leq K(\sg-1)$, i.e., each of the factors $1+\frac{g(p)}{p^s-1}$ and $1+\frac{|g(p)|}{p^s-1}$ is non-zero, provided $\sg-1$ is sufficiently small (or equivalently, if $t$ is sufficiently large). Observe, therefore, that 
\begin{align*}
\left|\frac{G(\sg)}{\mc{G}(\sg)}A^{-1}\right| &= \left|A^{-1} \prod_p \left(1+\frac{g(p)}{p^{\sg}-1}\right)\left(1+\frac{|g(p)|}{p^{\sg}-1}\right)^{-1}\right| \\
&= \left|\prod_{p \leq t} \left(1+\frac{g(p)-|g(p)|}{p^{\sg}-1 + |g(p)|}\right)e^{-\frac{g(p)-|g(p)|}{p}}\right| \prod_{p > t}\left|1+\frac{g(p)-|g(p)|}{p^{\sg}-1+|g(p)|}\right| \\
&= \exp\left(\text{Re}\left(\sum_{p \leq t} \left(\log\left(1+\frac{g(p)-|g(p)|}{p^{\sg}-1+|g(p)|}\right)-\frac{g(p)-|g(p)|}{p}\right)\right)\right) \\
&\cdot \exp\left(\sum_{p > t} \log\left|1+\frac{g(p)-|g(p)|}{p^{\sg}-1+|g(p)|}\right|\right) \\
&=: e^{P_1+P_2}.
\end{align*}
Estimating $P_2$ first, we note that $|g(p)-|g(p)|| \leq B\eta$ for each $p$, whence by \eqref{TAIL},
\begin{equation*}
P_2 \leq \sum_{p > t} \log\left(1+\frac{B\eta}{p^{\sg}-1}\right) \leq 2B\eta \sum_{p > t} \frac{1}{p^{\sg}} \ll_B B\eta.
\end{equation*}
Next, as $B\eta < 1$ we have
\begin{align*}
P_1 &= \sum_{p \leq t} \text{Re}\left(\left(\frac{g(p)-|g(p)|}{p^{\sg}-1+|g(p)|} - \frac{g(p)-|g(p)|}{p}\right) + \sum_{l \geq 2} \frac{(-1)^{l-1}}{l} \frac{(g(p)-|g(p)|)^l}{(p^{\sg}-1+|g(p)|)^{l}}\right) \\
&\leq B\eta\sum_{B^{\frac{1}{\sg}}+1< p \leq t} \frac{|g(p)| + 1}{p(p-1+|g(p)|)} + O(B^2\eta) 
\ll B(1+B)\eta.
\end{align*}
Furthermore, 
setting $R(\tau) := \frac{G(s)}{\mc{G}(s)}\left(\frac{G(\sg)}{\mc{G}(\sg)}\right)^{-1}$,
\begin{align*}
|R(\tau)| &= \left|\frac{G_0(s)}{G_0(\sg)}\right|\left|\frac{\mc{G}_0(\sg)}{\mc{G}_0(s)}\right|\exp\left(\text{Re}\left(\sum_p (g(p)-|g(p)|)\left(\frac{1}{p^s}-\frac{1}{p^{\sg}}\right)\right)\right) \\
&= \left|\frac{G_0(s)}{G_0(\sg)}\right|\left|\frac{\mc{G}_0(\sg)}{\mc{G}_0(s)}\right| \exp\left(O\left(B\eta \sum_p \left|\frac{1}{p^s}-\frac{1}{p^{\sg}}\right|\right)\right) = \left|\frac{G_0(s)}{G_0(\sg)}\right|\left|\frac{\mc{G}_0(s)}{\mc{G}_0(\sg)}\right|e^{O\left(B\eta \log\left(1+\frac{|\tau|}{\sg-1}\right)\right)}.
\end{align*}
the second last estimate following from Lemmas \ref{OLD} and \ref{StrongComp}. It follows that
\begin{align*}
\left|\frac{G_0(s)}{G_0(\sg)}\right| &= \left|\prod_p e^{-\frac{g(p)-|g(p)|}{p^s}}\left(1+\frac{g(p)-|g(p)|}{p^s-1+|g(p)|}\right)\right| \\
&= \exp\left(\text{Re}\left(\sum_p \left(\log(1+\frac{g(p)-|g(p)|}{p^s-1+|g(p)|})-\frac{g(p)-|g(p)|}{p^s}\right)\right)\right) \\
&= \exp\left(\text{Re}\left(\sum_p \log\left(1+\frac{g(p)-|g(p)|}{p^s}\right)- \frac{g(p)-|g(p)|}{p^s}\right) + O(B(1+B)\eta)\right),
\end{align*}
the error term above owing to
\begin{align*}
&\left|\log\left(1+\frac{g(p)-|g(p)}{p^s-1+|g(p)|}\right)-\log\left(1+\frac{g(p)-|g(p)|}{p^s}\right)\right| \\
&= \left|\log\left(1+(g(p)-|g(p)|)\frac{|g(p)|-1}{(p^s-1+|g(p)|)(p^s+g(p)-|g(p)|)}\right)\right| \\
&\leq \frac{|g(p)-|g(p)|||B-1|}{|p^s-1+|g(p)||(p^{\sg}-|g(p)-|g(p)||)} \leq \frac{B(1+B)\eta}{(p^{\sg}-B\eta)|p^s-1+|g(p)||},
\end{align*} 
the sum of which over all primes is $\ll B(1+B)\eta$, as $B\eta < 1$. For the remaining expression, 
\begin{align*}
&\text{Re}\left(\sum_p \log\left(1+\frac{g(p)-|g(p)|}{p^s}\right) - \frac{g(p)-|g(p)|}{p^s}\right) \leq \sum_{l \geq 2} \frac{1}{l}\left|\frac{g(p)-|g(p)|}{p^s}\right|^l \\
&\leq \frac{B^2\eta^2}{p^{2\sg}}\left(1-\frac{1}{p^{\sg}}\right)^{-1} = \frac{B^2\eta^2}{p^{\sg}(p^{\sg}-1)},
\end{align*}
and summing over all $p$ gives $\ll B^2 \eta^2$. It follows that $\left|\frac{G_0(s)}{\mc{G}_0(s)}\right| = e^{O(B(1+B)\eta)}$. The same argument holds in estimating $\left|\frac{\mc{G}_0(\sg)}{G_0(\sg)}\right|$ upon replacing $g(p)-|g(p)|$ by $|g(p)|-g(p)$ and $s$ by $\sg$.  Hence,
\begin{equation*}
|R(\tau)| = e^{O\left(B\eta\left(\log\left(1+\frac{|\tau|}{\sg-1}\right) + (1+B)\right)\right)} = 1+O\left(B\eta \log\left(1+\frac{|\tau|}{\sg-1}\right)\right),
\end{equation*}
by assumption on $K$, since $\log\left(1+\frac{|\tau|}{\sg-1}\right) \leq \log(1+K)$. 
%
Hence, whenever $B\eta \log(1+K) < 1$,
\begin{equation*}
G(s) = R(\tau)\mc{G}(s)\cdot\frac{G(\sg)}{\mc{G}(\sg)} = A\mc{G}(s)\left(1+O_B\left(B\eta \log\left(1+\frac{|\tau|}{\sg-1}\right)\right)\right),
\end{equation*}
and thus
\begin{equation*}
|H(s)| = |G(s)-A\mc{G}(s)| \ll_B B\eta|A||\mc{G}(s)|\log\left(1+\frac{|\tau|}{\sg-1}\right),
\end{equation*}
Inserting this into $J_1$, bounding $\left|\frac{\mc{G}(s)}{\mc{G}(\sg)}\right|$ by $e^{-\delta \log\left(1+\frac{|\tau|}{\sg-1}\right)}$ and applying Lemma \ref{MODMONT},
\begin{align*}
J_1 &\ll_B B^2\eta^2|A|^2\left(\int_{-K(\sg-1)}^{K(\sg-1)} \log^2\left(1+\frac{|\tau|}{\sg-1}\right)\left|\frac{\mc{G}(s)}{\mc{G}(\sg)}\right|^2\left|\sum_{n \geq 1} g(n)\Lambda(n)n^{-s}\right|^2 d\tau + K(\sg-1)\log^2(1+K)\right)\\
&\leq B^2\eta^2|A|^2\left(B^2\int_{-K(\sg-1)}^{K(\sg-1)} \log^2\left(1+\frac{|\tau|}{\sg-1}\right)e^{-2\delta \log\left(1+\frac{|\tau|}{\sg-1}\right)}\left|\sum_{n \geq 1} \Lambda(n)n^{-s}\right|^2 d\tau + K(\sg-1)\log^2(1+K)\right).
\end{align*}
Observe that the map $ue^{-\delta u}$ is maximized in $u$ when $u = \frac{1}{\delta}$. Hence, if $u = \log\left(1+\frac{|\tau|}{\sg-1}\right)$, this quantity is largest once $|\tau| = \left(e^{\frac{1}{\delta}}-1\right)(\sg-1) =: \tau_0$. Now, if $K \geq \tau_0$ then we can bound the product $\log\left(1+\frac{|\tau|}{\sg-1}\right)e^{-\delta \log\left(1+\frac{|\tau|}{\sg-1}\right)}$ by $\frac{1}{\delta}$, while if $K < \tau_0$ then our bound is $\log(1+K)(1+K)^{-\delta}$.  
Allowing either choice, we have
\begin{align*}
J_1&\ll B^2\eta^2|A|^2\left(B^2\min\left\{\delta^{-2},\frac{\log^2(1+K)}{(1+K)^{2\delta}}\right\} \int_{|\tau| \leq K(\sg-1)} \left|\sum_{n \geq 1} \Lambda(n)n^{-s}\right|^2d\tau + K(\sg-1)\log^2(1+K)\right).
\end{align*}
As before, since $\sum_{n \geq 1} \Lambda(n)n^{-s} = -\frac{\zeta'(s)}{\zeta(s)} = \frac{1}{s-1} + O(1)$,
\begin{align*}
J_1 &\ll B^2\eta^2|A|^2\left(B^2\min\left\{\delta^{-2},\frac{\log^2(1+K)}{(1+K)^{2\delta}}\right\} \int_{|\tau| \leq K(\sg-1)} \frac{d\tau}{|s-1|^2} + K(\sg-1)\log^2(1+K)\right)\\
&\ll B^2\eta^2|A|^2\left(B^2\min\left\{\delta^{-2},\frac{\log^2(1+K)}{(1+K)^{2\delta}}\right\}(\sg-1)^{-1} \int_{|u| \leq K} (1+u)^{-2} du + K(\sg-1)\log^2(1+K)\right) \\
&\ll B^2\eta^2|A|^2\left(B^2\min\left\{\delta^{-2},\frac{\log^2(1+K)}{(1+K)^{2\delta}}\right\}(\sg-1)^{-1} + K(\sg-1)\log^2(1+K)\right).
\end{align*}
Similarly, we can apply Lemma \ref{MODMONT} to $J_2$, noting that $(|g(n)|-g(n))\Lambda(n) \leq B\eta\Lambda(n)$ in this case, to get
\begin{equation*}
J_2 \ll B^2\eta^2|A|^2\int_{|\tau| \leq K(\sg-1)} \left|\sum_{n \geq 1} \Lambda(n)n^{-s}\right|^2 d\tau \ll B^2\eta^2 |A|^2\int_{|\tau| \leq K(\sg-1)} \frac{d\tau}{|s-1|^2} \ll B^2\eta^2|A|^2(\sg-1)^{-1}.
\end{equation*}
Collecting the estimates for $J_1$ and $J_2$,
\begin{equation*}
\int_{|\tau| \leq K(\sg-1)} |H'(s)|^2 \frac{|ds|}{|s|^2} \ll_B B^2(1+B^2)\eta^2 |A|^2\min\left\{\delta^{-2},\log^2(1+K)(1+K)^{-2\delta}\right\}(\sg-1)^{-1} 
\end{equation*}
This completes the proof.
\end{proof}
\begin{rem}
Similar estimates can be had if $g$ is completely multiplicative. In such a case any factor of the form $1+g(p)/(p^s-1)$ is replaced by $(1-g(p)/p^2)^{-1}$, and a Taylor expansion of the corresponding logarithm is available whenever $|g(p)|  < 2$.
\end{rem}
\begin{lem}\label{SHARPAPP}
Let $K > 0$ be such that $B\eta \log(1+K) < 1$ and $T \geq 1$ be fixed such that $K(\sg-1) < 2$ and $(\sg-1)^{D-1} < (\sg-1)^{-1}T^{-1} < 1$ for some $D> 1$.
Furthermore, set $\gamma := \min_{1 \leq j \leq n} \beta_j^3 \delta_j$. Then
\begin{align*}
J_{H,1}(\sg) &\ll_{B,D} B^2\mc{G}(\sg)^2(\sg-1)^{-1}(\eta^2|A|^2\min\left\{\delta^{-2},\log^2(1+K)(1+K)^{-2\delta}\right\} + (\sg-1)^{-1}T^{-1} \\
&+ (\sg-1)^{\gamma}|A|^{\frac{\gamma_0}{2(1+\gamma_0)}} + |A|^2(\sg-1)^{2\delta/3} + \gamma^{-1}|A|^{\frac{\gamma_0}{2(1+\gamma_0)}} (1+K)^{-(1+\gamma)} + \frac{|A|^2}{\delta} (1+K)^{-(1+2\delta)})
\end{align*}
\end{lem}
\begin{proof}
By definition
\begin{align*}
\mc{G}(\sg)^{-2}J_{H,1}(\sg)
&= \mc{G}(\sg)^{-2}\left(\int_{|\tau| \leq K(\sg-1)} + \int_{K(\sg-1) < |\tau| \leq 2} + \int_{2 < |\tau| \leq T} +\int_{|\tau| > T} \right) \left|H'(s)\right|^2 \frac{|ds|}{|s|^2} \\
&=: \mc{G}(\sg)^{-2}\left(I_1 + I_2 + I_3 +I_4\right).
\end{align*}
We deal with each of the $I_j$ in turn. The content of Lemma \ref{SHARPH} shows that 
\begin{equation*}
I_1 \ll_B B^2(1+B^2)|A|^2\eta^2(\sg-1)^{-1}\max\left\{\delta^{-2},\log^2(1+K)(1+K)^{-2\delta}\right\}.
\end{equation*}
Next, we consider $I_4$. Note that $|H'(s)|^2 \ll |G'(s)|^2 + |A|^2|\mc{G}(s)|^2$, thus
\begin{equation*}
I_4 \ll \mc{G}(\sg)^2\left(\int_{|\tau| > T} \left|\frac{G'(s)}{\mc{G}(s)}\right|^2\frac{|ds|}{|s|^2} + |A|^2\int_{|\tau|>T} \left|\frac{\mc{G}'(s)}{\mc{G}'(\sg)}\right|^2 \frac{|ds|}{|s|^2}\right).
\end{equation*}
The first integral is implicitly contained in the proof of Lemma \ref{MAIN1}, where we showed that it was bounded above by $B^2 (\sg-1)^{-2}T^{-1}$.  The proof of this estimate can be carried out in precisely the same way to get the same bound for the integral in $\mc{G}(s)$ as well. Hence, as $|A| \leq 1$,
\begin{equation*}
I_4 \ll B^2(1+|A|^2)\mc{G}(\sg)^2(\sg-1)^{-2}T^{-1} \ll B^2\mc{G}(\sg)^2(\sg-1)^{-2}T^{-1},
\end{equation*}
Consider now $I_3$.  Upon using Lemma \ref{HalDecay} and applying the arguments of Lemma \ref{MAIN1} to the first term, we have
\begin{align*}
\int_{2 < |\tau| \leq T} \left|\frac{G'(s)}{G(s)}\right|^2 \left|\frac{G(s)}{\mc{G}(\sg)}\right|^2 \frac{|ds|}{|s|^2}  &\leq 
\left(\max_{|\tau| \leq T} \left|\frac{G(s)}{\mc{G}(\sg)}\right|\right)\int_{2 < |\tau| \leq T}e^{-\gamma \log\left(1+\frac{|\tau|}{\sg-1}\right)} \left(\left|\sum_p \frac{g(p)\log p}{p^s}\right|^2  + O_B(1)\right)\frac{|ds|}{|s|^2} \\
&\leq \left|\frac{G(\sg)}{\mc{G}(\sg)}\right|^{\frac{\gamma_0}{2(1+\gamma_0)}}(\sg-1)^{\gamma} \int_{(\sg)}\left(\left|\sum_p \frac{g(p)\log p}{p^s}\right|^2  + O_B(1)\right)\frac{|ds|}{|s|^2}\\
&\ll_B B^2|A|^{\frac{\gamma_0}{2(1+\gamma_0)}}(\sg-1)^{-1+\gamma}
\end{align*}
where, as in Lemma \ref{SHARPH}, $\left|\frac{G(\sg)}{\mc{G}(\sg)}\right| \asymp_B |A|$, in the last inequality. In the $\mc{G}$ integral we use the pointwise estimate $\left|\frac{\mc{G}(s)}{\mc{G}(\sg)}\right| \ll_B \left|\frac{\zeta(s)}{\zeta(\sg)}\right|^{\delta}$ that we invoked earlier. As a corollary of the Korobov-Vinogradov zero-free region of $\zeta$ it can be deduced that when $|\tau| \geq 2$, $|\zeta(\sg + i\tau)| \ll \log^{\frac{2}{3}} |\tau|$ for $\sg > 1$ (see Note 3.9 in Chapter II.3 of \cite{Ten2}). As a result,
\begin{align*}
\int_{2 < |\tau| \leq T} \left|\frac{\mc{G}'(s)}{\mc{G}(s)}\right|^2 \left|\frac{\mc{G}(s)}{\mc{G}(\sg)}\right|^2 \frac{|ds|}{|s|^2}  &\leq (\sg-1)^{2\delta}\int_{2 < |\tau| \leq T}\log^{\frac{2}{3}\delta}|\tau|\left(\left|\sum_p \frac{g(p)\log p}{p^s}\right|^2 + O_B(1)\right)\frac{|ds|}{|s|^2} \\
&\leq \left((\sg-1) \log^{\frac{2}{3}}T\right)^{2\delta} \int_{(\sg)}\left(\left|\sum_p \frac{g(p)\log p}{p^s}\right|^2  + O_B(1)\right)\frac{|ds|}{|s|^2} \\
&\ll_D (\sg-1)^{\frac{2\delta}{3}} \int_{(\sg)}\left(\left|\sum_p \frac{g(p)\log p}{p^s}\right|^2  + O_B(1)\right)\frac{|ds|}{|s|^2} \\
&\ll_B B^2(\sg-1)^{-1+\frac{2\delta}{3}},
\end{align*}
where, in the last estimate, we used Parseval's theorem as in \eqref{PARS}. In sum, we have
\begin{equation*}
I_3 \ll_{B,D} B^2 \mc{G}(\sg)^2(\sg-1)^{-1}\left((\sg-1)^{\gamma}|A|^{\frac{\gamma_0}{2(1+\gamma_0)}} + |A|^2(\sg-1)^{\frac{2\delta}{3}}\right).
\end{equation*}
It remains to estimate $I_2$. Consider first the $G$ integral. We decompose the interval $(K(\sg-1),2]$ dyadically and let $L$ be the number of subintervals thus produced. Set $\psi_l := 2^lK(\sg-1)$ for each $0 \leq l \leq L-1$. Applying Lemma \ref{HalDecay} and arguing that we can ignore the effect of any of the zeros of $G$ as before (from $\frac{G_0'(s)}{G_0(s)}G(s) \ll_B \mc{G}(\sg)$),
\begin{align*}
&\int_{K(\sg-1) < |\tau| \leq 2} \left|\frac{G'(s)}{G(s)}\right|^2 \left|\frac{G(s)}{\mc{G}(\sg)}\right|^2 \frac{|ds|}{|s|^2} \\
&\leq |A|^{\frac{\gamma_0}{2(1+\gamma_0)}}\int_{K(\sg-1)<|\tau| \leq 2} e^{-\gamma \log\left(1+\frac{|\tau|}{\sg-1}\right)} \left(\left|\sum_{n \geq 1} \frac{g(n)\Lambda(n)}{n^s}\right|^2 + O_B(1)\right)\frac{|ds|}{|s|^2}\\
&\leq 2|A|^{\frac{\gamma_0}{2(1+\gamma_0)}}\sum_{0 \leq l \leq L-1} e^{-\gamma \log\left(1+2^lK\right)}\left( \int_{\psi_l}^{\psi_{l+1}} \left|\sum_{n \geq 1} \frac{g(n)\Lambda(n)}{n^s}\right|^2 \frac{|ds|}{|s|^2} + O_B(1)\right)\\
&= 2|A|^{\frac{\gamma_0}{2(1+\gamma_0)}}\sum_{0 \leq l \leq L-1} (1+2^lK)^{-\gamma} \left(\int_{\psi_l}^{\psi_{l+1}} \left|\sum_{n \geq 1} \frac{g(n)\Lambda(n)}{n^s}\right|^2 \frac{|ds|}{|s|^2} + O_B(1)\right).
\end{align*}
Let $q > 1$ be a parameter to be chosen, and let $p$ be its H\"{o}lder conjugate (i.e., $\frac{1}{q}+\frac{1}{p} =1$). By H\"{o}lder's inequality,
\begin{align*}
\int_{\psi_l}^{\psi_{l+1}} \left|\sum_{n \geq 1} \frac{g(n)\Lambda(n)}{n^s}\right|^2 \frac{|ds|}{|s|^2} &\leq \left(\int_{\psi_{l}}^{\psi_{l+1}} \frac{d\tau}{(\sg^2 + \tau^2)^q} \right)^{\frac{1}{q}} \left(\int_{\psi_{l}}^{\psi_{l+1}} \left|\sum_{n \geq 1} \frac{g(n)\Lambda(n)}{n^s}\right|^{2p} d\tau\right)^{\frac{1}{p}} \\
&\leq \psi_l^{\frac{1}{q}} \int_{\psi_{l}}^{\psi_{l+1}} \left|\sum_{n \geq 1} \frac{g(n)\Lambda(n)}{n^s}\right|^2 d\tau.
\end{align*}
For convenience, for each $0 \leq l \leq L-1$ let $\alpha_l := (1+2^lK)^{-\gamma}$ and, given an arithmetic function $a$, set $J_l(a) := \int_{\psi_{l}}^{\psi_{l+1}} \left|\sum_{n \geq 1} \frac{a(n)\Lambda(n)}{n^s}\right|^2 d\tau$. With these notations, we have
\begin{equation*}
\int_{K(\sg-1) < |\tau| \leq 2} \left|\frac{G'(s)}{G(s)}\right|^2 \left|\frac{G(s)}{\mc{G}(\sg)}\right|^2 \frac{|ds|}{|s|^2} \leq 2|A|^{\frac{\gamma_0}{2(1+\gamma_0)}}\sum_{0 \leq l \leq L-1} \alpha_l\psi_l^{\frac{1}{q}} J_l(g). 
\end{equation*}
By Lemma \ref{MODMONT}, as $|g(n)|\Lambda(n) \leq B\Lambda(n)$ for each $n$,
\begin{align*}
2J_l(g) &\leq \int_{|\tau| \leq \psi_{l+1}} \left|\sum_{n \geq 1} \frac{g(n)\Lambda(n)}{n^s}\right|^2 d\tau \leq 3B^2\int_{|\tau| \leq \psi_{l+1}} \left|\sum_{n \geq 1} \frac{\Lambda(n)}{n^s}\right|^2 d\tau = 3B^2 \int_{|\tau| \leq \psi_{l+1}} \left|\frac{\zeta'(s)}{\zeta(s)}\right|^2 d\tau =: 3B^2\iota_l
\end{align*}
Now, we clearly have $\iota_l = 2J_l(1) + \iota_{l-1}$, so that
\begin{equation*}
2\sum_{0 \leq l \leq L-1} \alpha_l\psi_l^{\frac{1}{q}}J_l(g) \leq 3B^2 \sum_{0 \leq l \leq L-1} \alpha_l\psi_l^{\frac{1}{q}} \iota_l = 3B^2\left(2\sum_{0 \leq l \leq L-1} \alpha_l \psi_l^{\frac{1}{q}}J_l(1) + \sum_{1 \leq l \leq L-1} \alpha_l\psi_l^{\frac{1}{q}} \iota_{l-1}\right).
\end{equation*}
Observe now that for any $1 \leq l \leq L-1$ and $K > 1$, we have $\frac{2^{l-1}K}{1+2^lK} \geq \frac{1}{3}$, whence
\begin{align*}
\frac{\alpha_l\psi_l^{\frac{1}{q}}}{\alpha_{l-1}\psi_{l-1}^{\frac{1}{q}}} &= 2^{\frac{1}{q}} \left(1-\frac{2^{l-1}K}{1+2^lK}\right)^{2\gamma} \leq 2^{\frac{1}{q}} \exp\left(-2\gamma \frac{2^{l-1}K}{1+2^lK}\right) \leq 2^{\frac{1}{q}}e^{-\frac{2\gamma}{3}} = 2^{\frac{1}{q}-\frac{2\gamma}{3\log 2}}.
\end{align*}
Choosing $q := \frac{3\log 2}{\gamma}$ yields, with $\omega := 2^{-\frac{\gamma}{3\log 2}} < 1$, the inequality $\alpha_l \psi_l^{\frac{1}{q}} \leq \omega \alpha_{l-1}\psi_{l-1}^{\frac{1}{q}}$ uniformly in $1 \leq l \leq L-1$. It follows that
\begin{align*}
\sum_{0 \leq l \leq L-1} \alpha_l\psi_l^{\frac{1}{q}} \iota_l &\leq 2\sum_{0 \leq l \leq L-1} \alpha_l \psi_l^{\frac{1}{q}}J_l(1) + \omega\sum_{1 \leq l \leq L-1} \alpha_{l-1}\psi_{l-1}^{\frac{1}{q}} \iota_{l-1} \\
&\leq 2\sum_{0 \leq l \leq L-1} \alpha_l \psi_l^{\frac{1}{q}}J_l(1) + \omega\sum_{0 \leq l \leq L-1} \alpha_{l}\psi_{l}^{\frac{1}{q}} \iota_{l}.
\end{align*}
Rearranging this inequality, we get
\begin{equation*}
\sum_{0 \leq l \leq L-1} \alpha_l\psi_l^{\frac{1}{q}} \iota_l \leq \frac{2}{1-\omega} \sum_{0 \leq l \leq L-1} \alpha_l\psi_l^{\frac{1}{q}}J_l(1).
\end{equation*}
We now have
\begin{align*}
\int_{K(\sg-1) < |\tau| \leq 2} \left|\frac{G'(s)}{G(s)}\right|^2 \left|\frac{G(s)}{\mc{G}(\sg)}\right|^2 \frac{|ds|}{|s|^2} &\leq \frac{6B^2}{1-\omega}|A|^{\frac{\gamma_0}{2(1+\gamma_0)}}\sum_{0 \leq l \leq L-1} \alpha_l\psi_l^{\frac{1}{q}} J_l(1) \ll \frac{B^2}{\gamma} \sum_{0 \leq l \leq L-1} \alpha_l\psi_l^{\frac{1}{q}} J_l(1) \\
&\ll \frac{B^2}{\gamma} |A|^{\frac{\gamma_0}{2(1+\gamma_0)}}\sum_{0 \leq l \leq L-1} \alpha_l \int_{\psi_l}^{\psi_{l+1}} \left|\frac{\zeta'(s)}{\zeta(s)}\right|^2 d\tau \\
&\ll \frac{B^2}{\gamma} |A|^{\frac{\gamma_0}{2(1+\gamma_0)}}\int_{K(\sg-1) < |\tau| \leq 2} \left(1+\frac{|\tau|}{\sg-1}\right)^{-2\gamma} \left|\frac{\zeta'(s)}{\zeta(s)}\right|^2 d\tau,
\end{align*}
the second last estimate following from $\psi_l \leq 2$ for all $l$, and the last being obvious from $\alpha_l \leq 2^{2\gamma} \left(1+2^{l+1}K\right)^{-2\gamma}$. Observe that by the Laurent series representation of $\zeta$, 
\begin{align*}
&\int_{K(\sg-1) < |\tau| \leq 2} \left(1+\frac{|\tau|}{\sg-1}\right)^{-2\gamma} \left|\frac{\zeta'(s)}{\zeta(s)}\right|^2 d\tau \ll \int_{K(\sg-1) < |\tau| \leq 2} \left(1+\frac{|\tau|}{\sg-1}\right)^{-2\gamma} \frac{d\tau}{|s-1|^2} \\
&\leq (\sg-1)^{-1} \int_K^{\infty} (1+u)^{-2\gamma} (1+u^2)^{-1} du \ll (\sg-1)^{-1}\int_K^{\infty} (1+u)^{-2(1+\gamma)} du \\
&= (\sg-1)^{-1}(1+K)^{-(1+\gamma)}.
\end{align*}
As such,
\begin{equation*}
\int_{K(\sg-1) < |\tau| \leq 2} \left|\frac{G'(s)}{G(s)}\right|^2 \left|\frac{G(s)}{\mc{G}(\sg)}\right|^2 \frac{|ds|}{|s|^2} \ll \frac{B^2}{\gamma}|A|^{\frac{\gamma_0}{2(1+\gamma_0)}}(\sg-1)^{-1}(1+K)^{-(1+\gamma)}.
\end{equation*}
For the $\mc{G}$ integral, we use $\left|\frac{\mc{G}(s)}{\mc{G}(\sg)}\right| \ll \left(1+\frac{|\tau|}{\sg-1}\right)^{-2\delta}$, so that our preceding arguments for the $G$ integral apply with $2\delta$ here in place of $\gamma$ (of course, without appealing to Lemma \ref{INTBOUND} in this case). It then follows that
\begin{equation*}
\int_{K(\sg-1) < |\tau| \leq 2} \left|\frac{\mc{G}'(s)}{\mc{G}(s)}\right|^2 \left|\frac{\mc{G}(s)}{\mc{G}(\sg)}\right|^2 \frac{|ds|}{|s|^2} \ll \frac{B^2}{\delta}(\sg-1)^{-1}(1+K)^{-(1+2\delta)}.
\end{equation*}
Thus, 
\begin{align*}
I_2 &\ll B^2\mc{G}(\sg)^2(\sg-1)^{-1}\left(\gamma^{-1}|A|^{\frac{\gamma_0}{2(1+\gamma_0)}}(1+K)^{-(1+\gamma)} + \delta^{-1}|A|^2(1+K)^{-(1+2\delta)}\right).
\end{align*}
Combining all of our integral estimates gives
\begin{align*}
J_{H,1}(\sg) &\ll B^2(1+B^2)\mc{G}(\sg)^2(\sg-1)^{-1} (|A|^2\eta^2\max\left\{\delta^{-2},\log^2(1+K)(1+K)^{-2\delta}\right\} + (\sg-1)^{-1}T^{-1} \\
&+ (\sg-1)^{\gamma}|A|^{\frac{\gamma_0}{2(1+\gamma_0)}} + |A|^2(\sg-1)^{2\delta/3} + \frac{1}{\gamma} |A|^{\frac{\gamma_0}{2(1+\gamma_0)}}(1+K)^{-(1+\gamma)} + \frac{|A|^2}{\delta} (1+K)^{-(1+2\delta)}),
\end{align*}
as claimed.
\end{proof}
\begin{rem}\label{REMGEN2}
In the same vein as Remark \ref{REMGEN}, we may observe that Lemmas \ref{SHARPH} and \ref{SHARPAPP} are equally valid upon replacing $h(n)$ by $|g(n)|-Xf(n)$, in which case $X = X_t := \exp\left(-\sum_{p \leq t} \frac{f(p)-|g(p)|}{p}\right)$, $|g(n)| \leq f(n)$, $|f(p)| \leq B$ and $\left|\frac{|g(p)|}{f(p)}-1\right| \leq \eta$.  To replace Lemma \ref{HalDecay} and its corollary, Lemma \ref{INTBOUND}, putting $F(s) := \sum_{n \geq 1} \frac{f(n)}{n^s}$ we simply use 
\begin{equation*}
\left|\frac{\mc{G}(s)}{F(\sg)}\right| \leq \frac{\mc{G}(\sg)}{F(\sg)} \ll X, 
\end{equation*}
the last estimate following by arguments similar to those in the proof of Lemma \ref{SHARPH}.
\end{rem}
\section{From Arithmetic to Analysis}
In this section we bridge the gap between the results of Section 4 and those of Section 5.  Our first lemma in this direction allows us to estimate with precision one of the terms arising in Lemma \ref{NUM}. For the definitions of \emph{reasonable} and \emph{non-decreasing} pairs $(g,\mbf{E})$, see Definition \ref{REAS}.
\begin{lem} \label{WITHR}
There exists $\lambda = \lambda(t) > 0$ such that for all $\kappa \in (0,1)$, 
\begin{equation}
\int_1^{t^{\kappa}} \frac{M_{|g|}(u)}{u^2 \log^{\lambda} u} \ll_B \left(\int_1^t \frac{M_{|g|}(u)}{u^2} du\right) \kappa^{-\lambda}e^{-\frac{1}{2}S_{\kappa}(t)}\log^{-\lambda} t ,\label{INTRAT}
\end{equation}
where $S_{\kappa}(t) = \sum_{t^{\kappa} < p \leq t} \frac{g(p)}{p}$. \\
In particular, for any $c_j > 0$, and $\gamma_{0,j}$ as defined in Lemma \ref{HalDecay} we can take 
\begin{equation*}
\lambda(t) := \frac{1}{\log_2 t}\sum_{1 \leq j \leq m} c_jd_j\frac{\gamma_{0,j}}{1+\gamma_{0,j}}\sum_{p \leq t \atop p \in E_j} \frac{|g(p)|}{p}, 
\end{equation*}
where $d_j = 1$ if $(g,\mbf{E})$ is non-decreasing, and $d_j := \frac{\delta_j}{B_j}$ otherwise.
\end{lem}
\begin{rem}\label{REMWITHR}
Note that, provided that $\log(1/\kappa) < \log_2 t$, which we will show to be true (see Lemma \ref{EXPONENTS} below and the remarks following its statement), the right side of \eqref{INTRAT} is decreasing in $\lambda$. Thus, we may always decrease $\lambda$ if necessary and the bound still holds with this smaller value of $\lambda$. \\
We should also remark that it follows from partial summation and the condition $\log P_t = \sum_{p \in S \atop p \leq t} \log p \ll_r r \log t$, that
\begin{equation*}
\sum_{p \in S \atop s < p \leq t} \frac{1}{p} = \int_s^t d\left\{\sum_{p \in S \atop p \leq u} \log p\right\} \frac{1}{u \log u} \ll_r \frac{1}{s} + \int_s^t \left(\sum_{p \in S \atop p \leq u} \log p\right) \frac{du}{u^2 \log u} \ll_r \frac{1}{s}.
\end{equation*}
Thus, for all intents and purposes, sums over primes from $S$ give a negligible contribution.
\end{rem}
\begin{proof}
By partial summation, for any $\lambda > 0$ we have
\begin{equation*}
\sum_{n \leq u} \frac{|g(n)|}{n\log^{\lambda}(3n)} = \frac{M_{|g|}(u)}{u\log^{\lambda}(3u)} + \int_1^u \left(1+\frac{\lambda}{\log u}\right)\frac{M_{|g|}(u)}{u^2 \log^{\lambda}(3u)}du \geq \int_1^u \frac{M_{|g|}(u)}{u^2 \log^{\lambda}(3u)}du.
\end{equation*}
Also, by Lemmas \ref{LOGSUM} and \ref{CHEAP}, the integral over $[1,t]$ in the statement of the lemma is bounded from below by $P_{|g|}(t)$. Thus, it suffices to estimate $P_{|g|}(t)^{-1}\sum_{n \leq t^{\kappa}} \frac{|g(n)|}{n\log^{\lambda}(3n)}$. \\
Set $\lambda$ as in the statement of the lemma (depending on whether or not $\{|g(p)|\}_{p \in E_j}$ is non-decreasing for each $j$).  
Define $\alpha > 0$ to be such that if $Z := e^{\log^{\alpha} t}$ then 
\begin{equation*}
\frac{1}{2\log_2 t}\sum_{Z < p \leq t} \frac{|g(p)|}{p} = \lambda.
\end{equation*}
Take $L := \llf \frac{\log \left(\kappa \log^{1-\alpha} t\right)}{\log 2} \rrf$, and if $L \geq 1$ and $0 \leq k \leq L-1$ set $D_k := (Z^{2^k},Z^{2^{k+1}}]$ (if $L = 0$ then the sum over $k$ below is empty).  We can suppose that the expression in the floor function brackets defining $L$ is an integer by perturbing $t$ or $\kappa$ slightly (which will not change the order of magnitude of the bounds). Let $\psi_k := \sum_{n \leq Z^{2^{k}}} \frac{|g(n)|}{n}$ (not to be confused with the $\psi_l$ in Lemma \ref{SHARPAPP}). Thus, we have
\begin{align*}
&P_{|g|}(t)^{-1}\sum_{n \leq t} \frac{|g(n)|}{n\log^{\lambda}(3n)} \leq P_{|g|}(t)^{-1}\sum_{n \leq Z} \frac{|g(n)|}{n} + P_{|g|}(t)^{-1}\sum_{0 \leq k \leq L-1} \sum_{n \in D_k} \frac{|g(n)|}{n\log^{\lambda}(3n)} \\
&\leq P_{|g|}(t)^{-1} \prod_{p \leq Z} \left(1+\frac{|g(p)|}{p}\right) + P_{|g|}(t)^{-1}\log^{-\lambda} Z \sum_{0 \leq k \leq L-1} 2^{-k\lambda} (\psi_{k+1}-\psi_k) \\
&\leq P_{|g|}(t)^{-1} \prod_{p \leq Z} \left(1+\frac{|g(p)|}{p}\right) + P_{|g|}(t)^{-1}(1-2^{-\lambda})\log^{-\lambda} Z \sum_{0 \leq k \leq L-1} 2^{-k\lambda} \psi_{k+1}+ 2^{-\lambda L}\log^{-\lambda}Ze^{-S_k(t)},
\end{align*}
where the last estimate follows by partial summation, noting that in the last term on the right, $L_{|g|}(t^{\kappa})$ is at most $P_{|g|}(t^{\kappa})$, and for $t^{\kappa}$ sufficiently large,
\begin{equation*}
\prod_{t^{\kappa} < p \leq t} \left(1+\frac{|g(p)|}{p}\right)^{-1} = \exp\left(-\sum_{t^{\kappa} < p \leq t} \log\left(1+\frac{|g(p)|}{p}\right)\right) \ll_B e^{-S_{\kappa}(t)}.
\end{equation*}
Since $1-2^{-\lambda} \asymp \lambda$, it follows that
\begin{align} 
P_{|g|}(t)^{-1}\sum_{n \leq t} \frac{|g(n)|}{n\log^{\lambda}(3n)} 
&\ll e^{-\frac{1}{2}S_{\kappa}(t)} e^{-\frac{1}{2}\sum_{Z < p \leq t} \frac{|g(p)|}{p}}\left(1 + \lambda \kappa^{-\lambda} \sum_{0\leq k \leq L-1} \frac{\kappa^{\lambda}}{2^{k\lambda}\log^{\lambda} Z} \exp\left(\frac{1}{2} \sum_{Z < p \leq Z^{2^{k+1}}} \frac{|g(p)|}{p}\right)\right) \\
&+ 2^{-\lambda L}e^{-S_{\kappa}(t)}\log^{-\lambda }Z. \label{DYAD}
\end{align}
By choice, the first factor is precisely $e^{-\frac{1}{2}S_{\kappa}(t)}\log^{-\lambda} t$.
Now, observe that for each $0 \leq j \leq L-1$, there is some $\xi_j \in [\delta,B]$ such that $\sum_{Z^{2^{j}} < p \leq Z^{2^{j+1}}} \frac{|g(p)|}{p} = \xi_j\log 2 + O_B\left(\frac{1}{2^j \log Z}\right)$. In particular, 
we can write
\begin{align*}
\lambda &= \frac{1}{2\log_2 t} \sum_{Z < p \leq t} \frac{|g(p)|}{p}
=\frac{\log 2}{2\log_2 t}\left(\sum_{0 \leq l \leq L-1} \xi_l +\frac{1}{\log 2} S_{\kappa}(t)+ O_B\left(\frac{1}{\log Z}\right)\right).
\end{align*}
Additionally, by definition we have
\begin{equation*}
\log^{-\lambda} Z = \exp\left(-\frac{1}{2}\alpha \sum_{Z<p \leq t} \frac{|g(p)|}{p}\right) \ll \exp\left(-\frac{1}{2}\alpha \log 2  \sum_{0 \leq l \leq L-1} \xi_l\right).
\end{equation*}
Hence, 
it follows that each term in the sum in \eqref{DYAD} takes the form
\begin{equation*}
2^{-k\lambda} \kappa^{\lambda}\exp\left(\frac{1}{2}\sum_{Z < p \leq Z^{2^{k+1}}} \frac{|g(p)|}{p}\right)\log^{-\lambda} Z \ll 2^{-\frac{1}{2}(k+1)\e_k},
\end{equation*}
where, for $\Sigma_j := \sum_{0 \leq l \leq j} \xi_l$, we have set
\begin{align}
\e_k &:= \lambda\left( 1+ \frac{2}{(k+1)\log 2}\log(1/\kappa)\right) + \frac{1}{(k+1)\log 2}\left(\log_2 Z - \log 2\sum_{0 \leq l \leq k} \xi_l\right) \nonumber\\
&= \frac{1}{k+1}\left(\frac{\log 2}{\log_2 t}\left((k+1)\Sigma_{L-1} + \frac{S_{\kappa}(t)}{\log 2} + \frac{2\log(1/\kappa)}{\log 2}\left(\Sigma_{L-1} + S_{\kappa}(t)\right)\right) + \alpha\Sigma_{L-1} - \Sigma_k\right). \label{CANCEL}
\end{align}
By choice, we have
\begin{equation*}
\frac{\log 2}{\log_2 t} = \frac{1-\alpha}{L}\left(1-\frac{\log(1/\kappa)}{\log_2 t} \right),
\end{equation*}
so that, as $\frac{1-\alpha}{L}\frac{(k+1)\log(1/\kappa)}{\log_2 t}$ is canceled by the $\log(1/\kappa)$ contribution in \eqref{CANCEL} uniformly in $k \leq L-1$, we have
\begin{align*}
\e_k &\geq \frac{1}{k+1}\left(\frac{1-\alpha}{L}\left((k+1)\Sigma_{L-1} + \frac{S_{\kappa}(t)}{\log 2}\right) + \alpha\Sigma_{L-1} - \Sigma_k\right)\\
&\geq \frac{1}{k+1}\left(\Sigma_{L-1}\left(\frac{k+1}{L}(1-\alpha) + \alpha\right) - \Sigma_k\right) = \frac{1}{k+1}\Sigma_{L-1}\left(\frac{k+1}{L} + \frac{L-k-1}{L}\alpha-\frac{\Sigma_k}{\Sigma_{L-1}}\right).
\end{align*}
%
%
%
Consider when $(g,\mbf{E})$ is non-decreasing for each $j$. We claim that $\frac{\Sigma_k}{\Sigma_{L}} \leq \frac{k+1}{L}$.  To see this, write $\xi_{l,j}$ to denote the contribution to $\xi_l$ coming from primes in $E_j$. Then
\begin{equation*}
\frac{1}{k+1}\sum_{0 \leq l \leq k} \xi_{l,j} \leq \frac{1}{L}\sum_{0 \leq l \leq L-1} \xi_{l,j}.
\end{equation*}
It follows that
\begin{equation*}
\frac{1}{k+1}\Sigma_k = \sum_{1 \leq j \leq m}  \left(\frac{1}{k+1}\sum_{0\leq l \leq k} \xi_{l,j}\right) \leq \frac{1}{L}\sum_{1 \leq j \leq m} \sum_{0 \leq l \leq L-1} \xi_{l,j}  = \frac{1}{L}\Sigma_L.
\end{equation*}
Hence, we have $(k+1)\e_k \geq (L-k-1)\alpha\Sigma_{L-1}/L =: (L-k-1)\tilde{\lambda}$. Inserting this into the geometric series in \eqref{DYAD}, reindexing and summing, we get
\begin{align*}
P_{|g|}(t)^{-1}\sum_{n \leq t} \frac{|g(n)|}{n\log^{\lambda}(3n)}&\ll_B e^{-\frac{1}{2}S_{\kappa}(t)}\log^{-\lambda} t\left(1 + \lambda\kappa^{-\lambda}\sum_{0\leq k \leq L-1} 2^{-k\tilde{\lambda}}\right) + 2^{-\lambda L}e^{-S_{\kappa}(t)}\log^{-\lambda }Z \\
&\ll e^{-\frac{1}{2}S_{\kappa}(t)}\log^{-\lambda} t\left(1 + \frac{\lambda}{\tilde{\lambda}}\kappa^{-\lambda}\right) + \kappa^{-\lambda}e^{-S_{\kappa}(t)}\log^{-\lambda} t,
\end{align*}
where, in the last line we used the definition of $L$. Now, we saw earlier that 
\begin{equation*}
\lambda  = \frac{1-\alpha}{L}\Sigma_{L-1}\left(1 + \frac{S_{\kappa}(t)}{\Sigma_{L-1}\log 2}\right) = \frac{1-\alpha}{\alpha}\tilde{\lambda}\left(1+ \frac{S_{\kappa}(t)}{\Sigma_{L-1}\log 2}\right).
\end{equation*}
Hence, $\lambda/\tilde{\lambda} \ll \frac{1-\alpha}{\alpha}$, and we claim that $\alpha > \frac{1}{2}$. Indeed, by monotonicity, 
\begin{equation*}
\sum_{p \leq e^{\sqrt{\log t}} \atop p \in E_j} \frac{|g(p)|}{p} \leq \sum_{e^{\sqrt{\log t}} < p \leq t \atop p \in E_j} \frac{|g(p)|}{p}, 
\end{equation*}
and since $\lambda < \frac{1}{4}$ as can be checked by the definition given, if $\alpha = \frac{1}{2}$ then
\begin{equation*}
\lambda \log_2 t = \sum_{1 \leq j \leq m} \frac{c_j\gamma_{0,j}}{1+\gamma_{0,j}} \sum_{p \leq t \atop p \in E_j} \frac{|g(p)|}{p} < \frac{1}{4}\sum_{1 \leq j \leq m} \sum_{p \leq t \atop p \in E_j} \frac{|g(p)|}{p} \leq \frac{1}{2} \sum_{Z < p \leq t} \frac{|g(p)|}{p} = \lambda \log_2 t,
\end{equation*}
a contradiction. Since the sum on the interval $[Z,t]$ decreases monotonically as $\alpha$ increases, $\alpha > \frac{1}{2}$, as required. This completes the proof of the claim in this case.\\ 
If $(g,\mbf{E})$ is not non-decreasing for some $j$, bound $|g(p)|$ from below trivially by the function $g_0$ defined on primes such that $g_0(p) := \delta_j$ for each $p \in E_j$ and each $j$. Then applying the foregoing analysis with $|g(p)|$ replaced by $g_0(p)$, we can replace $\lambda$ by
\begin{equation*}
\sum_{1 \leq j \leq m} \frac{\gamma_{0,j}}{1+\gamma_{0,j}}\delta_jE_j(t) +O_B(1) \geq \sum_{1 \leq j \leq m} \frac{\delta_j}{B_j}\frac{\gamma_{0,j}}{1+\gamma_{0,j}} \sum_{p \leq t \atop p \in E_j} \frac{|g(p)|}{p} + O_B(1)
\end{equation*}
(see Remark \ref{REMWITHR}). This completes the proof.
\end{proof}
\begin{lem}\label{EXPONENTS}
Let $\lambda = \lambda(t) > 0$ be defined as in Lemma \ref{WITHR}. \\
i) If $(g,\mbf{E})$ is reasonable then
\begin{equation*}
\lambda = \begin{cases} \frac{1}{\log_2 t} \sum_{1 \leq j \leq m} \min\left\{\frac{1}{4m^2}, \frac{\gamma_{0,j}}{1+\gamma_{0,j}}\right\}\sum_{p \leq t \atop p \in E_j} \frac{|g(p)|}{p} &\text{ if $(g,\mbf{E})$ is non-decreasing} \\ 
\frac{1}{\log_2 t} \left(\min_{1 \leq j \leq m}\frac{\delta_j}{B_j}\right)\sum_{1 \leq j \leq m} \min\left\{\frac{1}{4m^2},\frac{\gamma_{0,j}}{1+\gamma_{0,j}}\right\}\sum_{p \leq t \atop p \in E_j} \frac{|g(p)|}{p} & \text{ otherwise.}\end{cases}
\end{equation*}
ii) If $(g,\mbf{E})$ is not reasonable then
\begin{equation*}
\lambda = \begin{cases} \frac{1}{\log_2 t} \frac{\delta}{B}\sum_{1 \leq j \leq m} \frac{\gamma_{0,j}}{1+\gamma_{0,j}} \sum_{p \leq t \atop p \in E_j} \frac{|g(p)|}{p} &\text{ if $(g,\mbf{E})$ is non-decreasing} \\ 
\frac{\delta}{\log_2 t}\sum_{1 \leq j \leq m} \frac{\gamma_{0,j}}{1+\gamma_{0,j}}E_j(t)& \text{ otherwise.}\end{cases}
\end{equation*}
In each of these cases, for any constant $C > 0$ we may choose $C' > 0$ a constant depending at most on $B$ and $C$ such that if $\kappa =\kappa(C) := \left(\frac{\delta}{2C'B}\right)^{\frac{4}{\delta}}$ then 
\begin{equation*}
C \frac{B}{\delta}\kappa^{-\lambda}e^{-\frac{1}{2}S_{\kappa}(t)} \leq \frac{1}{2}.
\end{equation*}
\end{lem}
\begin{rem}\label{REMEXPONENTS}
Note that in the case where $\mbf{E}$ is good, $(g,\mbf{E})$ is necessarily reasonable. We prove the result in a greater general in part to motivate the utility of the ''good'' condition. \\
Choosing $\delta_j \gg \log_2^{-1+\e} t$ for $\e > 0$ fixed, $1/\kappa$ is at most $e^{\log_2^{1-\e} t}$ (in the worst case where $B$ is much larger than $\delta_j$), so our assumption that the inequality $\log(1/\kappa) < \log_2 t$ holds easily here.
\end{rem}
\begin{proof}
Suppose first that $(g,\mbf{E})$ is reasonable. Thus, choose $1 \leq j_0 \leq m$ such that 
\begin{align*}
E_{j_0}(t)-E_{j_0}\left(t^{\kappa}\right) &\geq \frac{1}{m}\log(1/\kappa), \\
\sum_{p \leq t \atop p \in E_{j_0}} \frac{|g(p)|}{p} &\geq \frac{1}{m}\sum_{p \leq t} \frac{|g(p)|}{p}.
\end{align*}
If $(g,\mbf{E})$ is non-decreasing then clearly
\begin{align*}
&E_j(t)\sum_{t^{\kappa} < p \leq t \atop p \in E_j} \frac{|g(p)|}{p} - (E_j(t)-E_j(t^{\kappa})) \sum_{p \leq t \atop p \in E_j} \frac{|g(p)|}{p} = \sum_{t^{\kappa} < p \leq t \atop p \in E_j} \sum_{q \leq t \atop q \in E_j} \frac{1}{pq}\left(|g(p)|-|g(q)|\right) \\
&= \sum_{t^{\kappa} < p,q \leq t \atop p,q \in E_j} \frac{1}{pq}\left(|g(p)|-|g(q)|\right) + \sum_{t^{\kappa} < p \leq t \atop p \in E_j} \sum_{q \leq t^{\kappa}\atop q \in E_j } \frac{1}{pq}\left(|g(p)|-|g(q)|\right).
\end{align*}
By symmetry, the first sum is 0, while by monotonicity, the second sum is non-negative.  Hence, it follows that
\begin{equation}
\sum_{t^{\kappa} < p \leq t \atop p \in E_j} \frac{|g(p)|}{p} \geq \frac{E_j(t)-E_j(t^{\kappa})}{E_j(t)} \sum_{p \leq t \atop p \in E_j} \frac{|g(p)|}{p}. \label{MONO1}
\end{equation}
Applying this for each $j$ and using the ''reasonable'' property, we have
\begin{align}
\sum_{t^{\kappa} < p \leq t} \frac{|g(p)|}{p} &\geq \sum_{t^{\kappa} < p \leq t \atop p \in E_{j_0}} \frac{|g(p)|}{p} \geq \frac{E_{j_0}(t)-E_{j_0}(t^{\kappa})}{E_{j_0}(t)} \sum_{p \leq t \atop p \in E_{j_0}} \frac{|g(p)|}{p} \geq \frac{1}{m} \frac{\log(1/\kappa)}{E_{j_0}(t)}\sum_{p \leq t \atop p \in E_{j_0}} \frac{|g(p)|}{p} +o(1) \nonumber\\
&\geq \frac{1}{m^2} \frac{\log(1/\kappa)}{\log_2 t} \sum_{p \leq t} \frac{|g(p)|}{p} + O_B(1), \label{MONO2}
\end{align}
upon using the trivial bound $E_j(t) \leq \log_2 t + O(1)$. Thus, setting $c_j :=\min\left\{\frac{1}{4m^2},\frac{\gamma_{0,j}}{1+\gamma_{0,j}}\right\}$ in the definition of $\lambda$, we have
\begin{equation}
\kappa^{-\lambda} e^{-\frac{1}{2}S_{\kappa}(t)} = \exp\left(-\frac{1}{2}\sum_{t^{\kappa} < p \leq t} \frac{|g(p)|}{p} + \frac{\log(1/\kappa)}{\log_2 t} \sum_{1 \leq j \leq m} c_j\sum_{p \leq t \atop p \in E_j}\frac{|g(p)|}{p}\right) \ll_B e^{-\frac{1}{4}S_{\kappa}(t)}. \label{KAPDEF}
\end{equation}
If $(g,\mbf{E})$ is not non-decreasing then we instead have (using our observation regarding $S$ from Remark \ref{REMWITHR})
\begin{align*}
\sum_{t^{\kappa} < p \leq t} \frac{|g(p)|}{p} &\geq \sum_{1 \leq j \leq m} \delta_j((E_j\bk S)(t)-(E_j\bk S)(t^{\kappa})) \geq \frac{1}{m}\delta_{j_0}\log(1/\kappa) +o(\delta_{j_0})\\
&\geq \frac{1}{m}\frac{\delta_{j_0}}{B_{j_0}} \frac{\log(1/\kappa)}{E_{j_0}(t)}\sum_{p \leq t \atop p \in E_{j_0}} \frac{|g(p)|}{p} + o(\delta_{j_0}) \geq \left(\min_{1 \leq j \leq m} \frac{\delta_j}{B_j}\right)\frac{1}{m^2\log_2 t}\sum_{p \leq t \atop p \in E_j} \frac{|g(p)|}{p} + O_B(1),
\end{align*}
and \eqref{KAPDEF} holds again with $c_j := \frac{\delta}{B}\min\left\{\frac{1}{4m^2},\frac{\gamma_{0,j}}{1+\gamma_{0,j}}\right\}$ in this case.  \\
Suppose now that $(g,\mbf{E})$ is not reasonable. Again, if $(g,\mbf{E})$ is non-decreasing then \eqref{MONO1} still applies for each $j$, and instead of \eqref{MONO2} we simply note that if $c = \frac{\delta}{B}$ in the definition of $\lambda$ then we have
\begin{equation*}
\sum_{t^{\kappa} < p \leq t} \frac{|g(p)|}{p} \geq \delta \log\left(1/\kappa\right) +o(\delta)= Bc\log\left(1/\kappa\right) +o(\delta)\geq \frac{c\log(1/\kappa)}{\log_2 t}\sum_{p \leq t} \frac{|g(p)|}{p} +O_B(1) \geq 2\lambda \log\left(1/\kappa\right),
\end{equation*}
the last inequality following because $\frac{\gamma_{0,j}}{1+\gamma_{0,j}} \leq \frac{1}{2}$ for each $j$.  With this value of $c$, we may take $\kappa$ as before. \\
Finally, suppose $(g,\mbf{E})$ is not reasonable or non-decreasing.  Then we simply bound $|g(p)|$ in the definition of $\lambda$ from below by $\delta$. \\
In each of these scenarios we may select $\kappa$ such that, with $C' > 0$ a constant larger than $C$ depending only on $B$, we have 
\begin{equation*}
C'\frac{B}{\delta}e^{-\frac{\delta}{4}\log(1/\kappa)}= \frac{1}{2},
\end{equation*}
i.e., $\kappa = \left(\frac{\delta}{2C' B}\right)^{\frac{4}{\delta}}$. This completes the proof. 
\end{proof}
The above estimates culminate in the following.
\begin{prop} \label{STARTER}
For $t$ sufficiently large there exists a $C = C(B) > 0$, a function $\lambda: [1,x] \ra (0,\infty)$ and $\kappa = \kappa(C)$ as in Lemma \ref{EXPONENTS} such that
\begin{equation*}
\frac{|M_h(t)|}{M_{|g|}(t)} - \frac{R_h(\lambda)}{2\log^{\lambda} t} \leq C_1\left(\frac{B}{\delta}\frac{P_t}{\phi(P_t)}\left(\frac{1}{(1-\kappa) \log^{\frac{1}{2} }t}\left(\mc{G}(\sg)^{-2}J_{H,1}(\sg)\right)^{\frac{1}{2}} + |A|\mu\right)+ R_h(\lambda)\frac{B\log_2 t}{\kappa \log^{1+\lambda} t} +\frac{1}{\log^2 t}\right)
\end{equation*}
\end{prop}
\begin{proof}
Recall that $\mu := \max_p |\theta_p|$. Take $\lambda$ as in Lemma \ref{WITHR}. Applying Lemmas \ref{NUM} and \ref{DENOM} combined with Lemma \ref{CHEAP}, 
\begin{align*}
\frac{|M_h(t)|}{M_{|g|}(t)} &\ll_B \frac{B}{\delta}\frac{P_t}{\phi(P_t)}\left(\frac{\mc{G}(\sg)^{-1}}{\log t}\int_{t^{\kappa}}^{t} \frac{|M_h(u)|\log u}{u^{2}} du+R_h(\lambda)\frac{\int_1^{t^{\kappa}} \frac{M_{|g|}(u)}{u^2\log^{\lambda}(3u)} du}{\int_{1}^t \frac{M_{|g|}(u)}{u^2}du}+ A\mu\right)\\
&+ R_h(\lambda)\frac{B\log_2 t}{\kappa \log^{1+\lambda} t} + \frac{1}{\log^2 t}.
\end{align*}
Now, using the definition of $\Delta_h(u)$, which was introduced in Section 4, we have
\begin{align*}
&\int_{t^{\kappa}}^{t} \frac{|M_h(u)|\log u}{u^{2}} du 
\leq \int_1^{t} \frac{|N_h(u)|}{u^{2}} du+ \int_{t^{\kappa}}^{t}\frac{|\Delta_{h}(u)|}{u^{2}} du,
\end{align*}
whence we have the estimate
\begin{equation*}
\frac{|M_h(t)|}{M_{|g|}(t)} \ll_B \frac{B}{\delta}\frac{P_t}{\phi(P_t)}\left(\frac{1}{(1-\kappa)\log t}\mc{G}(\sg)^{-1}(M_1 + M_2) +M_3 + A\mu\right) + R_h(\lambda)\frac{B\log_2 t}{\kappa \log^{1+\lambda} t}
\end{equation*}
where we have defined $M_1 := \int_1^t \frac{|N_h(u)|}{u^2}du$ and
\begin{align*}
M_2 &:= \left(\int_{1}^{t} \frac{M_{|g|}(u)}{u^{2}}du\right)^{-1}\int_{t^{\kappa}}^t \frac{|\Delta_h(u)|}{u^2} du, \\
M_3 &:= R_h(\lambda)\left(\int_1^t \frac{M_{|g|}(u)}{u^2}du\right)^{-1}\left(\int_1^{t^{\kappa}} \frac{M_{|g|}(u)}{u^2\log^{\lambda}(3u)} du\right).
\end{align*}
By Lemma \ref{WITHR}, $M_3 \ll R_h(\lambda)\kappa^{-\lambda}e^{-\frac{1}{2}S_{\kappa}(t)}\log^{-\lambda} t$.
Next, using the pointwise bound for $\Delta_h(u)$, \eqref{GAPMN}, we have
\begin{align*}
M_2\left(\int_{1}^{t} \frac{M_{|g|}(u)}{u^{2}}du\right) &\ll BR_h(\lambda)\int_{t^{\kappa}}^{t} \frac{M_{|g|}(u)\log_2 u}{u^{2}\log^{1+\lambda}u}du 
\leq BR_h(\lambda)\frac{\log_2 t}{\kappa \log^{1+\lambda} t}\left(\int_{1}^{t} \frac{M_{|g|}(u)}{u^{2}}du\right). 
\end{align*}
Therefore, $M_2 \ll BR_h(\lambda)\frac{\log_2 t}{\kappa \log^{1+\lambda} t}$.  
Consider now $M_1$. Note that $t^{\sg-1} = e$. Applying Cauchy-Schwarz and then invoking Parseval's theorem, we get
\begin{align*}
M_1 &\leq \left(\int_1^{t} \frac{du}{u}\right)^{\frac{1}{2}}\left(\int_1^{t} \frac{|N_h(u)|^2}{u^{3}}\right)^{\frac{1}{2}} \leq \log^{\frac{1}{2}} t\left(e^2\int_1^{\infty} \frac{|N_{h}(u)|^2}{u^{1+2\sg}} du\right)^{\frac{1}{2}} = e\log^{\frac{1}{2}} t \left(\int_0^{\infty} |N_{h}(e^v)|^2e^{-2v\sg} dv\right)^{\frac{1}{2}} \\
&= 2\pi e \log^{\frac{1}{2}} t \left(\int_{(\sg)} |H'(s)|^2\frac{|ds|}{|s|^2}\right)^{\frac{1}{2}}  \ll J_{H,1}(\sg)^{\frac{1}{2}}\log^{\frac{1}{2}} t.
\end{align*}
Combining these estimates, we can find some constant $C_1$ depending at most on $B$ such that
\begin{equation*}
\frac{|M_h(t)|}{M_{|g|}(t)} \leq C_1\left(\frac{B}{\delta}\frac{P_t}{\phi(P_t)}\left(\frac{\left(\mc{G}(\sg)^{-2}J_{H,1}(\sg)\right)^{\frac{1}{2}}}{(1-\kappa)\log^{\frac{1}{2}} t} + |A|\mu\right)+ R_h(\lambda)\left(\kappa^{-\lambda}e^{-\frac{1}{2}S_{\kappa}(t)}\log^{-\lambda} t + B\frac{\log_2 t}{\kappa \log^{1+\lambda} t}\right) + \frac{1}{\log^2 t}\right).
\end{equation*}
We now choose $\kappa > 0$ as in Lemma \ref{EXPONENTS} for $C = C_1$.
In sum, we have
\begin{equation*}
\frac{|M_h(t)|}{M_{|g|}(t)} - \frac{1}{2}\frac{R_h(\lambda)}{\log^{\lambda} t} \leq C_1\left(\frac{B}{\delta}\frac{P_t}{\phi(P_t)}\left(\frac{1}{\log^{\frac{1}{2}} t}\left(\mc{G}(\sg)^{-2}J_{H,1}(\sg)\right)^{\frac{1}{2}} + |A|\mu\right) + R_h(\lambda)\frac{\log_2 t}{\kappa \log^{1+\lambda} t}  + \frac{1}{\log^2 t}\right).
\end{equation*}
This gives the claim.
\end{proof}   
An immediate corollary of the above proposition is an upper bound for $R_h(\lambda)$ upon taking maxima.
\begin{cor} \label{CORSTARTER}
For each $2 \leq t \leq x$ let $\sg_t := 1+\frac{1}{\log t}$. Suppose $C > 0$ is sufficiently small, $\kappa = \kappa(C)$ and $\lambda$ are as in Lemma \ref{WITHR} and Lemma \ref{EXPONENTS}. 
Then there exists a constant $C = C(\lambda,\delta) > 0$ and a $t_0 \geq 2$ sufficiently large such that
\begin{equation*}
R_h(\lambda) \leq \max\left\{1+|A|,C\max_{t_0 \leq t \leq x} \left\{\frac{B}{\delta}\frac{P_t}{\phi(P_t)}\log^{\lambda} t\left(\frac{1}{\log^{\frac{1}{2}} t}\left(\mc{G}(\sg)^{-2}J_{H,1}(\sg_t)\right)^{\frac{1}{2}}+ |A|\mu\right) + \frac{1}{\log^2 t}\right\}\right\}.
\end{equation*}
\end{cor}
\begin{proof}
In Proposition \ref{STARTER} we must choose $t_0$ sufficiently large such that each $t \geq t_0$ satisfies its conclusion.  For the remaining elements we can only get $O(1)$ order terms, and since, by the triangle inequality we have $\frac{|M_h(t)|}{M_{|g|}(t)} \leq 1+|A|$, the first bound also follows. \\
For $t \geq t_0$, otherwise, the claim follows immediately upon multiplying each side of the estimate in Proposition \ref{STARTER} by $\log^{\lambda} t$, taking the maximum of both sides and then coalescing all of the terms in $R_h(\lambda)$ together on the left side. This gives a fixed positive multiple of $R$ when $C$ is sufficiently small.
\end{proof}
\begin{proof}[Proof of Theorem \ref{HalGen} i)]
Assume first that $A = 0$.  Then $H = G$, and by Lemmas \ref{INTBOUND} and \ref{MAIN1}, we have 
\begin{align*}
\mc{G}(\sg)^{-2}J_{G,1}(\sg_t) &\ll_{m,B,\mbf{\phi},\mbf{\beta}} B^2(\sg_t-1)^{-1}\prod_{1 \leq j \leq m} \left|\frac{G_j(\sg_t)}{\mc{G}_j(\sg_t)}\right|^{\frac{\gamma_{0,j}}{1+\gamma_{0,j}}} +1 \\
&\ll_B B^2(\sg_t-1)^{-1}\exp\left(-\sum_{1 \leq j \leq m} \frac{\gamma_{0,j}}{1+\gamma_{0,j}} \sum_{p\in E_j} \frac{|g(p)|-\text{Re}(g(p))}{p^{\sg_t}}\right) + 1.
\end{align*}
By Corollary \ref{CORSTARTER}, we need only show that if $\sg_t := 1+\frac{1}{\log t}$ then 
\begin{equation} \label{INCR}
h(t) := \frac{B^2}{\delta}\frac{P_t}{\phi(P_t)}\prod_{1 \leq j \leq m} \left|\frac{G_j(\sg_t)}{\mc{G}_j(\sg_t)}\right|^{c_j'\frac{\gamma_{0,j}}{2(1+\gamma_{0,j})}}\log^{\lambda} t
\end{equation}
where $c_j' >0$, is maximized when $t = x$, after which the theorem will immediately follow. Indeed, if the function is maximized at $t = x$ then we have
\begin{align*}
\frac{|M_g(x)|}{M_{|g|}(x)}\log^{\lambda(x)} x = R_g(\lambda) \ll_B \frac{B^2}{\delta}\frac{P_x}{\phi(P_x)}\exp\left(-\sum_{1 \leq j \leq m} \frac{c_j'\gamma_{0,j}}{1+\gamma_{0,j}} \sum_{p\in E_j} \frac{|g(p)|-\text{Re}(g(p))}{p^{\sg_x}}\right)\log^{\lambda(x)} x,
\end{align*}
in case $g \in \mc{C}_b$, and a similar argument holds when $g \in \mc{C}\bk \mc{C}_b$.
Now, \eqref{TAIL} implies that $\sum_{p > x} |g(p)|\left(\frac{1}{p^{\sg_x}} + \frac{1}{p^{\sg'}}\right) \ll_B 1$, and 
\begin{align*}
\exp\left(\sum_{p \leq x} \left(\frac{1}{p^{\sg}}-\frac{1}{p}\right)\right) &= \exp\left(\sum_{p \leq x} \frac{1}{p}\left(e^{-(\sg-1)\log p} - 1\right)\right) \leq \exp\left(-(\sg-1)\sum_{p \leq x} \frac{\log p}{p}\right) \ll 1.
\end{align*}
Thus,
\begin{equation} \label{MAX}
\frac{|M_g(x)|}{M_{|g|}(x)}\ll_B \frac{B^2}{\delta}\frac{P_x}{\phi(P_x)}\exp\left(-\sum_{1 \leq j \leq m} c_j'\frac{\gamma_{0,j}}{2(1+\gamma_{0,j})} \sum_{p\leq x \atop p \in E_j} \frac{|g(p)|-\text{Re}(g(p))}{p}\right).
\end{equation}
Now, $\frac{B^2}{\delta}$ is non-decreasing in $t$, and thus to show that $t = x$ maximizes $h(t)$, it suffices to show that the remaining factors defining it are maximal at $t = x$. Of course, the maximal value of $\frac{P_t}{\phi(P_t)}$ is $O\left(\log_2 t\right)$ (see, for instance, Theorem 4 in I.5 of \cite{Ten2}). We consider two cases, according to whether $g \in \mc{C}_b$ or not. Suppose first that $g \in \mc{C}_b$. Since $\log^{\lambda(t)} t = \exp\left(\sum_{1 \leq j \leq m} \frac{c_j\gamma_{0,j}}{1+\gamma_{0,j}}\sum_{p \leq t \atop p \in E_j} \frac{|g(p)|}{p}\right)$, where $c_j > 0$, we see that it suffices that $c_j' \leq c_j$, for in this case,
\begin{align*}
&\log^{\lambda(t)}t\exp\left(-\sum_{1 \leq j \leq m} \frac{c_j'\gamma_{0,j}}{1+\gamma_{0,j}} \sum_{p \leq t \atop p \in E_j} \frac{g(p)-\text{Re}(g(p))}{p}\right) \\
&\geq \exp\left(\sum_{1 \leq j \leq m}\frac{\gamma_{0,j}}{2(1+\gamma_{0,j})}\left((2c_j-c_j') \sum_{p \leq t \atop p \in E_j} \frac{g(p)}{p} +\sum_{p \leq t \atop p \in E_j} \frac{|\text{Re}(g(p))|}{p}\right)\right)\\
&\geq \exp\left(\sum_{1 \leq j \leq m}c_j'\frac{\gamma_{0,j}}{2(1+\gamma_{0,j})}\sum_{p \leq t \atop p \in E_j} \frac{|g(p)|+\text{Re}(g(p))}{p}\right)
\end{align*}
uniformly in $t_0 \leq t \leq x$, and the expression in the exponential is non-decreasing in $t$ since $|g(p)|+\text{Re}(g(p)) \geq 0$ for all $p$, provided we replace $c_j$ by $\frac{1}{2}c_j$, if necessary, to account for the $\log_2 t$ factor from $\frac{P_t}{\phi(P_t)}$. Thus, depending on the choice of $c_j$ emerging from Lemma \ref{EXPONENTS}, the right side of \eqref{MAX} with $t$ in place of $x$ is largest when $t = x$ when $c_j' := \frac{1}{2}c_j$, and the proof of the theorem is complete in this case. \\
On the other hand, if $g \in \mc{C} \bk \mc{C}_b$ by Lemma \ref{EASYHalDecay} we have
\begin{equation*}
\max_{|\tau| \leq (\sg-1)^{-D}}\left|\frac{G(s)}{\mc{G}(\sg)}\right| \ll_B \exp\left(-\sum_{1 \leq j \leq m} B_j\left(\rho_{E_j}(t;\tilde{g},T) + \sum_{p \in E_j \atop p \leq t} \frac{1-|\tilde{g}(p)|}{p}\right)\right).
\end{equation*}
Inserting this instead into the estimate from Lemma \ref{MAIN1} and then applying Corollary \ref{CORSTARTER} as above, we need to show that
\begin{align*}
&\lambda(t)\log_2 t +B_j\left(\sum_{p \in E_j \atop p \leq t} \frac{1}{p} - \rho_{E_j}(t;\tilde{g},T)\right) - \sum_{p \in E_j \atop p \leq t} \frac{|g(p)|}{p} \\
&= \sum_{p \in E_j \atop p \leq t} (c_j-1)\sum_{p \in E_j \atop p \leq t} \frac{|g(p)|}{p} + B_j\left(\sum_{p \in E_j \atop p \leq t} \frac{1}{p} - \rho_{E_j}(t;\tilde{g},T)\right)
\end{align*}
is increasing in $t$. Of course, $2\sum_{p \in E_j \atop p \leq t}\frac{1}{p}-\rho_{E_j}(t;\tilde{g},T)$ is non-decreasing in $t$ by definition, and thus, picking $c_j = B_j+1$ proves the claim in this case as well.
\end{proof}
\begin{proof}[Proof of Theorem \ref{HalGen} ii)]
In this case, $\mu = \max_p |\theta_p|  < \eta$.  Applying Corollary \ref{CORSTARTER} with 
\begin{equation*}
A := \exp\left(-\sum_{p \leq t} \frac{|g(p)|-g(p)}{p}\right) 
\end{equation*}
instead, in conjunction with Lemma \ref{SHARPAPP} gives, in the non-trivial case,
\begin{equation*}
R_h(\lambda) \ll_B \max_{t_0 \leq t \leq x} \left\{\frac{B}{\delta}\log^{\lambda} t\left(\frac{1}{\log^{\frac{1}{2}} t}\left(\mc{G}(\sg)^{-2}J_{H,1}(\sg_t)\right)^{\frac{1}{2}}+ A\eta\right) + \frac{1}{\log^2 t}\right\},
\end{equation*}
where, for $T = (\sg-1)^D$, $D > 2$ and $K > 0$ satisfying $K(\sg-1) < 1$ and $B\eta \log(1+K) < 1$,
\begin{align*}
\mc{G}(\sg_t)^{-2}J_{H,1}(\sg_t) &\leq B^2(\sg_t-1)^{-1}(\eta^2|A|^2\min\{\delta^{-2},\log^2(1+K)(1+K)^{-2\delta}\} + (\sg_t-1)^{-2}T^{-1} \\
&+ (\sg_t-1)^{\gamma}|A|^{\frac{\gamma_0}{2(1+\gamma_0)}} + |A|^2(\sg_t-1)^{2\delta/3} + \frac{1}{\gamma} (1+K)^{-(1+\gamma)} + \frac{|A|^2}{\delta} (1+K)^{-(1+\delta)}).
\end{align*}
If $B\eta \leq \delta < \eta^{\frac{1}{2}}$, take $K := e^{\frac{1}{\delta}}-1$.  Otherwise, take $K := e^{\frac{1}{\sqrt{\eta}}}-1$.  Set $d_1 := \delta^{-1}\sqrt{\eta}$ in the first case and $d_1 := 1$ in the second case. 
Then
\begin{align*}
\mc{G}(\sg_t)^{-2}J_{H,1}(\sg_t) &\ll_{B,\mbf{\phi},\mbf{\beta}} \frac{B^2}{\sg_t-1}\left(d_1^2\eta|A|^2+ (\sg_t-1)^{2\delta/3} + \delta^{-1}e^{-\frac{2d_1}{\sqrt{\eta}}} + |A|^{\frac{\gamma_0}{2(1+\gamma_0)}}\left((\sg_t-1)^{\gamma} + \frac{e^{-\frac{2d_1}{\sqrt{\eta}}}}{\gamma}\right)\right).
\end{align*}
It therefore follows, upon taking termwise square-roots of the above expression in the non-trivial bound in Corollary \ref{CORSTARTER} and using the definition of $\sg = \sg_t$,
\begin{align*}
R_h(\lambda) &\ll_{B,m,\mbf{\phi},\mbf{\beta}} \frac{B^2}{\delta}\log^{\lambda} t \left(d_1\eta^{\frac{1}{2}}|A| + \log^{-\frac{\delta}{3}} t + \delta^{-\frac{1}{2}}e^{-\frac{d_1}{\sqrt{\eta}}} + |A|^{\frac{\gamma_0}{4(1+\gamma_0)}}\left(\log^{-\frac{\gamma}{2}}t + e^{-\frac{d_1}{\sqrt{\eta}}}\right) +\log^{-2} t + |A|\eta\right).
\end{align*}
Now, we saw in Lemmas \ref{HalDecay} and \ref{WITHR} that $\gamma_{0,j} = C_j\beta_j^3$, where $C_j:= \frac{27\delta_j}{1024\pi B_j}$. Since $|\theta_p| \leq \eta_j$ for each $p \in E_j$ by assumption, we can take $\phi_j = \pi$ so that $\beta_j = \pi-\eta_j > 1$ is admissible. If it happens that $\gamma_{0,j} \geq 1$ then $\frac{\gamma_{0,j}}{2(1+\gamma_{0,j})} \geq \frac{1}{4}$. On the other hand, if $\gamma_{0,j} < 1$, we have $\frac{\gamma_{0,j}}{2(1+\gamma_{0,j})} \geq \frac{C_j}{8}\beta_j^3 > \frac{C_j}{8}$. Thus, provided that $\eta_j < \frac{C_j}{16}c_j$ for each $j$, we have $\frac{\gamma_{0,j}c_j}{2(1+\gamma_{0,j})} > \alpha\eta_j$, for any $\alpha \in (0,1]$. It therefore follows that for such $\alpha$,
\begin{align*}
\log^{\lambda} t|A|^{\alpha} &\geq \exp\left(\sum_{1 \leq j\leq m} \sum_{p \leq t \atop p \in E_j} \frac{|g(p)|}{p}\left(\frac{c_j\gamma_{0,j}}{2(1+\gamma_{0,j})} - \alpha\left|1-e^{i\theta_p}\right|\right)\right) \\
&\geq \exp\left(\sum_{1 \leq j \leq m} \sum_{p \leq t \atop p \in E_j} \frac{|g(p)|}{p}\left(\frac{c_j\gamma_{0,j}}{2(1+\gamma_{0,j})} - \alpha\eta_j\right)\right) \\
&\geq \exp\left(\sum_{1 \leq j \leq m} \frac{c_j\gamma_{0,j}}{4(1+\gamma_{0,j})}\sum_{p \leq t \atop p \in E_j} \frac{|g(p)|}{p}\right),
\end{align*}
which is increasing in $t$. The other terms depending on $t$ contribute strictly smaller contributions. Hence, for this choice of $\lambda$, the upper bound for $R_h(\lambda)$ is maximized at $t = x$. It follows that, in the non-trivial case,
\begin{align*}
\frac{|M_h(x)|}{M_{|g|}(x)}\log^{\lambda(x)}x = R_h(\lambda) &\ll_{B,m,\mbf{\phi},\mbf{\beta}} \frac{B^2}{\delta}\log^{\lambda(x)} x ((\eta^{\frac{1}{2}}|A|\left(d_1 + \log^{-\frac{\delta}{3}} x + \delta^{-\frac{1}{2}}e^{-\frac{d_1}{\sqrt{\eta}}}\right) \\
&+ |A|^{\frac{\gamma_0}{4(1+\gamma_0)}}\left(\log^{-\gamma/2}x + e^{-\frac{d_1}{\sqrt{\eta}}}\right) +\log^{-2} t) + |A|\eta).
\end{align*}
The last term being irrelevant, since $\eta < 1$ and $\frac{\gamma_0}{4(1+\gamma_0)} < 1$ we have
\begin{align*}
|M_h(x)| &\ll_{B,m,\mbf{\phi},\mbf{\beta}} \frac{B^2}{\delta} M_{|g|}(x)\left(\eta^{\frac{1}{2}}|A|\left(d_1 +\log^{-\frac{\delta}{3}} x + \delta^{-\frac{1}{2}}e^{-\frac{d_1}{\sqrt{\eta}}}\right)+ |A|^{\frac{\gamma_0}{4(1+\gamma_0)}}\left(\log^{-\frac{\gamma}{2}} x + \gamma^{-1}e^{-\frac{d_1}{\sqrt{\eta}}}\right)\right).
\end{align*}
The upper bound just given is precisely the error term in ii) of Theorem \ref{HalGen}.
\end{proof}
\section{Proof of Theorem \ref{WIRSINGEXT}}
In this section, we apply Theorem \ref{HalGen} to prove Theorem \ref{WIRSINGEXT}. We first need an estimate for the ratio $\frac{M_{|g|}(x)}{M_{f}(x)}$, which we can directly determine via Theorem \ref{LOWERMV}.
\begin{lem}\label{BASICLOW}
Let $\delta, B > 0$, and let $S$ be a set of primes for which $P_x := \prod_{p \in S \atop p \leq x} p \ll_r x^r$ for $r > 0$. Suppose $g$ is a strongly multiplicative function satisfying $|g(p)| \geq \delta$ for each prime $p \notin S$, and $f$ is a non-negative multiplicative function satisfying $|g(n)| \leq f(n)$ for all $n$ and $f(p) \leq B$. Then for any sufficiently large $x$,
\begin{equation*}
M_{|g|}(x) \ll_{B} \frac{B}{\delta} \frac{P_x}{\phi(P_x)}\exp\left(-\sum_{p \leq x} \frac{f(p)-|g(p)|}{p}\right)M_{f}(x).
\end{equation*}
\end{lem}
\begin{proof}
By assumption, $\delta \leq |g(p)| \leq f(p) \leq B$, so by Theorem \ref{LOWERMV}, we have
\begin{equation}
M_{f}(x) \gg_B \delta \frac{\phi(P_x)}{P_x}\frac{x}{\log x} \prod_{p \leq x} \left(1+\frac{f(p)}{p}\right). \label{LAMBDAEST}
\end{equation}
On the other hand, by the argument of Lemma \ref{NUM} and 
by Lemma \ref{LOGSUM},
\begin{equation} \label{GESTFORRAT}
M_{|g|}(x) \ll_B B\frac{x}{\log x} \int_1^x \frac{M_{|g|}(u)}{u^2}du = B\frac{x}{\log x} \left(\sum_{n \leq x}\frac{|g(n)|}{n} - \frac{M_{|g|}(x)}{x}\right) \ll_B B\frac{x}{\log x} \prod_{p \leq x} \left(1+\frac{|g(p)|}{p}\right).
\end{equation}
%
Dividing \eqref{GESTFORRAT} by \eqref{LAMBDAEST}, we get
\begin{align*}
M_{|g|}(x) &\ll_{B} \frac{B}{\delta}\frac{P_x}{\phi(P_x)}M_{f}(x)\prod_{p \leq x} \left(1+\frac{|g(p)|}{p}\right)\left(1+\frac{f(p)}{p}\right)^{-1} = \frac{B}{\delta}M_{f}(x)\prod_{p \leq x} \left(1-\frac{f(p)-|g(p)|}{p+f(p)}\right) \\
&\ll \frac{B}{\delta}\frac{P_x}{\phi(P_x)}M_{f}(x)\exp\left(-\sum_{p \leq x} \frac{f(p)-|g(p)|}{p+f(p)}\right) \ll_B \frac{B}{\delta}\frac{P_x}{\phi(P_x)}\exp\left(-\sum_{p \leq x} \frac{f(p)-|g(p)|}{p}\right),
\end{align*}
as claimed.
\end{proof}
\begin{proof}[Proof of Theorem \ref{WIRSINGEXT}] 
Part i) of the theorem is immediate upon making the factorization $\frac{|M_g(x)|}{M_{f}(x)} = \frac{|M_g(x)|}{M_{|g|}(x)}\cdot \frac{M_{|g|}(x)}{M_{f}(x)}$, and applying Theorem \ref{HalGen} to the first factor and Lemma \ref{BASICLOW} to the second one. Note that the right side of the estimate in the statement must vanish as $x \ra \infty$. Indeed, if $\sum_{p \leq x} \frac{f(p)-|g(p)|}{p}$ diverges as $x \ra \infty$ then this is obvious; otherwise, $|f(p)-|g(p)|| < \frac{1}{2}|f(p)-\text{Re}(g(p)p^{i\tau})|$ for $p$ sufficiently large, by virtue of 
\begin{equation*}
\prod_{p \leq x} \left(1+\sum_{k \geq 1} \frac{g(p)p^{-ik\tau}}{p^k}\right)^{-1}\left(1+\sum_{k \geq 1} \frac{f(p^k)}{p^k}\right) \ll_B \exp\left(\sum_{p \leq x} \frac{f(p)-\text{Re}(g(p)p^{-i\tau})}{p}\right),
\end{equation*}
for any $|\tau| \leq T$, and the left side of this last estimate tends to infinity as $x$ gets large by assumption (a similar estimate holds for $g$ completely multiplicative, with $g(p)^k$ in place of $g(p)$).  Therefore, 
\begin{equation*}
||g(p)|-\text{Re}(g(p)p^{i\tau})| \geq |f(p)-\text{Re}(g(p))| - |f(p)-|g(p)|| \geq \frac{1}{2}|f(p)-\text{Re}(g(p)p^{i\tau})|,
\end{equation*}
whence it follows that $\sum_{p \leq x} \frac{|g(p)|-\text{Re}(g(p)p^{-i\tau})}{p}$ tends to infinity as $x \ra \infty$ uniformly in compact intervals of $\tau$. Consequently, for some $E_j$ in a given partition $\mbf{E}$, the corresponding sum in the exponential in \eqref{UPPER} tends to $\infty$, and the exponential itself tends to 0.\\
For part ii), we first make some observations related to Remarks \ref{REMGEN} and \ref{REMGEN2}, which suggest that the conclusions of Sections 4 and 5 all apply more generally. Indeed, set $h(n) := |g(n)|-X_tf(n)$, where $X_t := \exp\left(-\sum_{p \leq t} \frac{f(p)-|g(p)|}{p}\right)$, $f(n)$ is a strongly multiplicative function satisfying $|g(n)| \leq f(n)$, $\delta \leq |g(p)| \leq f(p) \leq B$, and $\left|f(p)-|g(p)|\right| \leq B\eta$. We indicated in the aforementioned remarks that the arguments in sections 4,5 and 6 still apply with these choices, as do the arguments in Section 6 (in which we must instead define $R_h(\lambda) := \max_{2 \leq t \leq x} \frac{M_{|g|}(t)}{M_{f}(t)}\log^{\lambda} t$).  With the appropriate translations (which we leave to the reader, as they are straightforward restatements of our above arguments), the statement of Corollary \ref{CORSTARTER} gives (with $X := X_x$)
\begin{equation*}
R_{|g|-Xf}(\lambda) \leq \max\left\{1+X,D\max_{t_0 \leq t \leq x} \frac{B}{\delta}\frac{P_t}{\phi(P_t)}\log^{\lambda} t \left(\log^{-\frac{1}{2}} t (F(\sg_t)^{-2} J_{H,1}(\sg_t))^{\frac{1}{2}} + X\mu\right) + \frac{1}{\log^2 t}\right\},
\end{equation*}
where $F(s)$ is the Dirichlet series of $f$ and $D$ a constant depending at most on $B$. Moreover, for $\sg_t = 1 + \frac{1}{\log t}$ as above,
\begin{equation*}
F(\sg_t)^{-2} J_{H,1} \ll_B B^2(\sg_t-1)^{-1}\left(X^2\left(c_1^2\eta + (\sg_t-1)^{\delta\beta^3} + \frac{e^{-\frac{2d_1}{\sqrt{\eta}}}}{\delta}\right)+(\sg_t-1)^{\frac{2\delta}{3}} + \delta^{-1}e^{-\frac{2d_1}{\sqrt{\eta}}}\right) +(\sg_t-1)^2.
\end{equation*}
It follows by the same argument as at the end of Section 6 that
\begin{equation} \label{MOREGEN}
M_{|g|}(x) = M_{f}(x)\left(X + O_{B,m}\left(\mc{R}_1\right)\right).
\end{equation}
where we have set
\begin{equation*}
\mc{R}_1 := \left(\frac{B}{\delta}\right)^2 \frac{P_t}{\phi(P_t)}\left(X\left(d_1\eta^{\frac{1}{2}} + \log^{-\frac{\delta\beta^3}{2}} x + \delta^{-1}e^{-\frac{d_1}{\sqrt{\eta}}}\right) + \log^{-\frac{2\delta}{3}}x + \delta^{-1}e^{-\frac{2d_1}{\sqrt{\eta}}}\right).
\end{equation*}
With these observations, consider the modified hypotheses that $f$ is a strongly multiplicative function satisfying $|g(n)| \leq f(n)$ for each positive integer $n$, $\delta \leq |g(p)| \leq f(p) \leq B$ and $|g(p)-f(p)| \leq B\eta$ (rather than $||g(p)|-f(p)| \leq B\eta$) uniformly over all primes $p$.  By the triangle inequality, we have $||g(p)|-f(p)| \leq \eta$, as well as 
\begin{equation*}
\left||g(p)|-g(p)\right| \leq |g(p)-f(p)| + ||g(p)|-f(p)| \leq 2\eta.
\end{equation*}
We may therefore apply Theorem \ref{HalGen} ii) with $g$, giving
\begin{equation*}
M_g(x) = M_{|g|}(x)\left(A + O_{B,m}\left(\mc{R}_2\right)\right),
\end{equation*}
where 
\begin{equation*}
\mc{R}_2 := \left(\frac{B}{\delta}\right)^2 \left(d_1\eta^{\frac{1}{2}}|A| + \log^{-\frac{2\delta}{3}}x + \delta^{-1}e^{-\frac{2d_1}{\sqrt{\eta}}} + |A|^{\frac{\gamma_0}{4(1+\gamma_0)}}\left(\log^{-\frac{\gamma}{2}} x + \gamma^{-1}e^{-\frac{d_1}{\sqrt{\eta}}}\right)\right).
\end{equation*}
Inputting the asymptotic in \eqref{MOREGEN}, we get
\begin{equation*}
M_g(x) = M_{f}(x)\left(AX + O_{B,m}(|A|\mc{R}_1+X\mc{R}_2)\right).
\end{equation*}
Now clearly, $AX = \exp\left(-\sum_{p \leq x} \frac{f(p)-g(p)}{p}\right)$, so the proof of the theorem is complete.
\end{proof}
\section*{Acknowledgements}
The author would like to thank his Ph.D supervisor Dr. J. Friedlander for his ample patience and encouragement during the period in which this paper was written, and Dr. D. Koukoulopoulos for helpful comments.

\bibliographystyle{plain}
\bibliography{bibWirsing}
\end{document}